\documentclass{amsart}
\usepackage{amsmath,amsthm,amsfonts,amssymb,color,graphicx,appendix,ulem,comment,hyperref,mathrsfs}

\numberwithin{equation}{section}
\theoremstyle{plain}
\newtheorem{theorem}{\sc Theorem}[section]

\newtheorem{condition}[theorem]{\sc Condition}

\newtheorem{definition}[theorem]{\sc Definition}
\newtheorem{lemma}[theorem]{\sc Lemma}
\newtheorem{proposition}[theorem]{\sc Proposition}

\theoremstyle{remark}
\newtheorem{remark}[theorem]{\sc Remark}
\newtheorem{example}[theorem]{\sc Example}

\newcommand{\one}{{{\rm 1\mkern-1.5mu}\!{\rm I}}}

\newcommand{\Ham}{\overline H}
\newcommand{\Hamg}{\overline H}

\newcommand{\kes}{\text{frac}}

\newcommand{\Con}{\operatorname{C}}
\newcommand{\Cone}{\operatorname{C^1}}
\newcommand{\UC}{\operatorname{UC}}
\newcommand{\BUC}{\operatorname{BUC}}
\newcommand{\Lip}{\operatorname{Lip}}

\begin{document}

\title[Stochastic homogenization and effective Hamiltonians: The double-well case]{Stochastic homogenization and effective Hamiltonians of\\ HJ equations in one space dimension: The double-well case}

\author[A.\ Yilmaz]{Atilla Yilmaz}
\address{Atilla Yilmaz, Department of Mathematics, Temple University, 1805 N.\ Broad Street, Philadelphia, PA 19122, USA}
\email{atilla.yilmaz@temple.edu}
\urladdr{http://math.temple.edu/$\sim$atilla/}

\date{\today}


\begin{abstract}
	
	We consider Hamilton-Jacobi equations in one space dimension with Hamiltonians of the form $H(p,x,\omega) = G(p) + \beta V(x,\omega)$, where $V(\cdot,\omega)$ is a stationary \& ergodic potential of unit amplitude.
	The homogenization of such equations is established in a 2016 paper of Armstrong, Tran and Yu for all continuous and coercive $G$.
	Under the extra condition that $G$ is a double-well function (i.e., it has precisely two local minima), we give a new and fully constructive proof of homogenization which yields a formula for the effective Hamiltonian $\overline H$.
	We use this formula to provide a complete list of the heights at which the graph of $\overline H$ has a flat piece.
	We illustrate our results by analyzing basic classes of examples, highlight some corollaries that clarify the dependence of $\Ham$ on $G$, $\beta$ and the law of $V(\cdot,\omega)$, and discuss a generalization to even-symmetric triple-well Hamiltonians.
	
\end{abstract}

\maketitle

\section{Introduction}\label{sec:intro}

Let $(\Omega,\mathcal{F},\mathbb{P})$ be a probability space equipped with a group of measure-preserving transformations $\tau_x:\Omega\to\Omega$, $x\in\mathbb{R}$, such that $\tau_0 = \text{id}$ and $\tau_x\circ \tau_y = \tau_{x+y}$ for every $x,y\in\mathbb{R}$. Assume that $\mathbb{P}$ is ergodic under this group, i.e.,
\[ \mathbb{P}(\cap_{x\in\mathbb{R}}\tau_xA) \in\{0,1\}\quad\text{for every}\ A\in\mathcal{F}. \]
We will use $\mathbb{E}$ to denote expectation with respect to $\mathbb{P}$.

Consider a Hamilton-Jacobi (HJ) equation of the form
\begin{equation}\label{eq:originalHJ}
\partial_tu^\epsilon(t,x,\omega) + G(\partial_xu^\epsilon(t,x,\omega)) + \beta V(x/\epsilon,\omega) = 0,\quad(t,x)\in(0,+\infty)\times\mathbb{R},
\end{equation}
with the following ingredients: (i) $\beta,\epsilon>0$; (ii) $\omega\in\Omega$;
(iii) the function $G:\mathbb{R}\to[0,+\infty)$ satisfies
\begin{equation}\label{eq:coercive}
	G\in\Con(\mathbb{R}) \quad \text{and} \quad \lim_{p\to \pm\infty}G(p) = +\infty,
\end{equation}
i.e., it is continuous and coercive on $\mathbb{R}$;
(iv) the so-called potential $V:\mathbb{R}\times\Omega\to\mathbb{R}$ is nonconstant and measurable,
\begin{equation}\label{eq:stationary}
	V(x,\omega) = V(0,\tau_x\omega)\quad\text{for every}\ (x,\omega)\in\mathbb{R}\times\Omega,
\end{equation}
and the map
\begin{equation}\label{eq:bucon}
	\text{$x\mapsto V(x,\omega)$ is in $\BUC(\mathbb{R})$ for every $\omega\in\Omega$,}
\end{equation}
i.e., it is bounded and uniformly continuous on $\mathbb{R}$.

The parameter $\beta$ adjusts the amplitude of the potential. Since adding a constant to $V$ corresponds to adding a linear (in time) term to any solution of \eqref{eq:originalHJ}, we will assume without further loss of generality that
\begin{equation}\label{eq:evren}
\mathbb{P}\left( \inf\{V(x,\omega):\,x\in\mathbb{R}\} = 0 < 1 = \sup\{V(x,\omega):\,x\in\mathbb{R}\} \right) = 1.
\end{equation}

The parameter $\epsilon$ adjusts the length scale of the potential. As $\epsilon\to 0$, the HJ equation in \eqref{eq:originalHJ} homogenizes to a deterministic HJ equation of the form
\begin{equation}\label{eq:effectiveHJ}
\partial_t\overline u(t,x) + \Ham(\partial_x\overline u(t,x)) = 0,\quad(t,x)\in(0,+\infty)\times\mathbb{R},
\end{equation}
with a coercive $\Ham\in\Con(\mathbb{R})$, referred to as the effective Hamiltonian.
Precisely, for every $g\in\UC(\mathbb{R})$, the space of uniformly continuous functions on $\mathbb{R}$, let $u_g^\epsilon(\cdot,\cdot,\omega)$ and $\overline u_g$ be the unique viscosity solutions of \eqref{eq:originalHJ} and \eqref{eq:effectiveHJ} that satisfy $u_g^\epsilon(0,\cdot,\omega) = \overline u_g(0,\cdot) = g$. 
With this notation, there exists an $\Omega_0\in\mathcal{F}$ with $\mathbb{P}(\Omega_0) = 1$ such that, for every $\omega\in\Omega_0$ and $g\in\UC(\mathbb{R})$, $u_g^\epsilon(\cdot,\cdot,\omega)$ converges locally uniformly on $[0,+\infty)\times\mathbb{R}$ as $\epsilon\to0$ to $\overline u_g$. This homogenization result (without any assumptions beyond \eqref{eq:coercive}--\eqref{eq:evren}) is due to Armstrong, Tran and Yu \cite{ATY16}.

In the course of their proof in \cite{ATY16}, Armstrong, Tran and Yu establish several key properties of the effective Hamiltonian. Most notably:
\begin{itemize}
	\item $\Ham$ is level-set convex (also called quasiconvex) if and only if the parameter $\beta$ is greater than or equal to an explicit threshold (which depends on the local extreme values of $G$);
	\item the graph $\{ (\theta,\Ham(\theta)):\,\theta\in\mathbb{R} \}$ has some flat pieces.
\end{itemize}
Their approach relies on showing that, for every $\theta$ outside the flat pieces, there exists a corrector that is sublinear in $x$ as $x\to\pm\infty$. 

The main goal of this paper is to provide a more detailed picture of the effective Hamiltonian. For the sake of clarity and convenience, we choose to restrict our attention to the case where $G$ has precisely two local minima, i.e., it is a double-well function. In this special but representative case, we give a new proof of the homogenization result in \cite{ATY16}. Moreover,
\begin{itemize}
	\item we derive a formula for $\Ham$ on the whole real line that is easy to analyze;
	\item we provide a complete list of the values of $\Ham$ where its graph has a flat piece (which depends on the parameter $\beta$, the local extreme values of $G$ and whether certain explicit events involving the oscillations of $V(\cdot,\omega)$ are null sets under $\mathbb{P}$).
\end{itemize}
The distinguishing feature of our approach is that, for every $\theta$ outside the flat pieces, we construct a sublinear corrector ``by hand", i.e., without using any existence results or limit procedures.

\bigskip

Here is an outline of the rest of the paper.

\vspace{-1mm}
	
\subsubsection*{Section \ref{sec:prevresults}} We start by surveying the literature on the homogenization of HJ equations in the periodic and the stationary \& ergodic settings with convex, level-set convex and nonconvex Hamiltonians. We say a few words about some of the existing proof strategies and introduce (sublinear) correctors in a nontechnical way. After that, we focus on the aforementioned work of Armstrong, Tran and Yu \cite{ATY16}, outline their proof, and present their findings regarding the effective Hamiltonian in order to put our results into context.

\vspace{-1mm}
	
\subsubsection*{Section \ref{sec:ourresults}} For a general class of double-well functions $G$ with local extreme values $0\le m < M$,
we state our homogenization result and identify three qualitatively distinct cases depending on $\beta$, $m$ and $M$ (see Theorem \ref{thm:velinim}).
In each of these cases, we provide a piecewise formula for the effective Hamiltonian $\Ham$. Then, we present our results that determine the set $\mathcal{L}(\Ham)$ of heights at which the graph of $\Ham$ has a flat piece (see Theorems \ref{thm:ergthmapp} and \ref{thm:intermed}). Finally, we discuss some features, corollaries and generalizations of our results, in particular to even-symmetric triple-well Hamiltonians.
	
\vspace{-1mm}

\subsubsection*{Section \ref{sec:prelim}} After citing fundamental theorems on the existence \& uniqueness of viscosity solutions and on reducing homogenization to the almost sure existence of a pointwise limit, we present elementary comparison-based results on how this pointwise limit follows from the existence of correctors (see Proposition \ref{prop:demtut}) or certain one-sided variants of them (see Proposition \ref{prop:onesided}). All statements are customized to our setting and purposes.

\vspace{-1mm}

\subsubsection*{Sections \ref{sec:weak}--\ref{sec:summary}} In the first three of these sections, we consider the three qualitatively distinct cases in the statement of Theorem \ref{thm:velinim}. 
In each case, to obtain the nonflat pieces of the graph of $\Ham$, we construct correctors and use Proposition \ref{prop:demtut}. Then, we carefully concatenate some of these correctors and use Proposition \ref{prop:onesided} to obtain the flat pieces. For the convenience of the reader, each piece has a designated subsection. The most important one is Subsection \ref{sub:ister} where we construct and work with correctors that are merely piecewise continuously differentiable. Finally, in Section \ref{sec:summary}, we put everything together and conclude the proof of Theorem \ref{thm:velinim}.

\vspace{-1mm}

\subsubsection*{Section \ref{sec:newflat}--\ref{sec:interflat}}

The piecewise formula in Theorem \ref{thm:velinim} identifies all intervals on which $\Ham$ is constant. However, not all of these intervals necessarily have positive length. In the proof of Theorem \ref{thm:ergthmapp}, we show that the strict inequality between the two endpoints of some of these intervals is characterized by the potential almost surely attaining its supremum or infimum. In the proof of Theorem \ref{thm:intermed}, in order to characterize the same strict inequality for the remaining intervals in question, we introduce two natural events regarding the upcrossings and the downcrossings of the potential. Both proofs involve elementary probability bounds and ultimately rely on the ergodic theorem. When combined, they determine the aforementioned set $\mathcal{L}(\Ham)$ of heights at which the graph of $\Ham$ has a flat piece.

\vspace{-1mm}

\subsubsection*{Section \ref{sec:examples}}

We illustrate our results by writing down the set $\mathcal{L}(\Ham)$ for two basic classes of examples. In the first class, the potential is the periodic extension of a single-well or double-well function $v_0$. We demonstrate how $\beta,m,M$ and the local extreme values of $v_0$ precisely interact to yield $\mathcal{L}(\Ham)$. In the second class, the potential is constructed by linearly interpolating a stationary \& ergodic process with index set $\mathbb{Z}$ and law $\mathbb{Q}$. Under additional structural assumptions, we show that $\mathcal{L}(\Ham)$ is determined by $\beta,m,M$ and the atoms of one-dimensional marginals or conditionals of $\mathbb{Q}$.

\section{Previous results}\label{sec:prevresults}

\subsection{Homogenization of HJ equations: General}\label{sub:prevgenel}

In this subsection, we review the literature on the homogenization of HJ equations that is directly relevant to our results (which we will present in Section \ref{sec:ourresults}). For this purpose, we temporarily generalize the setting in Section \ref{sec:intro} (which we will promptly return to in Subsection \ref{sub:ATY}).

Fix any $d\ge1$. Let $(\Omega,\mathcal{F},\mathbb{P})$ be a probability space equipped with a group of measure-preserving transformations $\tau_x:\Omega\to\Omega$, $x\in\mathbb{R}^d$, such that $\tau_0 = \text{id}$ and $\tau_x\circ \tau_y = \tau_{x+y}$ for every $x,y\in\mathbb{R}^d$. Assume that $\mathbb{P}$ is ergodic under this group. For every $\epsilon > 0$ and $\omega\in\Omega$, consider the HJ equation
\begin{equation}\label{eq:litHJ}
\partial_tu^\epsilon(t,x,\omega) + H(\nabla_xu^\epsilon(t,x,\omega),x/\epsilon,\omega) = 0,\quad(t,x)\in(0,+\infty)\times\mathbb{R}^d,
\end{equation}
where the function $H:\mathbb{R}^d\times\mathbb{R}^d\times\Omega\to[0,+\infty)$ is measurable and it satisfies the following properties:
\begin{align}
		  (p,x) \mapsto &H(p,x,\omega)\ \text{is in}\, \BUC(K\times\mathbb{R}^d)\ \text{for every bounded}\ K\subset\mathbb{R}^d\ \text{and}\ \omega\in\Omega,\label{eq:toplukonut1}\\
  \lim_{|p|\to +\infty} &H(p,x,\omega) = +\infty\ \ \text{uniformly in}\ (x,\omega)\in\mathbb{R}^d\times\Omega,\ \text{and}\label{eq:toplukonut2}\\
						&H(p,x,\omega) = H(p,0,\tau_x\omega)\ \text{for every}\ (p,x,\omega) \in \mathbb{R}^d\times\mathbb{R}^d\times\Omega\label{eq:toplukonut3}.
\end{align}

The first homogenization result for the HJ equation in \eqref{eq:litHJ} is due to Lions, Papanicolaou and Varadhan \cite{LPV87}. They focus on the special case where
$x = (x_1,\ldots,x_d) \mapsto H(p,x,\omega)$ is 1-periodic in $x_i$ for every $i\in\{1,\ldots,d\}$ and use the compactness of $[0,1]^d$ to solve the following auxiliary problem (which is formulated by them and referred to as the cell problem): for every $\theta\in\mathbb{R}^d$, find a $\lambda = \lambda(\theta)\in\mathbb{R}$ and a periodic function $h = h_\theta:\mathbb{R}^d\times\Omega\to\mathbb{R}$ such that
\begin{equation}\label{eq:cellprob}
	H(\theta + \nabla_x h(x,\omega),x,\omega) = \lambda,\quad x\in\mathbb{R}^d,
\end{equation}
in the viscosity sense. The solutions of this cell problem are called correctors. Their existence readily implies that the HJ equation in \eqref{eq:litHJ} homogenizes to
\[ \partial_t\overline u(t,x) + \Ham(\nabla_x\overline u(t,x)) = 0,\quad(t,x)\in(0,+\infty)\times\mathbb{R}^d, \]
with $\Hamg(\theta) = \lambda(\theta)$ for every $\theta\in\mathbb{R}$, albeit when a priori restricted to affine initial conditions. Then, Lions, Papanicolaou and Varadhan extend the class of initial conditions to all uniformly continuous functions by identifying and using some key properties of the semigroup induced by \eqref{eq:litHJ}.

In the stationary \& ergodic setting, on top of \eqref{eq:toplukonut1}--\eqref{eq:toplukonut3} and some additional mild conditions, if one assumes that $p \mapsto H(p,x,\omega)$ is convex, then the HJ equation in \eqref{eq:litHJ} homogenizes. This was shown by Souganidis \cite{S99} (in the $\mathbb{P}$-a.s.\ sense that we described in Section \ref{sec:intro}) and independently by Rezakhanlou and Tarver \cite{RT00} (in an $L^1(\mathbb{P})$ sense). In both papers, the main idea is to apply the subadditive ergodic theorem to a variational representation for the viscosity solutions of \eqref{eq:litHJ} that involves the Legendre transform of $p \mapsto H(p,x,\omega)$.

In their paper \cite{RT00} mentioned above, Rezakhanlou and Tarver also describe how one can try to adapt the approach of Lions, Papanicolaou and Varadhan \cite{LPV87} to the stationary \& ergodic setting. In this direction, following the work of Evans \cite{E92} on periodic Hamiltonians, they show that the desired homogenization result for \eqref{eq:litHJ} would be established if one could solve the following generalization of the cell problem: for every $\theta\in\mathbb{R}^d$, find a $\lambda = \lambda(\theta)\in\mathbb{R}$ and a Lipschitz continuous viscosity solution $h(\cdot,\omega) = h_\theta(\cdot,\omega):\mathbb{R}^d\to\mathbb{R}$ of \eqref{eq:cellprob} that satisfies
\begin{equation}\label{eq:esref}
	\lim_{|x|\to+\infty}\frac{h(x,\omega)}{|x|} = 0
\end{equation}
for $\mathbb{P}$-a.e.\ $\omega$. However, in a later paper \cite{LS03}, Lions and Souganidis demonstrate (by providing a convex counterexample in one space dimension) that such a sublinear corrector does not exist in general (at least if it is required to be bounded). Furthermore, whenever homogenization holds, they give a variational formula for the effective Hamiltonian.

The convexity assumption can be somewhat relaxed by assuming level-set convexity instead, i.e., for every $a\in[0,1]$, $p,q,x\in\mathbb{R}^d$ and $\omega\in\Omega$,  
\[ H(ap + (1-a)q,x,\omega) \le \max\{H(p,x,\omega),H(q,x,\omega) \}. \]
In this case, when $d=1$, Davini and Siconolfi \cite{DS09} prove homogenization by solving the following relaxation of the generalized cell problem: for every $\theta\in\mathbb{R}$ and $\delta>0$, find a $\lambda = \lambda(\theta)\in\mathbb{R}$ and a Lipschitz continuous function $h(\cdot,\omega) = h_{\theta,\delta}(\cdot,\omega):\mathbb{R}\to\mathbb{R}$ such that
\[ \lambda - \delta \le H(\theta + h'(x,\omega),x,\omega) \le \lambda + \delta,\quad x\in\mathbb{R}, \]
in the viscosity sense and \eqref{eq:esref} holds for $\mathbb{P}$-a.e.\ $\omega$. Such functions are called $\delta$-approximate correctors. When $d\ge1$, Armstrong and Souganidis \cite{AS13} bypass the existence of $\delta$-approximate correctors and instead prove homogenization by applying the subadditive ergodic theorem to the maximal solutions of \eqref{eq:cellprob} when the condition $x\in\mathbb{R}^d$ there is replaced by $x\in\mathbb{R}^d\setminus\{y\}$ for any $y\in\mathbb{R}^d$. Since these maximal solutions give an intrinsic distance between $x$ and $y$, they are referred to as the solutions of the metric problem.

The first homogenization result for a genuinely nonconvex example in the stationary \& ergodic setting is due to Armstrong, Tran and Yu \cite{ATY15} who consider separable Hamiltonians which are of the form
\[ H(p,x,\omega) = G(p) + V(x,\omega). \]
They assume the following: (i) $G(p) = (|p|^2 - 1)^2$; (ii) $x\mapsto V(x,\omega)$ is in $\BUC(\mathbb{R}^d)$. Their special choice of $G$ (most notably its radial symmetry) allows them to construct a family of static HJ equations that involves a free parameter $\sigma\in[-1,1]$ and generalizes \eqref{eq:cellprob} (in the sense that the latter is formally obtained when $\sigma=\pm 1$). They establish the desired homogenization result by applying the subadditive ergodic theorem to the maximal subsolutions of these equations and using comparison arguments. 

When $d=1$, the function $G(p) = (|p|^2 - 1)^2$ considered above is an even-symmetric double-well function. In their subsequent paper \cite{ATY16} in one space dimension, Armstrong, Tran and Yu generalize this homogenization result to essentially all separable Hamiltonians with a coercive $G$. (This is the work we cited in Section \ref{sec:intro} and we will present it in Subsection \ref{sub:ATY} below.) Finally, Gao \cite{G16} removes the separability assumption in \cite{ATY16} and thereby proves that, when $d=1$, the HJ equation in \eqref{eq:litHJ} homogenizes under the assumptions \eqref{eq:toplukonut1}--\eqref{eq:toplukonut3}.

In a relatively recent paper, Qian, Tran and Yu \cite{QTY18} go back to the periodic setting and consider a class of separable Hamiltonians with a coercive and nonconvex $G\in\Con(\mathbb{R}^d)$ that is, in particular, even-symmetric (i.e., $G(p) = G(-p)$ for every $p\in\mathbb{R}^d$). They introduce a decomposition method that is tailored to this class and provide a min-max formula for the effective Hamiltonian $\Hamg$. Moreover, using analytical and numerical methods, they partially answer the following interesting questions regarding the dependence of $\Hamg$ on the potential $V$ (which we will revisit in Subsection \ref{sub:remarks}):
\begin{equation}\label{eq:questions}
	\begin{aligned}
		&\text{When is $\Hamg$ level-set convex (despite the fact that $G$ is not)?}\\
		&\text{When is $\Hamg$ not even-symmetric (despite the fact that $G$ is)?}
	\end{aligned} 
\end{equation}
Their approach generalizes to the stationary \& ergodic setting with the same class of Hamiltonians, yielding a simpler proof of homogenization that covers the aforementioned result of Armstrong, Tran and Yu \cite{ATY15} with $G(p) = (|p|^2 - 1)^2$.

For other (positive and negative) results on the homogenization of HJ equations with nonconvex Hamiltonians, see \cite{FeS17,G19,Z17}. 

We end this review by mentioning that there are closely related works on the homogenization of second-order HJ equations, starting with \cite{KRV06,LS05} under the assumption of convexity. However, the literature on nonconvex Hamiltonians is relatively sparse. In particular, when $d=1$, homogenization is established in \cite{DK17,DK20+,KYZ20,YZ19} for certain classes of nonconvex Hamiltonians, but the picture is far from being complete. See also \cite{AC18,CS17,FeFZ19+} for some (positive and negative) results in higher dimensions.

\subsection{Homogenization of HJ equations: Separable Hamiltonians in one space dimension}\label{sub:ATY}

In this subsection, we present the aforementioned work of Armstrong, Tran and Yu \cite{ATY16} in more detail. For this purpose, we return to the setting in Section \ref{sec:intro}. In particular, we assume that the Hamiltonian is separable, $d=1$, and \eqref{eq:coercive}--\eqref{eq:evren} hold.

By a stability argument and a gluing procedure (see \cite[Lemmas 2.3 and 4.1]{ATY16}), it suffices to prove the desired homogenization result for the HJ equation in \eqref{eq:originalHJ} under several additional assumptions, mainly (when translated to our setting):
	\begin{align*}
		&\text{$G$ has a unique absolute minimum at $0$ and $G(0) = 0$,}\\
		&\text{$G$ has $L$ local minima (and hence $L$ local maxima) in $(-\infty,0)$ for some $L\ge0$,}\\
		&\text{$G$ has no local minima in $(0,+\infty)$, and}\\
		&\text{$x\mapsto V(x,\omega)$ is a smooth function whose level sets have no cluster points.}
	\end{align*}

Whenever $L\ge1$, denote the local minimum and the local maximum values of $G$ in $(-\infty,0)$ by $m_1,\ldots,m_L$ and $M_1,\ldots,M_L$, respectively. Let
\[ m_{\min} = \min\{m_1,\ldots,m_L\}\qquad\text{and}\qquad M_{\max} = \max\{M_1,\ldots,M_L\}. \]
If $\beta < M_{\max} - m_{\min}$, then two further gluing procedures (see \cite[Lemmas 4.2 and 4.3]{ATY16}) imply that the desired result follows from that for Hamiltonians with strictly less number of local minima. This is a strong induction argument with two base cases.

\subsubsection*{Base case 1: $L=0$}

This case is covered by \cite{AS13,DS09}. The effective Hamiltonian $\Ham$ inherits the level-set convexity of $G$. Moreover, there exist $\theta_\ell,\theta_r\in\mathbb{R}$ such that $\theta_\ell < 0 < \theta_r$ and
\begin{equation}\label{eq:ATYdis1}
	\text{$\Ham$ is}\ \begin{cases} \text{strictly decreasing on}\ (-\infty,\theta_\ell],\\
								   \text{identically equal to $\beta$ on}\ (\theta_\ell,\theta_r),\\
								   \text{strictly increasing on}\ [\theta_r,+\infty).
					  \end{cases}
\end{equation}

\subsubsection*{Base case 2: $L\ge1$ and $\beta \ge  M_{\max} - m_{\min}$}

This case is treated in \cite[Section 3]{ATY16}. The effective Hamiltonian $\Ham$ is level-set convex despite the fact that $G$ is not. Moreover, there exist $\theta_\ell,\theta_r\in\mathbb{R}$ such that $\theta_\ell < 0 < \theta_r$ and
\begin{equation}\label{eq:ATYdis2}
	\text{$\Ham$ is}\ \begin{cases} \text{strictly decreasing on the semi-infinite interval}\ \{\theta\le\theta_\ell:\,\Ham(\theta) \ge M_{\max} + \beta\},\\
								   \text{nonincreasing on the nonempty interval}\ \{\theta\le\theta_\ell:\,m_{\min} < \Ham(\theta) < M_{\max} + \beta\},\\
								   \text{strictly decreasing on the possibly empty interval}\ \{\theta\le\theta_\ell:\,\Ham(\theta) \le m_{\min}\},\\
								   \text{identically equal to $\beta$ on}\ (\theta_\ell,\theta_r),\\
								   \text{strictly increasing on}\ [\theta_r,+\infty).
					  \end{cases}
\end{equation}

\smallskip

In each of these two base cases, for every $\theta\notin(\theta_\ell,\theta_r)$, there is a sublinear corrector. In particular, for every $\theta$ in the intervals on which $\Ham$ is stated to be strictly monotone, it is easy to directly write down a sublinear corrector and obtain a simple formula for $\Ham(\theta)$ (which we omit here because we will provide such formulas in Subsection \ref{sub:hom} after introducing some notation). In fact, the strict monotonicity of $\Ham$ on those intervals can be deduced from this formula. On the other hand, in the second base case, when $\theta\le\theta_\ell$ and $m_{\min} < \Ham(\theta) < M_{\max} + \beta$, the proof of the existence of a sublinear corrector is not constructive (see \cite[Lemma 3.5]{ATY16}) and it does not yield a formula for $\Ham(\theta)$.

\section{Our results}\label{sec:ourresults}

\subsection{Double-well Hamiltonians}\label{sub:assump}

Recall the setting we described in Section \ref{sec:intro}. On top of the basic assumptions we have adopted there (which we will implicitly accept in the rest of the paper), we will state and prove our results under the additional assumption that $G$ is a double-well function. Here is the precise formulation.

\begin{condition}\label{cond:pm}
	There exist $m,M,p_m,p_M\in\mathbb{R}$ such that
	\[ p_m < p_M < 0,\quad G(0) = 0 \le G(p_m) = m < G(p_M) = M, \]
	and $G$ is of the form
	\[ G(p) = \begin{cases}
	G_1(p)&\text{if}\ p\in(-\infty,p_m],\\
	G_2(p)&\text{if}\ p\in[p_m,p_M],\\
	G_3(p)&\text{if}\ p\in[p_M,0],\\
	G_4(p)&\text{if}\ p\in[0,+\infty),
	\end{cases} \]
	where $G_1$ and $G_3$ (resp.\ $G_2$ and $G_4$) are defined on the intervals indicated next to them and they are strictly decreasing (resp.\ strictly increasing). See Figure \ref{fig:G}.
\end{condition}

\begin{figure}
	\includegraphics{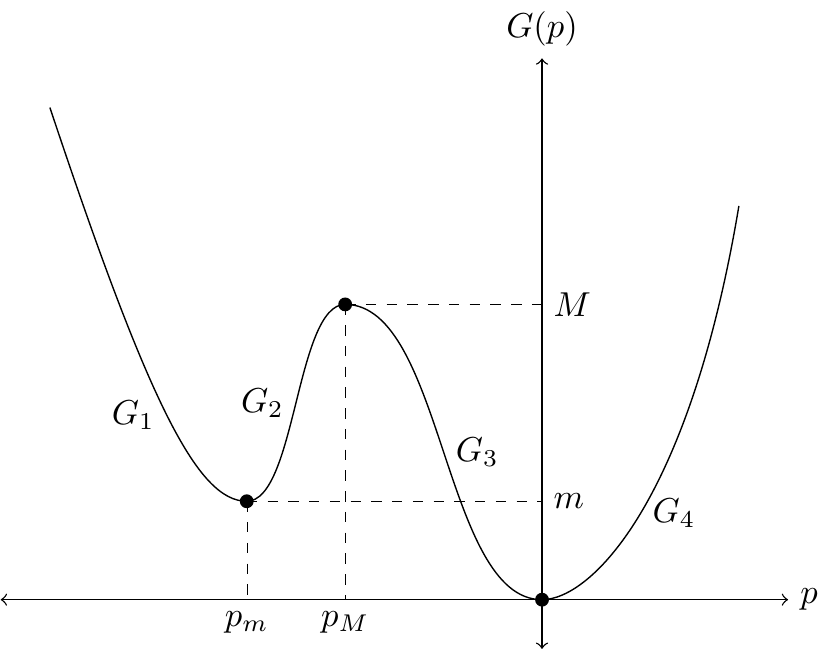}
	\caption{The graph of a function $G$ that satisfies Condition \ref{cond:pm} with $m > 0$.}
	\label{fig:G}
\end{figure}

Note that, if $G$ satisfies Condition \ref{cond:pm}, then it has exactly three extreme points: an absolute minimum at $0$, a local minimum at $p_m$ (which is an absolute minimum when $m=0$) and a local maximum at $p_M$. Conversely, if $G$ has exactly two local minima, then it satisfies Condition \ref{cond:pm} up to translation and reflection (if necessary) which would correspond to adding a linear (in $t$ and $x$) term to any viscosity solution of \eqref{eq:originalHJ} and substituting $-x$ for $x$ (see Subsection \ref{sub:remarks}).

\subsection{Homogenization}\label{sub:hom}

Define
\begin{align*}
&\theta_1: [m+\beta,+\infty) \to (-\infty,p_m), &&\theta_2: [m+\beta,M] \to (p_m,p_M),\\
&\theta_3: [\beta, M] \to (p_M,0),    &&\theta_4: [\beta,+\infty) \to (0,+\infty)
\end{align*}
by
\begin{equation}\label{eq:ddnono}
	\theta_i(\lambda) = \mathbb{E}\left[ G_i^{-1}(\lambda - \beta V(0,\omega)) \right],\quad i\in\{1,2,3,4\},
\end{equation}
whenever their domains are nonempty. It follows that $\theta_1$ and $\theta_3$ (resp.\ $\theta_2$ and $\theta_4$) are strictly decreasing (resp.\ strictly increasing) and continuous. With this notation, here is our first result.

\begin{theorem}\label{thm:velinim}
	Recall the (stationary \& ergodic) setting in Section \ref{sec:intro}. In particular, assume that \eqref{eq:coercive}--\eqref{eq:evren} hold. Moreover, impose Condition \ref{cond:pm} on $G$. Then, the HJ equation in \eqref{eq:originalHJ} homogenizes to the HJ equation in \eqref{eq:effectiveHJ} with a coercive $\Ham\in\Con(\mathbb{R})$. Precisely, there is an $\Omega_0\in\mathcal{F}$ with $\mathbb{P}(\Omega_0) = 1$ such that, for every $\omega\in\Omega_0$ and $g\in\UC(\mathbb{R})$, the unique viscosity solution $u_g^\epsilon(\cdot,\cdot,\omega)$ of \eqref{eq:originalHJ} with initial condition $g$ converges locally uniformly on $[0,+\infty)\times\mathbb{R}$ as $\epsilon\to0$ to the unique viscosity solution $\overline u_g$ of \eqref{eq:effectiveHJ} with the same initial condition.
	
	The effective Hamiltonian $\Ham$ has the following piecewise description.
\begin{description}
\item [Case I] If $\beta \le m + \beta < M$ (henceforth referred to as weak potential), then
\begin{equation}\label{eq:sakin1}
\Ham(\theta) =
\begin{cases}
\theta_1^{-1}(\theta)\ \text{on}\ (-\infty,\theta_1(m + \beta)]&\text{(strictly decreasing)},\\
m + \beta\ \text{on}\ (\theta_1(m+ \beta),\theta_2(m + \beta))&\text{(flat piece)},\\
\theta_2^{-1}(\theta)\ \text{on}\ [\theta_2(m + \beta),\theta_2(M)]&\text{(strictly increasing)},\\
M\ \text{on}\ (\theta_2(M),\theta_3(M))&\text{(flat piece)},\\
\theta_3^{-1}(\theta)\ \text{on}\ [\theta_3(M),\theta_3(\beta)]&\text{(strictly decreasing)},\\
\beta\ \text{on}\ (\theta_3(\beta),\theta_4(\beta))&\text{(flat piece)},\\
\theta_4^{-1}(\theta)\ \text{on}\ [\theta_4(\beta),+\infty)&\text{(strictly increasing)}.
\end{cases}
\end{equation} 
In particular, $\Ham$ is not level-set convex.

\item [Case II] If $\beta < M \le m + \beta$ (henceforth referred to as medium potential), then
\begin{equation}\label{eq:sakin2}
\Ham(\theta) =
\begin{cases}
\theta_1^{-1}(\theta)\ \text{on}\ (-\infty,\theta_1(m + \beta)]&\text{(strictly decreasing)},\\
\overline\Lambda(\theta)\ \text{on}\ (\theta_1(m + \beta),\theta_3(M))&\text{(nonincreasing)},\\
\theta_3^{-1}(\theta)\ \text{on}\ [\theta_3(M),\theta_3(\beta)]&\text{(strictly decreasing)},\\
\beta\ \text{on}\ (\theta_3(\beta),\theta_4(\beta))&\text{(flat piece)},\\
\theta_4^{-1}(\theta)\ \text{on}\ [\theta_4(\beta),+\infty)&\text{(strictly increasing)}.
\end{cases}
\end{equation}
The nonincreasing function $\overline\Lambda:(\theta_1(m + \beta),\theta_3(M))\to[M,m + \beta]$ is given in Definition \ref{def:lamb}. (See Remark \ref{rem:nov} below for a description.)
In particular, $\Ham$ is level-set convex.

\item [Case III] If $M \le \beta \le m + \beta$ (henceforth referred to as strong potential), then
\begin{equation}\label{eq:sakin3}
\Ham(\theta) =
\begin{cases}
\theta_1^{-1}(\theta)\ \text{on}\ (-\infty,\theta_1(m + \beta)]&\text{(strictly decreasing)},\\
\overline\Lambda(\theta)\ \text{on}\ (\theta_1(m + \beta),\theta_4(\beta))&\text{(nonincreasing)},\\
\theta_4^{-1}(\theta)\ \text{on}\ [\theta_4(\beta),+\infty)&\text{(strictly increasing)}.
\end{cases}
\end{equation}
The nonincreasing function $\overline\Lambda:(\theta_1(m + \beta),\theta_4(\beta))\to[\beta,m + \beta]$ is given in Definition \ref{def:lamb}. (See Remark \ref{rem:nov} below for a description.)
In particular, $\Ham$ is level-set convex.
\end{description}
\end{theorem}

\begin{remark}\label{rem:nov}
Theorem \ref{thm:velinim} is mostly covered by the homogenization result of Armstrong, Tran and Yu in \cite{ATY16}. Specifically, it corresponds to taking $L=1$ (so that $m_{\min} = m$ and $M_{\max} = M$) with the notation in Subsection \ref{sub:ATY} (assuming $m>0$). Therefore, Case I (where $\beta < M - m$) reduces to the first base case of their induction argument and $\Ham$ is obtained by gluing effective Hamiltonians of the form in \eqref{eq:ATYdis1}. (See \cite[Section 5]{ATY16}.) Similarly, Cases II and III (where $\beta \ge M - m$) are covered by the second base case of their induction argument and $\Ham$ satisfies \eqref{eq:ATYdis2}. Moreover, in all cases, the continuous functions $\theta_1,\theta_2,\theta_3,\theta_4$ (defined in \eqref{eq:ddnono}) whose inverses give the effective Hamiltonian on the intervals where it is stated to be strictly monotone are already provided in \cite{ATY16}. As we will see in Section \ref{sec:weak}, these four functions are associated to correctors in $\Cone(\mathbb{R})$.

The novelty of Theorem \ref{thm:velinim} lies in the identification of the function $\overline\Lambda$ in Cases II and III, i.e., when $\max\{\beta,M\} \le m + \beta$.
	\begin{itemize}
		\item If $\max\{\beta,M\} = m + \beta$, then $\overline\Lambda$ is identically equal to $\max\{\beta,M\}$ (see Definition \ref{def:lamb}(a,c)).
		\item If $\max\{\beta,M\} < m + \beta$, then $\overline\Lambda$ is the unique nonincreasing generalized inverse (see \eqref{eq:nowaste}) of a strictly decreasing and right-continuous function $\underline\theta_{1,3}$ (defined in \eqref{eq:hurma1}), possibly extended by adding flat pieces at the ends of its domain (see Definition \ref{def:lamb}(b,d)).
		The function $\underline\theta_{1,3}$ is associated to correctors that are piecewise continuously differentiable. Since the construction of those correctors involves several steps and more notation (see Subsection \ref{sub:ister}), the formal definition of $\overline\Lambda$ is postponed to Section \ref{sec:summary} where the proof of Theorem \ref{thm:velinim} is completed.
	\end{itemize}
\end{remark}

\subsection{Flat pieces within the nonincreasing piece}

Theorem \ref{thm:velinim} identifies intervals on which the effective Hamiltonian is either constant or strictly monotone, but it does not provide a complete list of such intervals. Precisely,
in the case of medium potential (i.e., when $\beta < M \le m + \beta$), on the interval $[\theta_1(m + \beta),\theta_3(M)]$, it merely states that $\Ham$ is nonincreasing,
\[ \Ham(\theta_1(m + \beta)) = m + \beta \quad \text{and} \quad \Ham(\theta_3(M)) = M. \]
Similarly, in the case of strong potential (i.e., when $M \le \beta \le m + \beta$), on the interval $[\theta_1(m + \beta),\theta_4(\beta)]$, it merely states that $\Ham$ is nonincreasing,
\[ \Ham(\theta_1(m + \beta)) = m + \beta \quad \text{and} \quad \Ham(\theta_4(\beta)) = \beta. \]
The following result characterizes the existence of flat pieces within these nonincreasing pieces that share one or two endpoints with them (i.e., at heights $\max\{\beta,M\}$ and $m + \beta$). It involves the events
\[ \{\omega\in\Omega: V(x,\omega) = 0\ \text{for some}\ x\in\mathbb{R}\}\quad\text{and}\quad\{\omega\in\Omega: V(x,\omega) = 1\ \text{for some}\ x\in\mathbb{R}\} \]
whose $\mathbb{P}$-probabilities are in $\{0,1\}$ by our ergodicity assumption. 

\begin{theorem}\label{thm:ergthmapp}
	Under the assumptions in Theorem \ref{thm:velinim}, the following are true.
	\begin{itemize}
		\item [(a)] If $\beta < M = m + \beta$ (henceforth referred to as the easy subcase of medium potential), then $\Ham$ is identically equal to $\max\{\beta,M\} = M$ on the whole interval $[\theta_1(M),\theta_3(M)]$.
		\item [(b)] If $\beta < M < m + \beta$ (i.e., medium potential, excluding the easy subcase above), then:
		\begin{itemize}
			\item [(i)] the graph of $\Ham$ has a flat piece at height $\max\{\beta,M\} = M$ iff
			\begin{equation}\label{eq:kimbegen}
				\mathbb{P}\left(V(x,\omega) = 0\ \text{for some}\ x\in\mathbb{R}\right) = 1;\ \text{and}
			\end{equation}
			\item [(ii)] the graph of $\Ham$ has a flat piece at height $m + \beta$ iff
			\begin{equation}\label{eq:ezgiha}
				\mathbb{P}\left(V(x,\omega) = 1\ \text{for some}\ x\in\mathbb{R}\right) = 1.
			\end{equation}
		\end{itemize}
		\item [(c)] If $M \le \beta = m + \beta$ (henceforth referred to as the easy subcase of strong potential), then $\Ham$ is identically equal to $\max\{\beta,M\} = \beta$ on the whole interval $[\theta_1(\beta),\theta_4(\beta)]$.
		\item [(d)] If $M \le \beta < m + \beta$ (i.e., strong potential, excluding the easy subcase above), then:
		\begin{itemize}
			\item [(i)] the graph of $\Ham$ always has a flat piece at height $\max\{\beta,M\} = \beta$; and
			\item [(ii)] the graph of $\Ham$ has a flat piece at height $m + \beta$ iff \eqref{eq:ezgiha} holds.
		\end{itemize}
	\end{itemize}
\end{theorem}

It remains to characterize the flat pieces in the graph of $\Ham$ that are fully in the interior of the aforementioned nonincreasing piece. To this end, for every $\lambda\in[\beta,+\infty)\cap(M,m + \beta)$, consider two bi-infinite sequences $(\underline x_i(\lambda,\omega))_{i\in\mathbb{Z}}$ and $(\overline y_i(\lambda,\omega))_{i\in\mathbb{Z}}$ that satisfy the coupled recursion
\begin{equation}\label{eq:dikdur1}
\begin{aligned}
	\underline x_i(\lambda,\omega) &= \inf\{ x\ge \overline y_{i-1}(\lambda,\omega):\, \lambda - \beta V(x,\omega) \ge M \},\\
	\overline y_i(\lambda,\omega) &= \inf\{ x\ge \underline x_i(\lambda,\omega):\, \lambda - \beta V(x,\omega) < m \}.
\end{aligned}
\end{equation}
Similarly, consider $(\overline x_i(\lambda,\omega))_{i\in\mathbb{Z}}$ and $(\underline y_i(\lambda,\omega))_{i\in\mathbb{Z}}$ that satisfy
\begin{equation}\label{eq:dikdur2}
\begin{aligned}
	\overline x_i(\lambda,\omega) &= \inf\{ x\ge \underline y_{i-1}(\lambda,\omega):\, \lambda - \beta V(x,\omega) > M \},\\
	\underline y_i(\lambda,\omega) &= \inf\{ x\ge \overline x_i(\lambda,\omega):\, \lambda - \beta V(x,\omega) \le m \}.
\end{aligned}
\end{equation}
Both pairs of sequences are well-defined for $\mathbb{P}$-a.e.\ $\omega$ under suitable ``anchoring" conditions (see Subsection \ref{sub:ister}). With this notation, we introduce the events
\begin{equation}\label{eq:upcross}
	\begin{aligned}
		U_\lambda &= \{\omega\in\Omega: \exists\,j\in\mathbb{Z}\ \text{s.t.}\ \underline x_0(\lambda,\omega) = \overline x_j(\lambda,\omega) \}\\
		&= \{\omega\in\Omega: \inf\{x\ge \underline x_0(\lambda,\omega):\,\lambda - \beta V(x,\omega) > M \} = \underline x_0(\lambda,\omega) \}\\
		&= \{\omega\in\Omega: \lambda - \beta V(\cdot,\omega)\ \text{does not have a local maximum at}\ \underline x_0(\lambda,\omega) \}
	\end{aligned}
\end{equation}
and
\begin{equation}\label{eq:downcross}
	\begin{aligned}
		D_\lambda &= \{\omega\in\Omega: \exists\,k\in\mathbb{Z}\ \text{s.t.}\ \underline y_0(\lambda,\omega) = \overline y_k(\lambda,\omega) \}\\
		&= \{\omega\in\Omega: \inf\{x\ge \underline y_0(\lambda,\omega):\,\lambda - \beta V(x,\omega) < m \} = \underline y_0(\lambda,\omega) \}\\
		&= \{\omega\in\Omega: \lambda - \beta V(\cdot,\omega)\ \text{does not have a local minimum at}\ \underline y_0(\lambda,\omega) \}
	\end{aligned}
\end{equation}
which, respectively, involve the upcrossings and the downcrossings of the function $x\mapsto\lambda - \beta V(x,\omega)$. In each of the displays \eqref{eq:upcross} and \eqref{eq:downcross}, the second and the third descriptions of the event are equivalent to the first one (see Section \ref{sec:interflat}) and they are provided here to convey some intuition.

\begin{theorem}\label{thm:intermed}
	Under the assumptions in Theorem \ref{thm:velinim}, when $\max\{\beta,M\} < m + \beta$ (i.e., medium or strong potential, excluding their easy subcases), for every $\lambda\in(\max\{\beta,M\},m + \beta)$, the graph of the effective Hamiltonian $\Ham$ has a flat piece at height $\lambda$ if and only if
	\[ \mathbb{P}(U_\lambda \cap D_\lambda) < 1. \]
\end{theorem}

\subsection{Some remarks and generalizations}\label{sub:remarks}

\subsubsection{The heights of all flat pieces}

It follows readily from Theorems \ref{thm:velinim}, \ref{thm:ergthmapp} and \ref{thm:intermed} that, under Condition \ref{cond:pm}, the set
\begin{equation}\label{eq:kume}
	\mathcal{L}(\Ham) = \{ \lambda\in\mathbb{R}:\,\text{the graph of $\Ham$ has a flat piece height $\lambda$} \}
\end{equation}
does not depend on the exact shape of the graphs of the strictly monotone functions $G_1,G_2,G_3,G_4$. Indeed:
\begin{itemize}
	\item when $\max\{\beta,M\} \ge m + \beta$ (which covers weak potential as well as the easy subcases of medium and strong potential), $\mathcal{L}(\Ham)$ is a subset of $\{\beta\}\cup\{m+\beta\}\cup\{M\}$ containing $\beta$ and it depends only on $\beta,m,M$;
	\item when $\max\{\beta,M\} < m + \beta$ (which covers medium and strong potential, except their easy subcases), $\mathcal{L}(\Ham)$ is a subset of $\{\beta\}\cup[\max\{\beta,M\},m + \beta]$ containing $\beta$ and it depends on $\beta,m,M$ plus the support of the law of $x\mapsto V(x,\omega)$ under $\mathbb{P}$.
\end{itemize}
Moreover, in the latter regime, random potentials can be constructed to make $\mathcal{L}(\Ham)$ equal to any desired finite or countable subset of $\{\beta\}\cup[\max\{\beta,M\},m + \beta]$ containing $\beta$. See Example \ref{ex:Markov} and Remark \ref{rem:desir} for such a construction.

\subsubsection{No reduction to smooth potentials, etc.}

As we mentioned in Subsection \ref{sub:ATY}, Armstrong, Tran and Yu \cite{ATY16} use a stability lemma to argue that, in order to prove homogenization, it suffices to consider smooth potentials whose level sets have no cluster points. However, since our results are mainly concerned with the fine properties of the effective Hamiltonian and their dependence on the potential, we do not make such reductions. Indeed, in all our constructions and proofs (except in Section \ref{sec:examples} where we analyze examples), we impose no conditions on the potential beyond \eqref{eq:stationary}--\eqref{eq:evren}. Similarly, we directly consider all admissible combinations of inequalities between $\beta, m + \beta$ and $M$ (see Cases I, II and III in Theorem \ref{thm:velinim}) instead of using a stability lemma to break equalities.

\subsubsection{Effective Hamiltonians under reflection}\label{subsub:mirror}

Given any double-well function $G^-$ that satisfies \eqref{eq:coercive} and Condition \ref{cond:pm}, consider its mirror image $G^+$ defined by
\[ \text{$G^+(p) = G^-(-p)$\ \ for every $p\in\mathbb{R}$.} \]
For any $\theta\in\mathbb{R}$, consider the HJ equation in \eqref{eq:originalHJ} with $G = G^+$ and initial condition $u^\epsilon(0,x,\omega) = \theta x$. Substituting $-x$ for $x$ corresponds to replacing $\theta$ with $-\theta$, the function $G^+$ with $G^-$, and the potential $x\mapsto V(x,\omega)$ with $x\mapsto V(-x,\omega)$. Therefore, Theorem \ref{thm:velinim} applies to both $G^-$ and $G^+$. Denote the corresponding effective Hamiltonians by $\Ham^-$ and $\Ham^+$, respectively. It is natural to ask the following question about them:
\begin{equation}\label{eq:simit}
	\text{Is\ \ $\Ham^+(\theta) = \Ham^-(-\theta)$\ \  for every $\theta\in\mathbb{R}$?}
\end{equation}
The answer is a corollary of Theorem \ref{thm:velinim}:

\begin{itemize}
	
	\item When $\max\{\beta,M\} \ge m + \beta$, 
	it is shown in \eqref{eq:sakin1}--\eqref{eq:sakin3} and Definition \ref{def:lamb}(a,c) that the effective Hamiltonian is determined by the functions $\theta_1,\theta_2,\theta_3,\theta_4$ (defined in \eqref{eq:ddnono}). Since the values of these functions do not change if we replace $x\mapsto V(x,\omega)$ with $x\mapsto V(-x,\omega)$, the answer to \eqref{eq:simit} is yes.
	
	\item When $\max\{\beta,M\} < m + \beta$, 
	it is shown in \eqref{eq:sakin2}--\eqref{eq:sakin3} and Definition \ref{def:lamb}(b,d) that the effective Hamiltonian depends not only on some of the functions $\theta_1,\theta_2,\theta_3,\theta_4$, but also on another function $\underline\theta_{1,3}$ (defined in \eqref{eq:hurma1}). Recall from Remark \ref{rem:nov} that this function is associated to correctors that are piecewise continuously differentiable. As we will see in Subsection \ref{sub:ister}, the construction of those correctors involves the bi-infinite sequences introduced in \eqref{eq:dikdur1}--\eqref{eq:dikdur2}. Therefore, in general, the values of $\underline\theta_{1,3}$ change if we replace $x\mapsto V(x,\omega)$ with $x\mapsto V(-x,\omega)$, and the answer to \eqref{eq:simit} is no.
	
\end{itemize}

\noindent Moreover, in the second parameter regime above, it follows from Theorem \ref{thm:intermed} that, even the sets $\mathcal{L}(\Ham^-)$ and $\mathcal{L}(\Ham^+)$ (with the definition in \eqref{eq:kume}) are not generally equal.

\subsubsection{A generalization to even-symmetric triple-well Hamiltonians}

Given any double-well function $G^-$ that satisfies \eqref{eq:coercive} and Condition \ref{cond:pm}, define $G^s$ by
\[ G^s(p) = \begin{cases} G^-(p) & \text{if}\ p\le 0,\\ G^+(p) & \text{if}\ p > 0, \end{cases} \]
where $G^+$ is the mirror image of $G^-$ as in the previous remark. Note that $G^s$ is a continuous, coercive and even-symmetric triple-well function with an absolute minimum at $0$ where $G^s(0) = 0$. The superscript $s$ stands for the word symmetric.

It follows immediately from Theorem \ref{thm:velinim} and a gluing result (see \cite[Lemma 4.1]{ATY16}) that the HJ equation in \eqref{eq:originalHJ} with $G = G^s$ homogenizes (which is of course covered by \cite{ATY16}) and the effective Hamiltonian $\Ham^s$ is given by
\[ \Ham^s(\theta) = \begin{cases} \Ham^-(\theta) & \text{if}\ \theta\le 0,\\ \Ham^+(\theta) & \text{if}\ \theta > 0, \end{cases} \]
with our notation in the previous remark. Recalling our answer to \eqref{eq:simit} there, we deduce the following properties of $\Ham^s$:
\begin{itemize}
	\item if $\max\{\beta,M\} > m + \beta$, 
		  then $\Ham^s$ is even-symmetric but not level-set convex;
	\item if $\max\{\beta,M\} = m + \beta$, 
		  then $\Ham^s$ is even-symmetric and level-set convex;
	\item if $\max\{\beta,M\} < m + \beta$, 
		  then $\Ham^s$ is level-set convex but not generally even-symmetric.
\end{itemize}

This trichotomy rigorously answers the questions in \eqref{eq:questions} in the context of triple-well Hamiltonians when $d=1$. To the best of our knowledge, it was first given by Qian, Tran and Yu \cite{QTY18}. Precisely, in the third parameter regime above, they consider a periodic potential and briefly mention how one can show that $\Ham^s$ is not even-symmetric (see \cite[Remark 4]{QTY18}). Additionally, for a periodic example with a piecewise linear and even-symmetric triple-well function $G^s$, they use a Lax-Friedrichs based numerical method to plot $\Ham^s$ which is then observed to be not even-symmetric (see \cite[Numerical Example 2]{QTY18}). We believe that our results completely clarify this connection between the emergence of level-set convexity and the loss of even-symmetry in the 1-d stationary \& ergodic setting.

\subsubsection{Removing Condition \ref{cond:pm}}

The fully constructive approach that we take in this paper under Condition \ref{cond:pm} can be adopted in the more general setting of the homogenization result of Armstrong, Tran and Yu \cite{ATY16} to obtain a formula for the effective Hamiltonian $\Ham$ on the whole real line and to identify the set $\mathcal{L}(\Ham)$ (defined in \eqref{eq:kume}). Recalling their proof which is outlined in Subsection \ref{sub:ATY}, it essentially suffices to consider the second base case of their induction argument and focus on the interval where $\Ham$ is stated to be nonincreasing in \eqref{eq:ATYdis2}. 
We plan to include this generalization in a future paper.

\section{Preliminaries}\label{sec:prelim}

For an introduction to viscosity solutions of first-order HJ equations, see, e.g., \cite{BC97,E10}. Solutions, subsolutions and supersolutions of all HJ equations considered in this paper are to be understood in the viscosity sense unless noted otherwise. We henceforth drop the word viscosity for the sake of brevity.

We start by stating an existence \& uniqueness result that is tailored to our setting and purposes. It is covered by, e.g.,  the version in \cite[Theorem A.1]{DZ15}.

\begin{theorem}\label{thm:vartek}
	If \eqref{eq:coercive} and \eqref{eq:bucon} hold, then for every $g\in\UC(\mathbb{R})$, the HJ equations in \eqref{eq:originalHJ} (with any fixed $\omega\in\Omega$) and \eqref{eq:effectiveHJ} (with a coercive $\overline H\in\Con(\mathbb{R})$ to be determined) have unique solutions $u_g^\epsilon(\cdot,\cdot,\omega),\overline u_g \in\UC([0,+\infty)\times\mathbb{R})$ that satisfy
	\[ u_g^\epsilon(0,x,\omega) = \overline u_g(0,x) = g(x),\quad x\in\mathbb{R}.\]
	Moreover, if $g\in\Lip(\mathbb{R})$, the space of Lipschitz continuous functions on $\mathbb{R}$ with a uniform Lipschitz constant, then $u_g^\epsilon(\cdot,\cdot,\omega),\overline u_g\in\Lip([0,+\infty)\times\mathbb{R})$.
\end{theorem}

When $g(x) = \theta x$ for some $\theta\in\mathbb{R}$, we will write $u_\theta^\epsilon$ and $\overline u_\theta$ instead of $u_g^\epsilon$ and $\overline u_g$, respectively. Moreover, when $\epsilon = 1$ and \eqref{eq:originalHJ} becomes
\begin{equation}\label{eq:originalHJ1}
\partial_tu(t,x,\omega) + G(\partial_xu(t,x,\omega)) + \beta V(x,\omega) = 0,\quad(t,x)\in(0,+\infty)\times\mathbb{R},
\end{equation}
we will drop the superscript of $u_\theta^\epsilon$ and simply write $u_\theta$. With this notation,
\begin{align}
	u_\theta^\epsilon(t,x,\omega) &= \epsilon u_\theta(t/\epsilon,x/\epsilon,\omega),\label{eq:simden}\\
	u_\theta(t,x,\tau_y\omega) &= u_\theta(t,x+y,\omega) - \theta y\quad\text{and}\label{eq:araya}\\
	\overline u_\theta(t,x) &= -\Ham(\theta)t + \theta x\label{eq:okyay}
\end{align}
for every $t\ge0$ and $x,y\in\mathbb{R}$. Indeed, in each line, the functions on the left- and the right-hand sides are solutions of the same HJ equation with the same linear initial condition, so they are equal by uniqueness.

If the HJ equation in \eqref{eq:originalHJ} homogenizes to the HJ equation in \eqref{eq:effectiveHJ} (as claimed in Theorem \ref{thm:velinim} with general $g\in\UC(\mathbb{R})$), then we can choose to restrict our attention to linear initial conditions and deduce that, for every $\theta\in\mathbb{R}$ and $\mathbb{P}$-a.e.\ $\omega$,
$u_\theta^\epsilon(\cdot,\cdot,\omega)$ converges locally uniformly on $[0,+\infty)\times\mathbb{R}$ as $\epsilon\to0$ to $\overline u_\theta$. In particular, we have pointwise convergence at $(t,x) = (1,0)$, which is equivalent to
\begin{equation}\label{eq:suhnaz}
	\mathbb{P}\left( \lim_{\epsilon\to0}\epsilon u_\theta(1/\epsilon,0,\omega) = -\Ham(\theta) \right) = 1
\end{equation}
by \eqref{eq:simden} and \eqref{eq:okyay}. In short, homogenization implies the almost sure limit in \eqref{eq:suhnaz}. The following result (Theorem \ref{thm:ottur}) states that the converse is also true. It involves the quantities
\[ \Ham^L(\theta) = \liminf_{t\to+\infty}\frac{-u_\theta(t,0,\omega)}{t} \quad \text{and} \quad\Ham^U(\theta) = \limsup_{t\to+\infty}\frac{-u_\theta(t,0,\omega)}{t}. \]
Since $u_\theta(\cdot,\cdot,\omega)\in\Lip([0,+\infty)\times\mathbb{R})$ by Theorem \ref{thm:vartek} and it satisfies \eqref{eq:araya}, our ergodicity assumption ensures that $\Ham^L(\theta)$ and $\Ham^U(\theta)$ are $\mathbb{P}$-essentially constant. Whenever $\Ham^L(\theta)$ and $\Ham^U(\theta)$ are equal, we set
\begin{equation}\label{eq:pack}
	\Ham(\theta) = \Ham^L(\theta) = \Ham^U(\theta).
\end{equation}

\begin{theorem}\label{thm:ottur}
	Recall the (stationary \& ergodic) setting in Section \ref{sec:intro}. In particular, assume that \eqref{eq:coercive}--\eqref{eq:bucon} hold. If $\Ham^L(\theta) = \Ham^U(\theta)$ for every $\theta\in\mathbb{R}$, then the HJ equation in \eqref{eq:originalHJ} homogenizes to a HJ equation of the form in \eqref{eq:effectiveHJ} as in the statement of Theorem \ref{thm:velinim}. Moreover, the effective Hamiltonian $\Ham$ is given by \eqref{eq:pack}.
\end{theorem}

\begin{proof}
	It follows from \eqref{eq:araya} and our ergodicity assumption that the uniform Lipschitz constant of $u_\theta(\cdot,\cdot,\omega)$ is $\mathbb{P}$-essentially constant. Hence, the desired result is a special case of \cite[Lemma 4.1]{DK17}.
\end{proof}

For every $\lambda \in [\beta,+\infty)$ and $\omega\in\Omega$, consider the static HJ equation
\begin{equation}\label{eq:aux}
	G(f'(x,\omega)) + \beta V(x,\omega) = \lambda,\quad x\in\mathbb{R}.
\end{equation}
Recall that a function $f = f(\cdot,\omega)\in\Con(\mathbb{R})$ is a subsolution of \eqref{eq:aux} if
\begin{equation}\label{eq:vsol}
	G(p) + \beta V(x,\omega) \le \lambda \quad\text{for every $p\in D^+f(x,\omega)$}
\end{equation}
and every $x\in\mathbb{R}$. Similarly, a function $f = f(\cdot,\omega)\in\Con(\mathbb{R})$ is a supersolution of \eqref{eq:aux} if
\begin{equation}\label{eq:vsag}
G(p) + \beta V(x,\omega) \ge \lambda \quad\text{for every $p\in D^-f(x,\omega)$}
\end{equation}
and every $x\in\mathbb{R}$. Finally, a function $f = f(\cdot,\omega)\in\Con(\mathbb{R})$ is a solution of \eqref{eq:aux} if both of \eqref{eq:vsol} and \eqref{eq:vsag} are satisfied at every $x\in\mathbb{R}$. Here, $D^+f(x,\omega)$ and $D^-f(x,\omega)$ denote the (Frech\'et) superdifferential and subdifferential of $f$ at $x$, respectively. In particular, if $f$ is left and right differentiable at $x$, then
\[ D^+f(x,\omega) = [f_+'(x,\omega), f_-'(x,\omega)]\quad\text{and}\quad D^-f(x,\omega) =  [f_-'(x,\omega), f_+'(x,\omega)]. \]

There is an elementary connection between the static HJ equation in \eqref{eq:aux} and the evolutionary HJ equation in \eqref{eq:originalHJ1}. 
Namely, if $f(\cdot,\omega)\in\Con(\mathbb{R})$ is a subsolution (resp.\ supersolution) of \eqref{eq:aux}, then $u(\cdot,\cdot,\omega)\in\Con([0,+\infty)\times\mathbb{R})$, defined by
\[ u(t,x,\omega) = -\lambda t + f(x,\omega), \]
is a subsolution (resp.\ supersolution) 
of \eqref{eq:originalHJ}.
In Sections \ref{sec:weak}--\ref{sec:strong}, we will construct subsolutions and supersolutions of \eqref{eq:originalHJ1} that are of this form. 
Moreover, whenever applicable (which will turn out to be outside the flat pieces of the graph of the effective Hamiltonian), we will refer to the following proposition to establish the almost sure limit in \eqref{eq:suhnaz}. It is a customized version of a classical result (see, e.g., \cite[Theorem 4.1]{RT00}).

\begin{proposition}\label{prop:demtut}
	Assume \eqref{eq:coercive} and \eqref{eq:bucon}. If there exist a constant $\theta\in\mathbb{R}$, a function $f:\mathbb{R}\times\Omega\to\mathbb{R}$ and an event $\Omega_0\in\mathcal{F}$ with $\mathbb{P}(\Omega_0) = 1$ such that
	\begin{align}
	&\text{$f_-'(\cdot,\omega),f_+'(\cdot,\omega)$ exist and they are uniformly bounded on $\mathbb{R}$},\label{eq:kostas}\\
	&\lim_{x\to\pm\infty} \frac{f(x,\omega)}{x} = \theta,\quad\text{and}\nonumber\\
	&\text{$f(\cdot,\omega)$ is a solution of \eqref{eq:aux} with some $\lambda \in [\beta,+\infty)$}\nonumber
	\end{align}
	for every $\omega\in\Omega_0$, then
	\[ \Ham(\theta) = \Ham^L(\theta) = \Ham^U(\theta) = \lambda. \]
\end{proposition}

\begin{proof}
	Let $\varphi(x) = \sqrt{1+ x^2}$. For every $\epsilon>0$ and $\omega\in\Omega_0$, define $v= v(\cdot,\cdot,\omega)\in\Lip([0,+\infty)\times\mathbb{R})$ by
	\[ v(t,x,\omega) = -(\lambda + \epsilon)t + f(x,\omega) - \delta\varphi(x) - K, \]
	where $\delta,K>0$ are to be determined. We will show that $v$ is a subsolution of \eqref{eq:originalHJ1}. To this end,
	fix any \[ p\in [\partial_x^+v(t,x,\omega),\partial_x^-v(t,x,\omega)] = [f_+'(x,\omega) - \delta\varphi'(x),f_-'(x,\omega) - \delta\varphi'(x)]. \]
	(If this interval is empty, then we are done.) Since \[ p + \delta\varphi'(x) \in [f_+'(x,\omega),f_-'(x,\omega)], \]
	it follows from \eqref{eq:vsol} that
	\begin{align*}
	G(p) + \beta V(x,\omega) &= G(p + \delta\varphi'(x) - \delta\varphi'(x)) + \beta V(x,\omega)\\
	&\le G(p + \delta\varphi'(x)) + \epsilon + \beta V(x,\omega) \le \lambda + \epsilon = -\partial_tv(t,x,\omega)
	\end{align*}
	when $\delta = \delta(\epsilon,\omega) > 0$ is sufficiently small. 
	Therefore, $v$ is indeed a subsolution of \eqref{eq:originalHJ1}. Moreover, 
	for $K = K(\delta,\omega) > 0$ sufficiently large,
	\[ v(0,x,\omega) =  f(x,\omega) - \delta\varphi(x) - K = \theta x + o(|x|) - \delta\varphi(x) - K \le \theta x = u_\theta(0,x,\omega) \]
	for every $x\in\mathbb{R}$. By the comparison principle in Proposition \ref{prop:comprin} (see below),
	\[ \Ham^U(\theta) = \limsup_{t\to +\infty}\frac{-u_\theta(t,0,\omega)}{t} \le \lim_{t\to +\infty}\frac{-v(t,0,\omega)}{t} = \lambda + \epsilon. \]
	Similarly, for every $\epsilon>0$ and $\omega\in\Omega_0$, there exist $\delta,K>0$ such that
	\[ w(t,x,\omega) = -(\lambda - \epsilon)t + f(x,\omega) + \delta\varphi(x) + K \]
	defines a supersolution $w = w(\cdot,\cdot,\omega)\in\Lip([0,+\infty)\times\mathbb{R})$ of \eqref{eq:originalHJ1} that satisfies
	\[ w(0,x,\omega) \ge \theta x = u_\theta(0,x,\omega) \]
	for every $x\in\mathbb{R}$. By the comparison principle in Proposition \ref{prop:comprin} (see below),
	\[ \Ham^L(\theta) = \liminf_{t\to +\infty}\frac{-u_\theta(t,0,\omega)}{t} \ge \lim_{t\to +\infty}\frac{-w(t,0,\omega)}{t} = \lambda - \epsilon. \]
	Since $\epsilon > 0$ is arbitrary, the desired equalities follow.
\end{proof}

\begin{proposition}\label{prop:comprin}
	Assume \eqref{eq:coercive} and \eqref{eq:bucon}. For any $\omega\in\Omega$, if $u_1(\cdot,\cdot,\omega),u_2(\cdot,\cdot,\omega)\in\Lip([0,+\infty)\times\mathbb{R})$ are, respectively, a subsolution and a supersolution of the HJ equation in \eqref{eq:originalHJ1}, then
	\[ u_1(t,x,\omega) - u_2(t,x,\omega) \le \sup\{ u_1(0,y,\omega) - u_2(0,y,\omega):\,y\in\mathbb{R} \}\ \ \text{for every}\ (t,x)\in[0,+\infty)\times\mathbb{R}. \]
\end{proposition}

\begin{proof}
	This comparison principle follows from, e.g., the more general one in \cite[Proposition A.2]{DZ15}.
\end{proof}

We will use the following one-sided variant of Proposition \ref{prop:demtut} to obtain the flat pieces of the graph of the effective Hamiltonian.

\begin{proposition}\label{prop:onesided}
	Assume \eqref{eq:coercive} and \eqref{eq:bucon}. Suppose there exist constants $\theta_-,\theta_+\in\mathbb{R}$, a function $f:\mathbb{R}\times\Omega\to\mathbb{R}$ and an event $\Omega_0\in\mathcal{F}$ with $\mathbb{P}(\Omega_0) = 1$ such that \eqref{eq:kostas} holds,
	\[ \lim_{x\to-\infty} \frac{f(x,\omega)}{x} = \theta_- \quad\text{and}\quad \lim_{x\to+\infty} \frac{f(x,\omega)}{x} = \theta_+ \]
	for every $\omega\in\Omega_0$.
	\begin{itemize}
		\item [(a)] If $\theta_+ < \theta_-$ and $f(\cdot,\omega)$ is a subsolution of \eqref{eq:aux}	with some $\lambda \in [\beta,+\infty)$ for every $\omega\in\Omega_0$, then
		\[ \Ham^U(\theta) \le \lambda\ \ \text{for every}\ \theta\in(\theta_+,\theta_-). \]
		\item [(b)] If $\theta_- < \theta_+$ and $f(\cdot,\omega)$ is a supersolution of \eqref{eq:aux} with some $\lambda \in [\beta,+\infty)$ for every $\omega\in\Omega_0$, then
		\[ \Ham^L(\theta) \ge \lambda\ \ \text{for every}\ \theta\in(\theta_-,\theta_+). \]
	\end{itemize}
\end{proposition}

\begin{proof}
	Under the conditions that apply to part (a), for any $\theta\in(\theta_+,\theta_-)$,
	it is easy to check that
	\[ v(t,x,\omega) = -\lambda t + f(x,\omega) - K \]
	gives a subsolution of \eqref{eq:originalHJ1} and $v(0,x,\omega) \le \theta x = u_\theta(0,x,\omega)$ when $K = K(\theta,\omega)$ is sufficiently large. By the comparison principle in Proposition \ref{prop:comprin},
	\[ \Ham^U(\theta) = \limsup_{t\to +\infty}\frac{-u_\theta(t,0,\omega)}{t} \le \lim_{t\to +\infty}\frac{-v(t,0,\omega)}{t} = \lambda. \]
	The proof of part (b) is similar.
\end{proof}

It is evident from the proofs above that the conditions in Propositions \ref{prop:demtut} and \ref{prop:onesided} are not optimal (e.g., \eqref{eq:kostas} can be replaced with $f\in\Lip(\mathbb{R})$), but they are sufficient for our setting and purposes. Recall Subsection \ref{sub:prevgenel} and note that, for any $\theta\in\mathbb{R}$ and $f:\mathbb{R}\times\Omega\to\mathbb{R}$ as in Proposition \ref{prop:demtut}, 
\[ (x,\omega)\mapsto f(x,\omega) - \theta x \]
is a sublinear corrector. However, in the rest of the paper, we will not use this terminology because we will directly construct and work with functions as in Propositions \ref{prop:demtut} and \ref{prop:onesided} rather than their sublinearized versions. 

\section{Case I: Weak potential ($\beta \le m + \beta < M$)}\label{sec:weak}
	
\subsection{Strictly decreasing piece involving $G_1$}\label{sub:dec1}

In this subsection, we will simultaneously consider Cases I, II and III of Theorem \ref{thm:velinim}.

For every $\lambda \in [m + \beta,+\infty)$ and $\omega\in\Omega$, construct $f_1^\lambda(\cdot,\omega)\in\Cone(\mathbb{R})$ by setting $f_1^\lambda(0,\omega) = 0$ and
\begin{equation}\label{eq:ezgi1}
	(f_1^\lambda)'(x,\omega) = G_1^{-1}(\lambda - \beta V(x,\omega)) \in [G_1^{-1}(\lambda),G_1^{-1}(\lambda - \beta)] \subset (-\infty,p_m].
\end{equation}
$f_1^\lambda(\cdot,\omega)$ is a (classical) solution of \eqref{eq:aux}. Moreover, the ergodic theorem ensures that, for $\mathbb{P}$-a.e.\ $\omega$,
\[ \lim_{x\to\pm\infty}\frac{f_1^\lambda(x,\omega)}{x} = \lim_{x\to\pm\infty}\frac1{x}\int_0^x G_1^{-1}(\lambda- \beta V(y,\omega)) dy = \mathbb{E}\left[ G_1^{-1}(\lambda - \beta V(0,\omega)) \right] = \theta_1(\lambda) \in (-\infty,p_m) \]
with the definition in \eqref{eq:ddnono}. Therefore,
\[ \Ham(\theta_1(\lambda)) = \Ham^L(\theta_1(\lambda)) = \Ham^U(\theta_1(\lambda)) = \lambda \]
by Proposition \ref{prop:demtut}. Finally, since $\theta_1:\,[m + \beta,+\infty)\to (-\infty,\theta_1(m + \beta)]$ is a strictly decreasing bijection, we deduce that $\Ham$ is strictly decreasing on $(-\infty,\theta_1(m + \beta)]$.

\subsection{Strictly increasing piece involving $G_2$}\label{sub:inc2}

For every $\lambda \in [m + \beta,M]$ and $\omega\in\Omega$, construct $f_2^\lambda(\cdot,\omega)\in\Cone(\mathbb{R})$ by setting $f_2^\lambda(0,\omega) = 0$ and
\begin{equation}\label{eq:ezgi2}
	(f_2^\lambda)'(x,\omega) = G_2^{-1}(\lambda - \beta V(x,\omega)) \in [G_2^{-1}(\lambda - \beta),G_2^{-1}(\lambda)] \subset [p_m,p_M].
\end{equation}
$f_2^\lambda(\cdot,\omega)$ is a (classical) solution of \eqref{eq:aux}. Moreover, the ergodic theorem ensures that, for $\mathbb{P}$-a.e.\ $\omega$,
\[ \lim_{x\to\pm\infty}\frac{f_2^\lambda(x,\omega)}{x} = \lim_{x\to\pm\infty}\frac1{x}\int_0^x G_2^{-1}(\lambda- \beta V(y,\omega)) dy = \mathbb{E}\left[ G_2^{-1}(\lambda - \beta V(0,\omega)) \right] =\theta_2(\lambda) \in (p_m,p_M) \]
with the definition in \eqref{eq:ddnono}. Therefore,
\[ \Ham(\theta_2(\lambda)) = \Ham^L(\theta_2(\lambda)) = \Ham^U(\theta_2(\lambda)) = \lambda \]
by Proposition \ref{prop:demtut}. Finally, since $\theta_2:\,[m + \beta,M]\to [\theta_2(m + \beta),\theta_2(M)]$ is a strictly increasing bijection, we deduce that $\Ham$ is strictly increasing on $[\theta_2(m + \beta),\theta_2(M)]$.

\subsection{Strictly decreasing piece involving $G_3$}\label{sub:dec3}

In this subsection, we will simultaneously consider Cases I and II of Theorem \ref{thm:velinim}.

Assume that $\beta < M$. For every $\lambda \in [\beta, M]$ and $\omega\in\Omega$, construct $f_3^\lambda(\cdot,\omega)\in\Cone(\mathbb{R})$ by setting $f_3^\lambda(0,\omega) = 0$ and
\begin{equation}\label{eq:ezgi3}
	(f_3^\lambda)'(x,\omega) = G_3^{-1}(\lambda - \beta V(x,\omega)) \in [G_3^{-1}(\lambda),G_3^{-1}(\lambda - \beta)] \subset [p_M,0].
\end{equation}
$f_3^\lambda(\cdot,\omega)$ is a (classical) solution of \eqref{eq:aux}. Moreover, the ergodic theorem ensures that, for $\mathbb{P}$-a.e.\ $\omega$,
\[ \lim_{x\to\pm\infty}\frac{f_3^\lambda(x,\omega)}{x} = \lim_{x\to\pm\infty}\frac1{x}\int_0^x G_3^{-1}(\lambda - \beta V(y,\omega)) dy = \mathbb{E}\left[ G_3^{-1}(\lambda - \beta V(0,\omega)) \right] = \theta_3(\lambda) \in (p_M,0) \]
with the definition in \eqref{eq:ddnono}. Therefore,
\[ \Ham(\theta_3(\lambda)) = \Ham^L(\theta_3(\lambda)) = \Ham^U(\theta_3(\lambda)) = \lambda \]
by Proposition \ref{prop:demtut}. Finally, since $\theta_3:\,[\beta, M]\to [\theta_3(M),\theta_3(\beta)]$ is a strictly decreasing bijection, we deduce that $\Ham$ is strictly decreasing on $[\theta_3(M),\theta_3(\beta)]$.

\subsection{Strictly increasing piece involving $G_4$}\label{sub:inc4}

In this subsection, we will simultaneously consider Cases I, II and III of Theorem \ref{thm:velinim}.

For every $\lambda \in [\beta,+\infty)$ and $\omega\in\Omega$, construct $f_4^\lambda(\cdot,\omega)\in\Cone(\mathbb{R})$ by setting $f_4^\lambda(0,\omega) = 0$ and
\begin{equation}\label{eq:ezgi4}
	(f_4^\lambda)'(x,\omega) = G_4^{-1}(\lambda - \beta V(x,\omega)) \in [G_4^{-1}(\lambda - \beta),G_4^{-1}(\lambda)] \subset [0,+\infty).
\end{equation}
$f_4^\lambda(\cdot,\omega)$ is a (classical) solution of \eqref{eq:aux}. Moreover, the ergodic theorem ensures that, for $\mathbb{P}$-a.e.\ $\omega$,
\[ \lim_{x\to\pm\infty}\frac{f_4^\lambda(x,\omega)}{x} = \lim_{x\to\pm\infty}\frac1{x}\int_0^x G_4^{-1}(\lambda - \beta V(y,\omega)) dy = \mathbb{E}\left[ G_4^{-1}(\lambda - \beta V(0,\omega)) \right] = \theta_4(\lambda) \in (0,+\infty) \]
with the definition in \eqref{eq:ddnono}. Therefore,
\[ \Ham(\theta_4(\lambda)) = \Ham^L(\theta_4(\lambda)) = \Ham^U(\theta_4(\lambda)) = \lambda \]
by Proposition \ref{prop:demtut}. Finally, since $\theta_4:\,[\beta,+\infty)\to [\theta_4(\beta),+\infty)$ is a strictly increasing bijection, we deduce that $\Ham$ is strictly increasing on $[\theta_4(\beta),+\infty)$.

\subsection{Flat piece at height $\beta$}\label{sub:flatbeta}

In this subsection, we will simultaneously consider Cases I and II of Theorem \ref{thm:velinim}.

Assume that $\beta < M$. For every $\epsilon\in(0,\beta)$ and $\mathbb{P}$-a.e.\ $\omega$, there exists a $y_0 = y_0(\omega)\in\mathbb{R}$ such that $\beta V(y_0,\omega) = \beta - \epsilon$. Construct $f_{4,3}^\beta(\cdot,\omega),f_{3,4}^\beta(\cdot,\omega)\in\Lip(\mathbb{R})$ by setting $f_{4,3}^\beta(0,\omega) = f_{3,4}^\beta(0,\omega) = 0$,
\[ (f_{4,3}^\beta)'(x,\omega) = \begin{cases} (f_4^\beta)'(x,\omega)&\text{if}\ x < y_0,\\ (f_3^\beta)'(x,\omega)&\text{if}\ x > y_0\end{cases} \quad\text{and}\quad (f_{3,4}^\beta)'(x,\omega) = \begin{cases} (f_3^\beta)'(x,\omega)&\text{if}\ x < y_0,\\ (f_4^\beta)'(x,\omega)&\text{if}\ x > y_0\end{cases} \]
with the notation in \eqref{eq:ezgi3} and \eqref{eq:ezgi4}. Observe that, at every $x\in\mathbb{R}\setminus\{y_0\}$, the functions $f_{4,3}^\beta(\cdot,\omega)$ and $f_{3,4}^\beta(\cdot,\omega)$ satisfy both \eqref{eq:vsol} and \eqref{eq:vsag} with $\lambda = \beta$. In addition,
\begin{align*}
(f_{4,3}^\beta)_+'(y_0,\omega) &= (f_{3,4}^\beta)_-'(y_0,\omega) = (f_3^\beta)'(y_0,\omega) = G_3^{-1}(\beta - \beta V(y_0,\omega)) = G_3^{-1}(\epsilon) < 0\quad\text{and}\\
(f_{4,3}^\beta)_-'(y_0,\omega) &= (f_{3,4}^\beta)_+'(y_0,\omega) = (f_4^\beta)'(y_0,\omega) = G_4^{-1}(\beta - \beta V(y_0,\omega)) = G_4^{-1}(\epsilon) > 0.
\end{align*}
It is clear from Condition \ref{cond:pm} that
\[ \beta - \epsilon \le G(p) + \beta V(y_0,\omega) = G(p) + \beta - \epsilon \le \beta \]
for every $p\in [G_3^{-1}(\epsilon),G_4^{-1}(\epsilon)],$
i.e., $f_{4,3}^\beta(\cdot,\omega)$ satisfies \eqref{eq:vsol} at $x = y_0$ with $\lambda = \beta$ and $f_{3,4}^\beta(\cdot,\omega)$ satisfies \eqref{eq:vsag} at $x = y_0$ with $\lambda = \beta - \epsilon$. We deduce that $f_{4,3}^\beta(\cdot,\omega)$ is a subsolution of \eqref{eq:aux} with $\lambda = \beta$ and $f_{3,4}^\beta(\cdot,\omega)$ is a supersolution of \eqref{eq:aux} with $\lambda = \beta - \epsilon$. Moreover,
\[ \lim_{x\to+\infty} \frac{f_{4,3}^\beta(x,\omega)}{x} = \lim_{x\to-\infty} \frac{f_{3,4}^\beta(x,\omega)}{x} = \theta_3(\beta) < 0 < \theta_4(\beta) = \lim_{x\to-\infty} \frac{f_{4,3}^\beta(x,\omega)}{x} = \lim_{x\to+\infty} \frac{f_{3,4}^\beta(x,\omega)}{x} \]
for $\mathbb{P}$-a.e.\ $\omega$. Therefore,
\[ \Ham^U(\theta) \le \beta \ \ \text{and}\ \  \Ham^L(\theta) \ge \beta - \epsilon\ \ \text{for every}\ \theta\in(\theta_3(\beta),\theta_4(\beta)) \]
by Proposition \ref{prop:onesided}. Since $\epsilon\in(0,\beta)$ is arbitrary, we conclude that
\[ \Ham(\theta) = \Ham^L(\theta) = \Ham^U(\theta) = \beta\ \ \text{for every}\ \theta\in(\theta_3(\beta),\theta_4(\beta)). \]

\subsection{Flat piece at height $m + \beta$}

For every $\epsilon\in(0,\beta)$ and $\mathbb{P}$-a.e.\ $\omega$, there exists a $y_0 = y_0(\omega)\in\mathbb{R}$ such that $\beta V(y_0,\omega) = \beta - \epsilon$. Construct $f_{2,1}^{m + \beta}(\cdot,\omega),f_{1,2}^{m + \beta}(\cdot,\omega)\in\Lip(\mathbb{R})$ by setting $f_{2,1}^{m + \beta}(0,\omega) = f_{1,2}^{m + \beta}(0,\omega) = 0$, 
\[ (f_{2,1}^{m + \beta})'(x,\omega) = \begin{cases} (f_2^{m + \beta})'(x,\omega)&\text{if}\ x < y_0,\\ (f_1^{m + \beta})'(x,\omega)&\text{if}\ x > y_0\end{cases} \quad\text{and}\quad (f_{1,2}^{m + \beta})'(x,\omega) = \begin{cases} (f_1^{m + \beta})'(x,\omega)&\text{if}\ x < y_0,\\ (f_2^{m + \beta})'(x,\omega)&\text{if}\ x > y_0\end{cases} \]
with the notation in \eqref{eq:ezgi1} and \eqref{eq:ezgi2}. Observe that, at every $x\in\mathbb{R}\setminus\{y_0\}$, the functions $f_{2,1}^{m + \beta}(\cdot,\omega)$ and $f_{1,2}^{m + \beta}(\cdot,\omega)$ satisfy both \eqref{eq:vsol} and \eqref{eq:vsag} with $\lambda = m + \beta$. In addition,
\begin{align*}
(f_{2,1}^{m + \beta})_+'(y_0,\omega) = (f_{1,2}^{m + \beta})_-'(y_0,\omega) = (f_1^{m + \beta})'(y_0,\omega) &= G_1^{-1}(m + \beta - \beta V(y_0,\omega))\\
&= G_1^{-1}(m + \epsilon) < p_m\quad\text{and}\\
(f_{2,1}^{m + \beta})_-'(y_0,\omega) = (f_{1,2}^{m + \beta})_+'(y_0,\omega) = (f_2^{m + \beta})'(y_0,\omega) &= G_2^{-1}(m + \beta - \beta V(y_0,\omega))\\
&= G_2^{-1}(m + \epsilon) > p_m.
\end{align*}
It is clear from Condition \ref{cond:pm} that
\[ m + \beta - \epsilon \le G(p) + \beta V(y_0,\omega) = G(p) + \beta - \epsilon \le m + \beta \]
for every $p\in [G_1^{-1}(m + \epsilon),G_2^{-1}(m + \epsilon)],$
i.e., $f_{2,1}^{m + \beta}(\cdot,\omega)$ satisfies \eqref{eq:vsol} at $x = y_0$ with $\lambda = m + \beta$ and $f_{1,2}^{m + \beta}(\cdot,\omega)$ satisfies \eqref{eq:vsag} at $x = y_0$ with $\lambda = m + \beta - \epsilon$. We deduce that $f_{2,1}^{m + \beta}(\cdot,\omega)$ is a subsolution of \eqref{eq:aux} with  $\lambda = m + \beta$ and $f_{1,2}^{m + \beta}(\cdot,\omega)$ is a supersolution of \eqref{eq:aux} with $\lambda = m + \beta - \epsilon$. Moreover,
\begin{align*}
	\lim_{x\to+\infty} \frac{f_{2,1}^{m + \beta}(x,\omega)}{x} &= \lim_{x\to-\infty} \frac{f_{1,2}^{m + \beta}(x,\omega)}{x} = \theta_1(m + \beta) < p_m\\
	&< \theta_2(m + \beta) = \lim_{x\to-\infty} \frac{f_{2,1}^{m + \beta}(x,\omega)}{x} = \lim_{x\to+\infty} \frac{f_{1,2}^{m + \beta}(x,\omega)}{x}
\end{align*}
for $\mathbb{P}$-a.e.\ $\omega$. Therefore,
\[ \Ham^U(\theta) \le m + \beta \ \ \text{and}\ \  \Ham^L(\theta) \ge m + \beta - \epsilon\ \ \text{for every}\ \theta\in(\theta_1(m + \beta),\theta_2(m + \beta)) \]
by Proposition \ref{prop:onesided}. Since $\epsilon\in(0,\beta)$ is arbitrary, we conclude that
\[ \Ham(\theta) = \Ham^L(\theta) = \Ham^U(\theta) = m + \beta\ \ \text{for every}\ \theta\in(\theta_1(m + \beta),\theta_2(m + \beta)). \]

\subsection{Flat piece at height $M$}

For every $\epsilon\in(0,\beta)$ and $\mathbb{P}$-a.e.\ $\omega$, there exists an $x_0 = x_0(\omega)\in\mathbb{R}$ such that $\beta V(x_0,\omega) = \epsilon$. Construct $f_{3,2}^{M}(\cdot,\omega),f_{2,3}^{M}(\cdot,\omega)\in\Lip(\mathbb{R})$ by setting $f_{3,2}^{M}(0,\omega) = f_{2,3}^{M}(0,\omega) = 0$,
\[ (f_{3,2}^{M})'(x,\omega) = \begin{cases} (f_3^{M})'(x,\omega)&\text{if}\ x < x_0,\\ (f_2^{M})'(x,\omega)&\text{if}\ x > x_0\end{cases} \quad\text{and}\quad (f_{2,3}^{M})'(x,\omega) = \begin{cases} (f_2^{M})'(x,\omega)&\text{if}\ x < x_0,\\ (f_3^{M})'(x,\omega)&\text{if}\ x > x_0\end{cases} \]
with the notation in \eqref{eq:ezgi2} and \eqref{eq:ezgi3}. Observe that, at every $x\in\mathbb{R}\setminus\{x_0\}$, the functions $f_{3,2}^{M}(\cdot,\omega)$ and $f_{2,3}^{M}(\cdot,\omega)$ satisfy both \eqref{eq:vsol} and \eqref{eq:vsag} with $\lambda = M$. In addition,
\begin{align*}
(f_{3,2}^{M})_+'(x_0,\omega) = (f_{2,3}^{M})_-'(x_0,\omega) = (f_2^{M})'(x_0,\omega) &= G_2^{-1}(M - \beta V(x_0,\omega)) = G_2^{-1}(M - \epsilon) < p_M\quad\text{and}\\
(f_{3,2}^{M})_-'(x_0,\omega) = (f_{2,3}^{M})_+'(x_0,\omega) = (f_3^{M})'(x_0,\omega) &= G_3^{-1}(M - \beta V(x_0,\omega)) = G_3^{-1}(M - \epsilon) > p_M.
\end{align*}
It is clear from Condition \ref{cond:pm} that
\[ M \le G(p) + \beta V(x_0,\omega) = G(p) + \epsilon \le M + \epsilon \]
for every $p\in [G_2^{-1}(M - \epsilon),G_3^{-1}(M - \epsilon)],$
i.e., $f_{3,2}^{M}(\cdot,\omega)$ satisfies \eqref{eq:vsol} at $x = x_0$ with $\lambda = M + \epsilon$ and $f_{2,3}^{M}(\cdot,\omega)$ satisfies \eqref{eq:vsag} at $x = x_0$ with $\lambda = M$. We deduce that $f_{3,2}^{M}(\cdot,\omega)$ is a subsolution of \eqref{eq:aux} with $\lambda = M + \epsilon$ and $f_{2,3}^{M}(\cdot,\omega)$ is a supersolution of \eqref{eq:aux} with $\lambda = M$.
Moreover,
\[ \lim_{x\to+\infty} \frac{f_{3,2}^M(x,\omega)}{x} = \lim_{x\to-\infty} \frac{f_{2,3}^M(x,\omega)}{x} = \theta_2(M) < p_M < \theta_3(M) = \lim_{x\to-\infty} \frac{f_{3,2}^M(x,\omega)}{x} = \lim_{x\to+\infty} \frac{f_{2,3}^M(x,\omega)}{x} \]
for $\mathbb{P}$-a.e.\ $\omega$. Therefore,
\[ \Ham^U(\theta) \le M + \epsilon \ \ \text{and}\ \  \Ham^L(\theta) \ge M\ \ \text{for every}\ \theta\in(\theta_2(M),\theta_3(M)) \]
by Proposition \ref{prop:onesided}. Since $\epsilon\in(0,\beta)$ is arbitrary, we conclude that
\[ \Ham(\theta) = \Ham^L(\theta) = \Ham^U(\theta) = M\ \ \text{for every}\ \theta\in(\theta_2(M),\theta_3(M)). \]

\section{Case II: Medium potential ($\beta < M \le m + \beta$)}\label{sec:medium}

\subsection{Strictly decreasing piece involving $G_1$}

We proved in Subsection \ref{sub:dec1} that, in Cases I, II and III, for every $\lambda \in [m + \beta,+\infty)$,
\[ \Ham(\theta_1(\lambda)) = \Ham^L(\theta_1(\lambda)) = \Ham^U(\theta_1(\lambda)) = \lambda \]
with the definition in \eqref{eq:ddnono}. We added that $\Ham$ is strictly decreasing on $(-\infty,\theta_1(m + \beta)]$.

\subsection{Nonincreasing piece involving both $G_1$ and $G_3$ (excluding the easy subcase)}\label{sub:ister}

In this subsection, we will simultaneously consider Cases II and III of Theorem \ref{thm:velinim} (except their easy subcases which are deferred to Subsections \ref{sub:medeasy} and \ref{sub:streasy}).

Assume that $\max\{\beta,M\} < m + \beta$. For every $\lambda\in[\beta,+\infty)\cap(M,m + \beta)$, define two bi-infinite sequences $(\underline x_i)_{i\in\mathbb{Z}} = (\underline x_i(\lambda,\omega))_{i\in\mathbb{Z}}$ and $(\overline y_i)_{i\in\mathbb{Z}} = (\overline y_i(\lambda,\omega))_{i\in\mathbb{Z}}$ by the coupled recursion
\begin{equation}\label{eq:noanchor}
\begin{aligned}
\underline x_i(\lambda,\omega) &= \inf\{ x\ge \overline y_{i-1}(\lambda,\omega):\, \lambda - \beta V(x,\omega) \ge M \},\\
\overline y_i(\lambda,\omega) &= \inf\{ x\ge \underline x_i(\lambda,\omega):\, \lambda - \beta V(x,\omega) < m \}
\end{aligned}
\end{equation}
and the following anchoring condition:
\begin{equation}\label{eq:cip1}
\underline x_{-1}(\lambda,\omega) \le 0 < \underline x_0(\lambda,\omega).
\end{equation}
Similarly, define $(\overline x_i)_{i\in\mathbb{Z}} = (\overline x_i(\lambda,\omega))_{i\in\mathbb{Z}}$ and $(\underline y_i)_{i\in\mathbb{Z}} = (\underline y_i(\lambda,\omega))_{i\in\mathbb{Z}}$ by the coupled recursion
\begin{equation}\label{eq:cipayok}
\begin{aligned}
\overline x_i(\lambda,\omega) &= \inf\{ x\ge \underline y_{i-1}(\lambda,\omega):\, \lambda - \beta V(x,\omega) > M \},\\
\underline y_i(\lambda,\omega) &= \inf\{ x\ge \overline x_i(\lambda,\omega):\, \lambda - \beta V(x,\omega) \le m \}
\end{aligned}
\end{equation}
and the following anchoring condition:
\begin{equation}\label{eq:cip2}
\overline x_{-1}(\lambda,\omega) < 0 \le \overline x_0(\lambda,\omega).
\end{equation}
Note that the set
\begin{equation}\label{eq:sopra}
	\Omega_0 := \!\!\!\!\!\! \bigcap_{\lambda\in[\beta,+\infty)\cap(M,m + \beta)} \!\!\!\!\!\! \{ \omega\in\Omega: \underline{x}_i(\lambda,\omega),\overline{y}_i(\lambda,\omega),\overline{x}_i(\lambda,\omega),\underline{y}_i(\lambda,\omega)\in(-\infty,+\infty)\ \text{for every}\ i\in\mathbb{Z} \}
\end{equation}
satisfies $\mathbb{P}(\Omega_0) = 1$ by our ergodicity assumption. For every $\lambda\in[\beta,+\infty)\cap(M,m + \beta)$ and $\omega\in\Omega_0$, it is easy to see that
\begin{equation}\label{eq:ofkeyonet}
\begin{aligned}
&\lambda - \beta V(x,\omega) \ge m \quad\text{if}\ x\in [\underline x_i,\overline y_i)\quad\text{and}\\
&\lambda - \beta V(x,\omega) < M \quad\text{if}\ x\in [\overline y_i,\underline x_{i+1}).
\end{aligned}
\end{equation}
In particular, whenever $\lambda - \beta V(x,\omega) = M$, we know that $x\in[\underline x_i,\overline y_i)$ for some $i\in\mathbb{Z}$. Consequently, $(\underline x_i)_{i\in\mathbb{Z}}$ and $(\overline y_i)_{i\in\mathbb{Z}}$ are well-defined by \eqref{eq:noanchor}--\eqref{eq:cip1}. Similarly,
\begin{equation}\label{eq:ayakbas}
\begin{aligned}
&\lambda - \beta V(x,\omega) > m \quad\text{if}\ x\in [\overline x_i,\underline y_i)\ \ \text{and}\\
&\lambda - \beta V(x,\omega) \le M \quad\text{if}\ x\in [\underline y_i,\overline x_{i+1}).
\end{aligned}
\end{equation}
In particular, whenever $\lambda - \beta V(x,\omega) = m$, we know that $x\in[\underline y_i,\overline x_{i+1})$ for some $i\in\mathbb{Z}$.
Consequently, $(\overline x_i)_{i\in\mathbb{Z}}$ and $(\underline y_i)_{i\in\mathbb{Z}}$ are well-defined by \eqref{eq:cipayok}--\eqref{eq:cip2}. See Figure \ref{fig:V}.

\begin{figure}
	\includegraphics{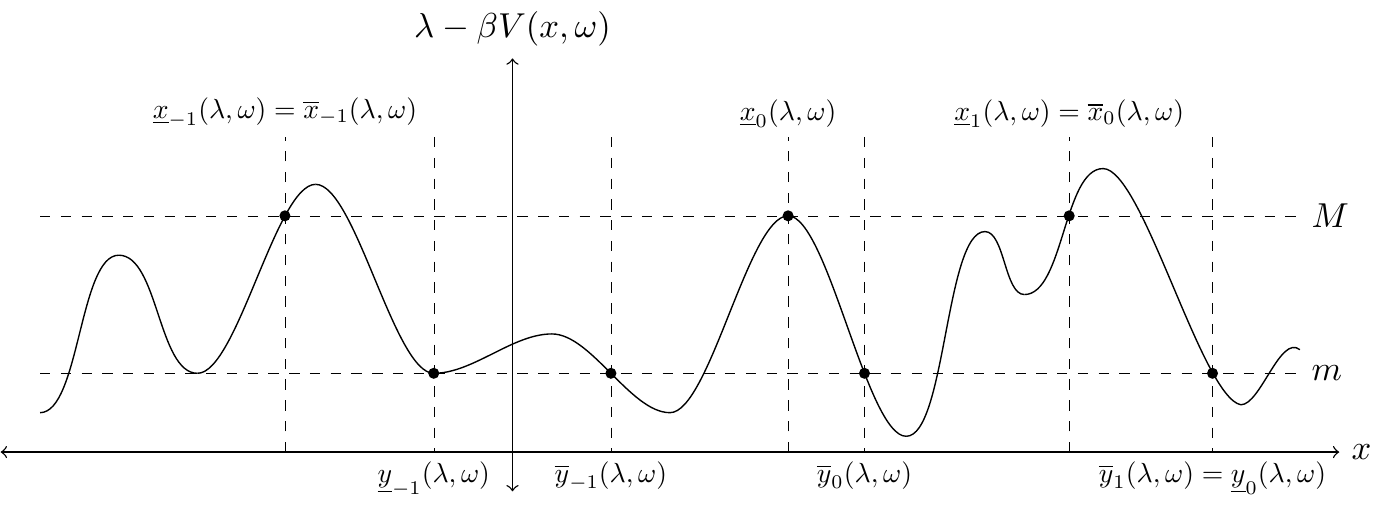}
	\caption{A realization of the two bi-infinite sequences defined by \eqref{eq:noanchor}--\eqref{eq:cip2}.}
	\label{fig:V}
\end{figure}

For every $\lambda\in[\beta,+\infty)\cap(M,m + \beta)$ and $\omega\in\Omega_0$, construct $\underline f_{1,3}^\lambda(\cdot,\omega),\overline f_{1,3}^\lambda(\cdot,\omega)\in\Lip(\mathbb{R})$ by setting $\underline f_{1,3}^\lambda(0,\omega) = \overline f_{1,3}^\lambda(0,\omega) = 0$,
\begin{equation}\label{eq:bizdik}
(\underline f_{1,3}^\lambda)'(x,\omega) = \begin{cases} G_1^{-1}(\lambda - \beta V(x,\omega)) \in (-\infty,p_m]&\text{if}\ x\in(\underline x_i,\overline y_i),\\ G_3^{-1}(\lambda - \beta V(x,\omega)) \in (p_M,0]&\text{if}\ x\in(\overline y_i,\underline x_{i+1})\end{cases}
\end{equation}
and
\begin{equation}\label{eq:anbizdik}
(\overline f_{1,3}^\lambda)'(x,\omega) = \begin{cases} G_1^{-1}(\lambda - \beta V(x,\omega)) \in  (-\infty,p_m)&\text{if}\ x\in(\overline x_i,\underline y_i),\\ G_3^{-1}(\lambda - \beta V(x,\omega)) \in [p_M,0]&\text{if}\ x\in(\underline y_i,\overline x_{i+1}).\end{cases}
\end{equation}
Observe that, at every $x\in\mathbb{R}\setminus\{\ldots,\overline y_{-1},\underline x_0,\overline y_0,\underline x_1,\ldots\}$, the function $\underline f_{1,3}^\lambda(\cdot,\omega)$ satisfies both \eqref{eq:vsol} and \eqref{eq:vsag}. In addition,
\begin{align*}
	&(\underline f_{1,3}^\lambda)_+'(\underline x_i,\omega) = G_1^{-1}(M) < p_m < p_M = G_3^{-1}(M) = (\underline f_{1,3}^\lambda)_-'(\underline x_i,\omega)\quad\text{and}\\
	&(\underline f_{1,3}^\lambda)_-'(\overline y_i,\omega) = G_1^{-1}(m) = p_m < p_M < G_3^{-1}(m) = (\underline f_{1,3}^\lambda)_+'(\overline y_i,\omega).
\end{align*}
It is clear from Condition \ref{cond:pm} that
\[ G(p) + \beta V(\underline x_i,\omega) \le M + \beta V(\underline x_i,\omega) = \lambda \]
for every $p\in [G_1^{-1}(M),G_3^{-1}(M)]$, i.e., $\underline f_{1,3}^\lambda(\cdot,\omega)$ satisfies \eqref{eq:vsol} at $x = \underline x_i$,
and
\[ G(p) + \beta V(\overline y_i,\omega) \ge m + \beta V(\overline y_i,\omega) = \lambda \]
for every $p\in [ G_1^{-1}(m),G_3^{-1}(m)]$, i.e., $\underline f_{1,3}^\lambda(\cdot,\omega)$ satisfies \eqref{eq:vsag} at $x = \overline y_i$. We deduce that $\underline f_{1,3}^\lambda(\cdot,\omega)$ is a solution of \eqref{eq:aux}. Similarly for $\overline f_{1,3}^\lambda(\cdot,\omega)$.

It follows easily from \eqref{eq:stationary}, \eqref{eq:noanchor}, \eqref{eq:cipayok}, \eqref{eq:bizdik} and \eqref{eq:anbizdik} that
\[ (\underline f_{1,3}^\lambda)_\pm'(x,\omega) = (\underline f_{1,3}^\lambda)_\pm'(0,\tau_x\omega)\quad\text{and}\quad (\overline f_{1,3}^\lambda)_\pm'(x,\omega) = (\overline f_{1,3}^\lambda)_\pm'(0,\tau_x\omega) \]
for every $(x,\omega)\in\mathbb{R}\times\Omega_0$. By the ergodic theorem,
\begin{align}
	\underline\theta_{1,3}(\lambda) &:= \lim_{x\to\pm\infty}\frac{\underline f_{1,3}^\lambda(x,\omega)}{x} = \lim_{x\to\pm\infty}\frac1{x}\int_0^x(\underline f_{1,3}^\lambda)'(y,\omega)dy = \mathbb{E}\left[(\underline f_{1,3}^\lambda)'(0,\omega)\right]\quad\text{and}\label{eq:hurma1}\\
	\overline\theta_{1,3}(\lambda) &:= \lim_{x\to\pm\infty}\frac{\overline f_{1,3}^\lambda(x,\omega)}{x} = \lim_{x\to\pm\infty}\frac1{x}\int_0^x(\overline f_{1,3}^\lambda)'(y,\omega)dy = \mathbb{E}\left[(\overline f_{1,3}^\lambda)'(0,\omega)\right]\label{eq:hurma2}
\end{align}
for $\mathbb{P}$-a.e.\ $\omega$. Therefore,
\begin{equation}\label{eq:sflow}
\begin{aligned}
	\Ham(\underline\theta_{1,3}(\lambda)) &= \Ham^L(\underline\theta_{1,3}(\lambda)) = \Ham^U(\underline\theta_{1,3}(\lambda)) = \lambda\quad\text{and}\\
	\Ham(\overline\theta_{1,3}(\lambda)) &= \Ham^L(\overline\theta_{1,3}(\lambda)) = \Ham^U(\overline\theta_{1,3}(\lambda)) = \lambda
\end{aligned}
\end{equation}
by Proposition \ref{prop:demtut}.

Examining the definitions of $(\underline x_i(\lambda,\omega))_{i\in\mathbb{Z}}$, $(\overline x_i(\lambda,\omega))_{i\in\mathbb{Z}}$, $(\underline y_i(\lambda,\omega))_{i\in\mathbb{Z}}$ and $(\overline y_i(\lambda,\omega))_{i\in\mathbb{Z}}$ reveals that they are not interlaced (i.e., $\underline x_i(\lambda,\omega) \le \overline x_i(\lambda,\omega) \le \underline y_i(\lambda,\omega) \le \overline y_i(\lambda,\omega) \le \underline x_{i+1}(\lambda,\omega)$ does not hold) in general. (See Figure \ref{fig:V} for an example.) However, they satisfy the following weaker property.

\begin{lemma}\label{lem:interpm}
	For every $\lambda\in [\beta,+\infty)\cap(M,m + \beta)$, $i\in\mathbb{Z}$ and $\omega\in\Omega_0$, there exist $j,k\in\mathbb{Z}$ such that
	\[ [\overline x_i(\lambda,\omega),\underline y_i(\lambda,\omega)) \subset [\underline x_j(\lambda,\omega),\overline y_j(\lambda,\omega)) \ \text{and}\ \  [\overline y_i(\lambda,\omega),\underline x_{i+1}(\lambda,\omega)) \subset [\underline y_k(\lambda,\omega),\overline x_{k+1}(\lambda,\omega)). \]
\end{lemma}

\begin{proof}
	For every $\lambda\in [\beta,+\infty)\cap(M,m + \beta)$, $i\in\mathbb{Z}$ and $\omega\in\Omega_0$,
	\[ \lambda - \beta V(\overline x_i(\lambda,\omega),\omega) = M. \]
	Therefore, $\overline x_i(\lambda,\omega)\in [\underline x_j(\lambda,\omega),\overline y_j(\lambda,\omega))$ for some $j\in\mathbb{Z}$. Since
	\[ \lambda - \beta V(\overline y_j(\lambda,\omega),\omega) = m, \]
	we deduce that
	\[ \underline y_i(\lambda,\omega) = \inf\{ x\ge \overline x_i(\lambda,\omega):\, \lambda - \beta V(x,\omega) \le m \} \le \overline y_j(\lambda,\omega). \]
	This proves the first set inclusion. The second set inclusion is proved similarly.
\end{proof}

\begin{lemma}\label{lem:basic}
	$\underline \theta_{1,3}(\lambda) \le \overline \theta_{1,3}(\lambda)$ for every $\lambda\in [\beta,+\infty)\cap(M,m + \beta)$.
\end{lemma}

\begin{proof}
	For every $\lambda\in [\beta,+\infty)\cap(M,m + \beta)$ and $\omega\in\Omega_0$, on any interval of the form $(\underline x_i(\lambda,\omega),\overline y_i(\lambda,\omega))$,
	\[ (\underline f_{1,3}^\lambda)'(x,\omega) = G_1^{-1}(\lambda - \beta V(x,\omega)) \le (\overline f_{1,3}^\lambda)'(x,\omega). \]
	Moreover, on any interval of the form $(\overline y_i(\lambda,\omega),\underline x_{i+1}(\lambda,\omega))$,
	\[ (\underline f_{1,3}^\lambda)'(x,\omega) = G_3^{-1}(\lambda - \beta V(x,\omega)) = (\overline f_{1,3}^\lambda)'(x,\omega). \]
	The last equality follows from Lemma \ref{lem:interpm}. In short, $(\underline f_{1,3}^\lambda)'(x,\omega) \le (\overline f_{1,3}^\lambda)'(x,\omega)$ for a.e.\ $x\in\mathbb{R}$ (with respect to Lebesgue measure). The desired inequality is now evident from \eqref{eq:hurma1} and \eqref{eq:hurma2}.
\end{proof}

\begin{lemma}\label{lem:ciftkas}
	For every $\lambda_1,\lambda_2\in [\beta,+\infty)\cap(M,m + \beta)$, if $\lambda_1<\lambda_2$, then
	\[ \underline \theta_{1,3}(\lambda_1) > \overline \theta_{1,3}(\lambda_2).\]
\end{lemma}

\begin{proof}
	For every $\omega\in\Omega_0$, on any interval of the form $(\overline y_i(\lambda_1,\omega),\underline x_{i+1}(\lambda_1,\omega))$,
	\[ (\underline f_{1,3}^{\lambda_1})'(x,\omega) = G_3^{-1}(\lambda_1 - \beta V(x,\omega)) > G_3^{-1}(\lambda_2 - \beta V(x,\omega)) \ge (\overline f_{1,3}^{\lambda_2})_\pm'(x,\omega). \]
	Similarly, on any interval of the form $(\overline x_j(\lambda_2,\omega),\underline y_j(\lambda_2,\omega))$,
	\[ (\underline f_{1,3}^{\lambda_1})_\pm'(x,\omega) \ge G_1^{-1}(\lambda_1 - \beta V(x,\omega)) > G_1^{-1}(\lambda_2 - \beta V(x,\omega)) = (\overline f_{1,3}^{\lambda_2})'(x,\omega). \]
	Since the closures of such invervals cover $\mathbb{R}$ by Lemma \ref{lem:interlace} below, the desired inequality is evident from \eqref{eq:hurma1} and \eqref{eq:hurma2}.
\end{proof}

\begin{lemma}\label{lem:interlace}
	For every $\lambda_1,\lambda_2\in [\beta,+\infty)\cap(M,m + \beta)$ and $\omega\in\Omega_0$, if $\lambda_1 < \lambda_2$, then
	\[ \bigcup_{i\in\mathbb{Z}}[\overline y_i(\lambda_1,\omega),\underline x_{i+1}(\lambda_1,\omega)) \cup \bigcup_{j\in\mathbb{Z}}[\overline x_j(\lambda_2,\omega),\underline y_j(\lambda_2,\omega)) = \mathbb{R}. \]
\end{lemma}

\begin{proof}
	For every $j\in\mathbb{Z}$ and $\omega\in\Omega_0$,
	\[ \lambda_1 - \beta V(\underline y_j(\lambda_2,\omega),\omega) = (\lambda_1 - \lambda_2) + \lambda_2 - \beta V(\underline y_j(\lambda_2,\omega),\omega) = (\lambda_1 - \lambda_2) + m < m. \]
	Therefore, $\underline y_j(\lambda_2,\omega)\in [\overline y_i(\lambda_1,\omega),\underline x_{i+1}(\lambda_1,\omega))$ for some $i\in\mathbb{Z}$. Since
	\[ \lambda_2 - \beta V(\underline x_{i+1}(\lambda_1,\omega),\omega) = (\lambda_2 - \lambda_1) + \lambda_1 - \beta V(\underline x_{i+1}(\lambda_1,\omega),\omega) = (\lambda_2 - \lambda_1) + M > M, \]
	we deduce that
	\[ \overline x_{j+1}(\lambda_2,\omega) = \inf\{ x\ge \underline y_j(\lambda_2,\omega):\, \lambda_2 - \beta V(x,\omega) > M \} < \underline x_{i+1}(\lambda_1,\omega). \]
	In short, for every $j\in\mathbb{Z}$ and $\omega\in\Omega_0$, there exists an $i\in\mathbb{Z}$ such that
	\[ [\underline y_j(\lambda_2,\omega),\overline x_{j+1}(\lambda_2,\omega)) \subset [\overline y_i(\lambda_1,\omega),\underline x_{i+1}(\lambda_1,\omega)), \]
	which readily implies the desired result.
\end{proof}

\begin{lemma}\label{lem:solsag}
	The maps $\lambda\mapsto\underline \theta_{1,3}(\lambda)$ and $\lambda\mapsto\overline \theta_{1,3}(\lambda)$ are right-continuous and left-continuous, respectively, on $[\beta,+\infty)\cap(M,m + \beta)$.
\end{lemma}

\begin{proof}
	It is easy to see from \eqref{eq:noanchor}--\eqref{eq:cip1} that $\lambda\mapsto\underline x_i(\lambda,\omega)$ and $\lambda\mapsto\overline y_i(\lambda,\omega)$ are right-continuous for every $i\in\mathbb{Z}$ and $\omega\in\Omega_0$. 
	Hence, the right continuity of $\lambda\mapsto\underline \theta_{1,3}(\lambda)$ follows from \eqref{eq:bizdik} and \eqref{eq:hurma1}. There is a similar argument for the left continuity of $\lambda\mapsto\overline \theta_{1,3}(\lambda)$.
\end{proof}

Define the quantities
\begin{equation}\label{eq:chll}
\begin{aligned}
	\theta_{1,3}(\max\{\beta,M\}) &:= \underline \theta_{1,3}(\max\{\beta,M\}+) = \overline \theta_{1,3}(\max\{\beta,M\}+)\quad\text{and}\\
	\theta_{1,3}(m + \beta) &:= \underline \theta_{1,3}(m + \beta-) = \overline \theta_{1,3}(m + \beta-).
\end{aligned}
\end{equation}
Note that the second equality in each line follows from Lemmas \ref{lem:basic} and \ref{lem:ciftkas}.

\begin{lemma}\label{lem:carmsop}
	The collection
	\[ \mathcal{C} = \{ [\underline \theta_{1,3}(\lambda),\overline \theta_{1,3}(\lambda)]:\, \lambda\in (\max\{\beta,M\},m + \beta) \} \]
	is a partition of $( \theta_{1,3}(m + \beta), \theta_{1,3}(\max\{\beta,M\}) )$. In other words, for every $\theta$ in this open interval, there is a unique $\Lambda(\theta)\in (\max\{\beta,M\},m + \beta)$ such that $\theta\in [\underline \theta_{1,3}(\Lambda(\theta)),\overline \theta_{1,3}(\Lambda(\theta))]$.
\end{lemma}

\begin{proof}
	The intervals in the collection $\mathcal{C}$ are disjoint by Lemma \ref{lem:ciftkas}. It remains to show that their union is $( \theta_{1,3}(m + \beta), \theta_{1,3}(\max\{\beta,M\}) )$. To this end, fix an arbitrary $\theta$ in this open interval and let 
	\begin{equation}\label{eq:nowaste}
		\Lambda(\theta) = \inf\{\lambda\in (\max\{\beta,M\},m + \beta):\,\underline \theta_{1,3}(\lambda)\le \theta \}.
	\end{equation}
	It follows from  Lemmas \ref{lem:basic}, \ref{lem:ciftkas} and \ref{lem:solsag} that
	\[ \underline \theta_{1,3}(\Lambda(\theta)) = \underline \theta_{1,3}(\Lambda(\theta)+) \le \theta \le \underline \theta_{1,3}(\Lambda(\theta)-) = \overline \theta_{1,3}(\Lambda(\theta)-) = \overline \theta_{1,3}(\Lambda(\theta)).\qedhere \]
\end{proof}

\begin{theorem}\label{thm:gennet}
	Under the assumptions in Theorem \ref{thm:velinim}, if $\max\{\beta,M\} < m + \beta$, then
	\[ \Ham(\theta) = \Ham^L(\theta) = \Ham^U(\theta) = \Lambda(\theta)\ \ \text{for every}\ \theta\in ( \theta_{1,3}(m + \beta), \theta_{1,3}(\max\{\beta,M\}) ) \]
	with the definitions in \eqref{eq:chll}--\eqref{eq:nowaste}. Consequently, $\Ham$ is nonincreasing on this interval.
\end{theorem}

\begin{proof}
	Fix $\theta\in ( \theta_{1,3}(m + \beta), \theta_{1,3}(\max\{\beta,M\}) )$.	Note that $\theta\in [\underline \theta_{1,3}(\Lambda(\theta)),\overline \theta_{1,3}(\Lambda(\theta))]$ by Lemma \ref{lem:carmsop}. If $\theta\in \{\underline \theta_{1,3}(\Lambda(\theta)),\overline \theta_{1,3}(\Lambda(\theta))\}$, then the desired equalities follow from \eqref{eq:sflow}.
	
	It remains to consider the case $\theta\in(\underline \theta_{1,3}(\Lambda(\theta)),\overline \theta_{1,3}(\Lambda(\theta)))\ne\emptyset$. (We will characterize the nonemptiness of this interval in the proof of Theorem \ref{thm:intermed}. See Section \ref{sec:interflat}.) For every $\omega\in\Omega_0$, construct $f_{1,3,d}^{\Lambda(\theta)}(\cdot,\omega),f_{1,3,u}^{\Lambda(\theta)}(\cdot,\omega)\in\Lip(\mathbb{R})$ by setting $f_{1,3,d}^{\Lambda(\theta)}(0,\omega) = f_{1,3,u}^{\Lambda(\theta)}(0,\omega) = 0$,
	\[ (f_{1,3,d}^{\Lambda(\theta)})'(x,\omega) = \begin{cases} (\overline f_{1,3}^{\Lambda(\theta)})'(x,\omega)&\text{if}\ x < \underline x_0(\Lambda(\theta),\omega),\\ (\underline f_{1,3}^{\Lambda(\theta)})'(x,\omega)&\text{if}\ x > \underline x_0(\Lambda(\theta),\omega)\end{cases} \]
	and
	\[ (f_{1,3,u}^{\Lambda(\theta)})'(x,\omega) = \begin{cases} (\underline f_{1,3}^{\Lambda(\theta)})'(x,\omega)&\text{if}\ x < \underline y_0(\Lambda(\theta),\omega),\\ (\overline f_{1,3}^{\Lambda(\theta)})'(x,\omega)&\text{if}\ x > \underline y_0(\Lambda(\theta),\omega).\end{cases} \]
	Repeating the analysis we carried out earlier for $\underline f_{1,3}^\lambda(\cdot,\omega)$, we deduce that $f_{1,3,d}^{\Lambda(\theta)}(\cdot,\omega),f_{1,3,u}^{\Lambda(\theta)}(\cdot,\omega)$ are solutions of \eqref{eq:aux} with $\lambda = \Lambda(\theta)$. Moreover,
	\begin{align*}
		\lim_{x\to+\infty} \frac{f_{1,3,d}^{\Lambda(\theta)}(x,\omega)}{x} &= \lim_{x\to-\infty} \frac{f_{1,3,u}^{\Lambda(\theta)}(x,\omega)}{x} = \underline \theta_{1,3}(\Lambda(\theta))\\
		&< \overline \theta_{1,3}(\Lambda(\theta)) = \lim_{x\to-\infty} \frac{f_{1,3,d}^{\Lambda(\theta)}(x,\omega)}{x} = \lim_{x\to+\infty} \frac{f_{1,3,u}^{\Lambda(\theta)}(x,\omega)}{x}
	\end{align*}
	for $\mathbb{P}$-a.e.\ $\omega$ by \eqref{eq:hurma1}--\eqref{eq:hurma2}. Therefore,
	\begin{equation}\label{eq:nadad}
		\Ham^U(\theta) \le \Lambda(\theta) \ \ \text{and}\ \  \Ham^L(\theta) \ge \Lambda(\theta)\ \ \text{for every}\ \theta\in (\underline \theta_{1,3}(\Lambda(\theta)),\overline \theta_{1,3}(\Lambda(\theta)))
	\end{equation}
	by Proposition \ref{prop:onesided}.
\end{proof}

\subsection{Strictly decreasing piece involving $G_3$}

We proved in Subsection \ref{sub:dec3} that, in Cases I and II, for every $\lambda \in [\beta,M]$,
\[ \Ham(\theta_3(\lambda)) = \Ham^L(\theta_3(\lambda)) = \Ham^U(\theta_3(\lambda)) = \lambda \]
with the definition in \eqref{eq:ddnono}. We added that $\Ham$ is strictly decreasing on $[\theta_3(M),\theta_3(\beta)]$.

\subsection{Strictly increasing piece involving $G_4$}

We proved in Subsection \ref{sub:inc4} that, in Cases I, II and III, for every $\lambda \in [\beta,+\infty)$,
\[ \Ham(\theta_4(\lambda)) = \Ham^L(\theta_4(\lambda)) = \Ham^U(\theta_4(\lambda)) = \lambda \]
with the definition in \eqref{eq:ddnono}. We added that $\Ham$ is strictly increasing on $[\theta_4(\beta),+\infty)$.

\subsection{Flat piece at height $\beta$}

We proved in Subsection \ref{sub:flatbeta} that, in Cases I and II,
\[ \Ham(\theta) = \Ham^L(\theta) = \Ham^U(\theta) = \beta\ \ \text{for every}\ \theta\in(\theta_3(\beta),\theta_4(\beta)). \]

\subsection{Flat piece at height $M = m + \beta$ (in the easy subcase)}\label{sub:medeasy}

Assume that $\beta < M = m + \beta$. For every $\epsilon\in(0,\beta)$ and $\mathbb{P}$-a.e.\ $\omega$, there exist $x_0 = x_0(\omega)\in\mathbb{R}$ and $y_0 = y_0(\omega)\in\mathbb{R}$ such that $\beta V(x_0,\omega) = \epsilon$ and $\beta V(y_0,\omega) = \beta - \epsilon$. Construct $f_{3,1}^{M}(\cdot,\omega),f_{1,3}^{M}(\cdot,\omega)\in\Lip(\mathbb{R})$ by setting $f_{3,1}^{M}(0,\omega) = f_{1,3}^{M}(0,\omega) = 0$,
\[ (f_{3,1}^{M})'(x,\omega) = \begin{cases} (f_3^{M})'(x,\omega)&\text{if}\ x < x_0,\\ (f_1^{M})'(x,\omega)&\text{if}\ x > x_0\end{cases} \quad\text{and}\quad (f_{1,3}^{M})'(x,\omega) = \begin{cases} (f_1^{M})'(x,\omega)&\text{if}\ x < y_0,\\ (f_3^{M})'(x,\omega)&\text{if}\ x > y_0\end{cases} \]
with the notation in \eqref{eq:ezgi1} and \eqref{eq:ezgi3}.
Observe that, at every $x\in\mathbb{R}\setminus\{x_0\}$, the function $f_{3,1}^{M}(\cdot,\omega)$ satisfies both \eqref{eq:vsol} and \eqref{eq:vsag} with $\lambda = M$. In addition,
\begin{align*}
(f_{3,1}^{M})_+'(x_0,\omega) = (f_1^{M})'(x_0,\omega) &= G_1^{-1}(M - \beta V(x_0,\omega)) = G_1^{-1}(M - \epsilon) < p_m < p_M\quad\text{and}\\
(f_{3,1}^{M})_-'(x_0,\omega) = (f_3^{M})'(x_0,\omega) &= G_3^{-1}(M - \beta V(x_0,\omega)) = G_3^{-1}(M - \epsilon) > p_M.
\end{align*}
It is clear from Condition \ref{cond:pm} that
\[ G(p) + \beta V(x_0,\omega) = G(p) + \epsilon \le M + \epsilon \]
for every $p\in [G_1^{-1}(M - \epsilon),G_3^{-1}(M - \epsilon)],$
i.e., $f_{3,1}^{M}(\cdot,\omega)$ satisfies \eqref{eq:vsol} at $x = x_0$ with $\lambda = M + \epsilon$. We deduce that $f_{3,1}^{M}(\cdot,\omega)$ is a subsolution of \eqref{eq:aux} with $\lambda = M + \epsilon$.
Similarly, at every $x\in\mathbb{R}\setminus\{y_0\}$, the function $f_{1,3}^{M}(\cdot,\omega)$ satisfies both \eqref{eq:vsol} and \eqref{eq:vsag} with $\lambda = M$. In addition,
\begin{align*}
(f_{1,3}^{M})_-'(y_0,\omega) = (f_1^{M})'(y_0,\omega) &= G_1^{-1}(m + \beta - \beta V(y_0,\omega)) = G_1^{-1}(m + \epsilon) < p_m < p_M\quad\text{and}\\
(f_{1,3}^{M})_+'(y_0,\omega) = (f_3^{M})'(y_0,\omega) &= G_3^{-1}(m + \beta - \beta V(y_0,\omega)) = G_3^{-1}(m + \epsilon) > p_M.
\end{align*}
It is clear from Condition \ref{cond:pm} that
\[ G(p) + \beta V(y_0,\omega) = G(p) + \beta - \epsilon \ge m + \beta - \epsilon = M - \epsilon \]
for every $p\in [G_1^{-1}(m + \epsilon),G_3^{-1}(m + \epsilon)],$
i.e., $f_{1,3}^{M}(\cdot,\omega)$ satisfies \eqref{eq:vsag} at $x = y_0$ with $\lambda = M - \epsilon$. We deduce that $f_{1,3}^{M}(\cdot,\omega)$ is a supersolution of \eqref{eq:aux} with $\lambda = M - \epsilon$.
Moreover,
\begin{align*}
	\lim_{x\to+\infty} \frac{f_{3,1}^M(x,\omega)}{x} &= \lim_{x\to-\infty} \frac{f_{1,3}^M(x,\omega)}{x} = \theta_1(M) < p_m < p_M\\
	&< \theta_3(M) = \lim_{x\to-\infty} \frac{f_{3,1}^M(x,\omega)}{x} = \lim_{x\to+\infty} \frac{f_{1,3}^M(x,\omega)}{x}
\end{align*}
for $\mathbb{P}$-a.e.\ $\omega$.
Therefore,
\[ \Ham^U(\theta) \le M + \epsilon \ \ \text{and}\ \  \Ham^L(\theta) \ge M - \epsilon\ \ \text{for every}\ \theta\in(\theta_1(M),\theta_3(M)) \]
by Proposition \ref{prop:onesided}. Since $\epsilon\in(0,\beta)$ is arbitrary, we conclude that
\[ \Ham(\theta) = \Ham^L(\theta) = \Ham^U(\theta) = M\ \ \text{for every}\ \theta\in(\theta_1(M),\theta_3(M)). \]

\subsection{Possible flat piece at height $M$ (excluding the easy subcase)}\label{sub:apis}

Assume that $\beta < M < m + \beta$. Fix an arbitrary $\epsilon\in(0, m + \beta - M)$. Recall $\underline x_0 = \underline x_0(M + \epsilon,\omega)$ and $\overline y_0 = \overline y_0(M + \epsilon,\omega)$ from Subsection \ref{sub:ister}. Note that
\[ M + \epsilon - \beta V(\underline x_0,\omega) = M \quad\text{and}\quad M + \epsilon - \beta V(\overline y_0,\omega) = m \]
for $\mathbb{P}$-a.e.\ $\omega$. Construct $f_{3,1,3}^M(\cdot,\omega),f_{1,3,3}^M(\cdot,\omega)\in\Lip(\mathbb{R})$ by setting $f_{3,1,3}^M(0,\omega) = f_{1,3,3}^M(0,\omega) = 0$,
\[ (f_{3,1,3}^M)'(x,\omega) = \begin{cases} (f_{3}^M)'(x,\omega)&\text{if}\ x < \underline x_0,\\ (\underline f_{1,3}^{M + \epsilon})'(x,\omega)&\text{if}\ x > \underline x_0\end{cases} \quad\text{and}\quad (f_{1,3,3}^M)'(x,\omega) = \begin{cases} (\underline f_{1,3}^{M + \epsilon})'(x,\omega)&\text{if}\ x < \overline y_0,\\ (f_3^M)'(x,\omega)&\text{if}\ x > \overline y_0\end{cases} \]
with the notation in \eqref{eq:ezgi3} and \eqref{eq:bizdik}.
Observe that, at every $x\in\mathbb{R}\setminus\{\underline x_0\}$, the function $f_{3,1,3}^M(\cdot,\omega)$ satisfies \eqref{eq:vsol} with $\lambda = M + \epsilon$. In addition,
\begin{align*}
(f_{3,1,3}^M)_+'(\underline x_0,\omega) &= (\underline f_{1,3}^{M + \epsilon})_+'(\underline x_0,\omega) =  G_1^{-1}(M + \epsilon - \beta V(\underline x_0,\omega)) = G_1^{-1}(M) < p_m < p_M\quad\text{and}\\
(f_{3,1,3}^M)_-'(\underline x_0,\omega) &= (f_3^M)'(\underline x_0,\omega) = G_3^{-1}(M - \beta V(\underline x_0,\omega)) = G_3^{-1}(M - \epsilon) > p_M.
\end{align*}
It is clear from Condition \ref{cond:pm} that
\[ G(p) + \beta V(\underline x_0,\omega) = G(p) + \epsilon \le M + \epsilon \]
for every $p\in [G_1^{-1}(M),G_3^{-1}(M - \epsilon)],$
i.e., $f_{3,1,3}^M(\cdot,\omega)$ satisfies \eqref{eq:vsol} at $x = \underline x_0$ with $\lambda = M + \epsilon$. We deduce that $f_{3,1,3}^M(\cdot,\omega)$ is a subsolution of \eqref{eq:aux} with $\lambda = M + \epsilon$. Similarly, at every $x\in\mathbb{R}\setminus\{\overline y_0\}$, the function $f_{1,3,3}^M(\cdot,\omega)$ satisfies \eqref{eq:vsag} with $\lambda = M$. In addition,
\begin{align*}
(f_{1,3,3}^M)_-'(\overline y_0,\omega) &= (\underline f_{1,3}^{M + \epsilon})_-'(\overline y_0,\omega) = G_1^{-1}(M + \epsilon - \beta V(\overline y_0,\omega)) = G_1^{-1}(m) = p_m < p_M\quad\text{and}\\
(f_{1,3,3}^M)_+'(\overline y_0,\omega) &= (f_3^{M})'(\overline y_0,\omega) = G_3^{-1}(M - \beta V(\overline y_0,\omega)) = G_3^{-1}(m - \epsilon) > p_M.
\end{align*}
It is clear from Condition \ref{cond:pm} that
\[ G(p) + \beta V(\overline y_0,\omega) = G(p) + M + \epsilon - m \ge M \]
for every $p\in [G_1^{-1}(m),G_3^{-1}(m - \epsilon)],$
i.e., $f_{1,3,3}^M(\cdot,\omega)$ satisfies \eqref{eq:vsag} at $x = \overline y_0$ with $\lambda = M$. We deduce that $f_{1,3,3}^M(\cdot,\omega)$ is a supersolution of \eqref{eq:aux} with $\lambda = M$. Moreover,
\begin{align*}
\lim_{x\to+\infty} \frac{f_{3,1,3}^M(x,\omega)}{x} &= \lim_{x\to-\infty} \frac{f_{1,3,3}^M(x,\omega)}{x} = \underline\theta_{1,3}(M + \epsilon) \\
&< \theta_3(M) = \lim_{x\to-\infty} \frac{f_{3,1,3}^M(x,\omega)}{x} = \lim_{x\to+\infty} \frac{f_{1,3,3}^M(x,\omega)}{x}
\end{align*}
for $\mathbb{P}$-a.e.\ $\omega$. Therefore,
\[ \Ham^U(\theta) \le M + \epsilon \ \ \text{and}\ \  \Ham^L(\theta) \ge M\ \ \text{for every}\ \theta\in(\underline\theta_{1,3}(M + \epsilon),\theta_3(M)) \]
by Proposition \ref{prop:onesided}.
Since $\epsilon\in(0, m + \beta - M)$ is arbitrary, we recall \eqref{eq:chll} and conclude that
\[ \Ham(\theta) = \Ham^L(\theta) = \Ham^U(\theta) = M\ \ \text{for every}\ \theta\in[\theta_{1,3}(M),\theta_3(M)). \]

We hereby showed that the graph of $\Ham$ has a flat piece at height $M$ if and only if $\theta_{1,3}(M) < \theta_3(M)$. We will characterize this strict inequality in the proof of Theorem \ref{thm:ergthmapp} (see Section \ref{sec:newflat}).

\subsection{Possible flat piece at height $m + \beta$ (excluding the easy subcase)}\label{sub:eflat}

In this subsection, we will simultaneously consider Cases II and III of Theorem \ref{thm:velinim} (except their easy subcases).

Assume that $\max\{\beta,M\} < m + \beta$. Fix an arbitrary $\epsilon\in(0, m + \beta - \max\{\beta,M\})$. Recall $\overline x_0 = \overline x_0(m + \beta - \epsilon,\omega)$ and $\underline y_0 = \underline y_0(m + \beta - \epsilon,\omega)$ from Subsection \ref{sub:ister}. Note that
\[ m + \beta - \epsilon - \beta V(\overline x_0,\omega) = M \quad\text{and}\quad m + \beta - \epsilon - \beta V(\underline y_0,\omega) = m \]
for $\mathbb{P}$-a.e.\ $\omega$. Construct $f_{1,3,1}^{m + \beta}(\cdot,\omega),f_{1,1,3}^{m + \beta}(\cdot,\omega)\in\Lip(\mathbb{R})$ by setting $f_{1,3,1}^{m + \beta}(0,\omega) = f_{1,1,3}^{m + \beta}(0,\omega) = 0$,
\begin{align*}
	(f_{1,3,1}^{m + \beta})'(x,\omega) &= \begin{cases} (\overline f_{1,3}^{m + \beta - \epsilon})'(x,\omega)&\text{if}\ x < \overline x_0,\\ (f_1^{m + \beta})'(x,\omega)&\text{if}\ x > \overline x_0\end{cases}\quad \text{and}\\
	(f_{1,1,3}^{m + \beta})'(x,\omega) &= \begin{cases} (f_1^{m + \beta})'(x,\omega)&\text{if}\ x < \underline y_0,\\ (\overline f_{1,3}^{m + \beta - \epsilon})'(x,\omega)&\text{if}\ x > \underline y_0\end{cases}
\end{align*}
with the notation in \eqref{eq:ezgi1} and \eqref{eq:anbizdik}.
Observe that, at every $x\in\mathbb{R}\setminus\{\overline x_0\}$, the function $f_{1,3,1}^{m + \beta}(\cdot,\omega)$ satisfies \eqref{eq:vsol} with $\lambda = m + \beta$. In addition,
\begin{align*}
(f_{1,3,1}^{m + \beta})_+'(\overline x_0,\omega) &= (f_1^{m + \beta})'(\overline x_0,\omega) =  G_1^{-1}(m + \beta - \beta V(\overline x_0,\omega)) = G_1^{-1}(M + \epsilon) < p_m < p_M\quad\text{and}\\
(f_{1,3,1}^{m + \beta})_-'(\overline x_0,\omega) &= (\overline f_{1,3}^{m + \beta - \epsilon})_-'(\overline x_0,\omega) = G_3^{-1}(m + \beta - \epsilon - \beta V(\overline x_0,\omega)) = G_3^{-1}(M) = p_M.
\end{align*}
It is clear from Condition \ref{cond:pm} that
\[ G(p) + \beta V(\overline x_0,\omega) = G(p) + m + \beta - \epsilon - M \le m + \beta \]
for every $p\in [G_1^{-1}(M + \epsilon),G_3^{-1}(M)],$
i.e., $f_{1,3,1}^{m + \beta}(\cdot,\omega)$ satisfies \eqref{eq:vsol} at $x = \overline x_0$ with $\lambda = m + \beta$. We deduce that $f_{1,3,1}^{m + \beta}(\cdot,\omega)$ is a subsolution of \eqref{eq:aux} with $\lambda = m + \beta$.
Similarly, at every $x\in\mathbb{R}\setminus\{\underline y_0\}$, the function $f_{1,1,3}^{m + \beta}(\cdot,\omega)$ satisfies \eqref{eq:vsag} with $\lambda = m + \beta - \epsilon$. In addition,
\begin{align*}
(f_{1,1,3}^{m + \beta})_-'(\underline y_0,\omega) &= (f_1^{m + \beta})'(\underline y_0,\omega) = G_1^{-1}(m + \beta - \beta V(\underline y_0,\omega)) = G_1^{-1}(m + \epsilon) < p_m < p_M\quad\text{and}\\
(f_{1,1,3}^{m + \beta})_+'(\underline y_0,\omega) &= (\overline f_{1,3}^{m + \beta - \epsilon})_+'(\underline y_0,\omega) = G_3^{-1}(m + \beta - \epsilon - \beta V(\underline y_0,\omega)) = G_3^{-1}(m) > p_M.
\end{align*}
It is clear from Condition \ref{cond:pm} that
\[ G(p) + \beta V(\underline y_0,\omega) = G(p) + \beta - \epsilon \ge m + \beta - \epsilon \]
for every $p\in [G_1^{-1}(m + \epsilon),G_3^{-1}(m)],$
i.e., $f_{1,1,3}^{m + \beta}(\cdot,\omega)$ satisfies \eqref{eq:vsag} at $x = \underline y_0$ with $\lambda = m + \beta - \epsilon$. We deduce that $f_{1,1,3}^{m + \beta}(\cdot,\omega)$ is a supersolution of \eqref{eq:aux} with $\lambda = m + \beta - \epsilon$. Moreover,
\begin{align*}
\lim_{x\to+\infty} \frac{f_{1,3,1}^{m + \beta}(x,\omega)}{x} &= \lim_{x\to-\infty} \frac{f_{1,1,3}^{m + \beta}(x,\omega)}{x} = \theta_1(m + \beta) \\
&< \overline\theta_{1,3}(m + \beta - \epsilon) = \lim_{x\to-\infty} \frac{f_{1,3,1}^{m + \beta}(x,\omega)}{x} = \lim_{x\to+\infty} \frac{f_{1,1,3}^{m + \beta}(x,\omega)}{x}
\end{align*}
for $\mathbb{P}$-a.e.\ $\omega$. Therefore,
\[ \Ham^U(\theta) \le m + \beta \ \ \text{and}\ \  \Ham^L(\theta) \ge m + \beta - \epsilon\ \ \text{for every}\ \theta\in(\theta_1(m + \beta),\overline\theta_{1,3}(m + \beta - \epsilon)) \]
by Proposition \ref{prop:onesided}.
Since $\epsilon\in(0, m + \beta - \max\{\beta,M\})$ is arbitrary, we recall \eqref{eq:chll} and conclude that
\[ \Ham(\theta) = \Ham^L(\theta) = \Ham^U(\theta) = m + \beta\ \ \text{for every}\ \theta\in(\theta_1(m + \beta),\theta_{1,3}(m + \beta)]. \]

We hereby showed that the graph of $\Ham$ has a flat piece at height $m + \beta$ if and only if $\theta_1(m + \beta) < \theta_{1,3}(m + \beta)$. We will characterize this strict inequality in the proof of Theorem \ref{thm:ergthmapp} (see Section \ref{sec:newflat}).

\section{Case III: Strong potential ($M \le \beta \le m + \beta$)}\label{sec:strong}

\subsection{Strictly decreasing piece involving $G_1$}

We proved in Subsection \ref{sub:dec1} that, in Cases I, II and III, for every $\lambda \in [m + \beta,+\infty)$,
\[ \Ham(\theta_1(\lambda)) = \Ham^L(\theta_1(\lambda)) = \Ham^U(\theta_1(\lambda)) = \lambda \]
with the definition in \eqref{eq:ddnono}. We added that $\Ham$ is strictly decreasing on $(-\infty,\theta_1(m + \beta)]$.

\subsection{Nonincreasing piece involving both $G_1$ and $G_3$ (excluding the easy subcase)}

We proved in Subsection \ref{sub:ister} that, in Cases II and III (except their easy subcases), 
\[ \Ham(\theta) = \Ham^L(\theta) = \Ham^U(\theta) = \Lambda(\theta)\ \ \text{for every}\ \theta\in ( \theta_{1,3}(m + \beta), \theta_{1,3}(\max\{\beta,M\}) ). \]
We added that $\Ham$ is nonincreasing on this interval. See Theorem \ref{thm:gennet}.

\subsection{Strictly increasing piece involving $G_4$}

We proved in Subsection \ref{sub:inc4} that, in Cases I, II and III, for every $\lambda \in [\beta,+\infty)$,
\[ \Ham(\theta_4(\lambda)) = \Ham^L(\theta_4(\lambda)) = \Ham^U(\theta_4(\lambda)) = \lambda \]
with the definition in \eqref{eq:ddnono}. We added that $\Ham$ is strictly increasing on $[\theta_4(\beta),+\infty)$.

\subsection{Flat piece at height $\beta = m + \beta$ (in the easy subcase)}\label{sub:streasy}

Assume that $M \le \beta = m + \beta$. For every $\epsilon\in(0,\beta)$ and $\mathbb{P}$-a.e.\ $\omega$, there exist $x_0 = x_0(\omega)\in\mathbb{R}$ and $y_0 = y_0(\omega)\in\mathbb{R}$ such that $\beta V(x_0,\omega) = \epsilon$ and $\beta V(y_0,\omega) = \beta - \epsilon$. Construct $f_{4,1}^\beta(\cdot,\omega),f_{1,4}^\beta(\cdot,\omega)\in\Lip(\mathbb{R})$ by setting $f_{4,1}^\beta(0,\omega) = f_{1,4}^\beta(0,\omega) = 0$,
\[ (f_{4,1}^\beta)'(x,\omega) = \begin{cases} (f_4^\beta)'(x,\omega)&\text{if}\ x < x_0,\\ (f_1^\beta)'(x,\omega)&\text{if}\ x > x_0\end{cases} \quad\text{and}\quad (f_{1,4}^\beta)'(x,\omega) = \begin{cases} (f_1^\beta)'(x,\omega)&\text{if}\ x < y_0,\\ (f_4^\beta)'(x,\omega)&\text{if}\ x > y_0\end{cases} \]
with the notation in \eqref{eq:ezgi1} and \eqref{eq:ezgi4}.
Observe that, at every $x\in\mathbb{R}\setminus\{x_0\}$, the function $f_{4,1}^\beta(\cdot,\omega)$ satisfies both \eqref{eq:vsol} and \eqref{eq:vsag} with $\lambda = \beta$. In addition,
\begin{align*}
(f_{4,1}^\beta)_+'(x_0,\omega) &= (f_1^\beta)'(x_0,\omega) = G_1^{-1}(\beta - \beta V(x_0,\omega)) = G_1^{-1}(\beta - \epsilon) < p_m < 0\quad\text{and}\\
(f_{4,1}^\beta)_-'(x_0,\omega) &= (f_4^\beta)'(x_0,\omega) = G_4^{-1}(\beta - \beta V(x_0,\omega)) = G_4^{-1}(\beta - \epsilon) > 0.
\end{align*}
It is clear from Condition \ref{cond:pm} that
\[ G(p) + \beta V(x_0,\omega) = G(p) + \epsilon \le \max\{\beta - \epsilon,M\} + \epsilon \le \beta + \epsilon \]
for every $p\in [G_1^{-1}(\beta - \epsilon),G_4^{-1}(\beta - \epsilon)],$
i.e., $f_{4,1}^\beta(\cdot,\omega)$ satisfies \eqref{eq:vsol} at $x = x_0$ with $\lambda = \beta + \epsilon$. We deduce that $f_{4,1}^\beta(\cdot,\omega)$ is a subsolution of \eqref{eq:aux} with $\lambda = \beta + \epsilon$.
Similarly, at every $x\in\mathbb{R}\setminus\{y_0\}$, the function $f_{1,4}^\beta(\cdot,\omega)$ satisfies both \eqref{eq:vsol} and \eqref{eq:vsag} with $\lambda = \beta$. In addition,
\begin{align*}
(f_{1,4}^\beta)_-'(y_0,\omega) &= (f_1^\beta)'(y_0,\omega) = G_1^{-1}(\beta - \beta V(y_0,\omega)) = G_1^{-1}(\epsilon) < p_m < 0\quad\text{and}\\
(f_{1,4}^\beta)_+'(y_0,\omega) &= (f_4^\beta)'(y_0,\omega) = G_4^{-1}(\beta - \beta V(y_0,\omega)) = G_4^{-1}(\epsilon) > 0.
\end{align*}
It is clear from Condition \ref{cond:pm} that
\[ G(p) + \beta V(y_0,\omega) = G(p) + \beta - \epsilon \ge \beta - \epsilon \]
for every $p\in [G_1^{-1}(\epsilon),G_4^{-1}(\epsilon)],$
i.e., $f_{1,4}^\beta(\cdot,\omega)$ satisfies \eqref{eq:vsag} at $x = y_0$ with $\lambda = \beta - \epsilon$. We deduce that $f_{1,4}^\beta(\cdot,\omega)$ is a supersolution of \eqref{eq:aux} with $\lambda = \beta - \epsilon$.
Moreover,
\begin{align*}
\lim_{x\to+\infty} \frac{f_{4,1}^\beta(x,\omega)}{x} &= \lim_{x\to-\infty} \frac{f_{1,4}^\beta(x,\omega)}{x} = \theta_1(\beta) < p_m < 0\\
&< \theta_4(\beta) = \lim_{x\to-\infty} \frac{f_{4,1}^\beta(x,\omega)}{x} = \lim_{x\to+\infty} \frac{f_{1,4}^\beta(x,\omega)}{x}
\end{align*}
for $\mathbb{P}$-a.e.\ $\omega$.
Therefore,
\[ \Ham^U(\theta) \le \beta + \epsilon \ \ \text{and}\ \  \Ham^L(\theta) \ge \beta - \epsilon\ \ \text{for every}\ \theta\in(\theta_1(\beta),\theta_4(\beta)) \]
by Proposition \ref{prop:onesided}. Since $\epsilon\in(0,\beta)$ is arbitrary, we conclude that
\[ \Ham(\theta) = \Ham^L(\theta) = \Ham^U(\theta) = \beta\ \ \text{for every}\ \theta\in(\theta_1(\beta),\theta_4(\beta)). \]

\subsection{Flat piece at height $\beta$ (excluding the easy subcase)}

Assume that $M \le \beta < m + \beta$. For every $\epsilon\in(0,m/2)$ and $\mathbb{P}$-a.e.\ $\omega$, there exists a $z_0 = z_0(\omega)\in\mathbb{R}$ such that $\beta V(z_0,\omega) = \beta - \epsilon$. Construct $f_{4,1,3}^\beta(\cdot,\omega),f_{1,3,4}^\beta(\cdot,\omega)\in\Lip(\mathbb{R})$ by setting $f_{4,1,3}^\beta(0,\omega) = f_{1,3,4}^\beta(0,\omega) = 0$,
\[ (f_{4,1,3}^\beta)'(x,\omega) = \begin{cases} (f_4^\beta)'(x,\omega)&\text{if}\ x < z_0,\\ (\underline f_{1,3}^{\beta + \epsilon})'(x,\omega)&\text{if}\ x > z_0\end{cases}\quad \text{and}\quad (f_{1,3,4}^\beta)'(x,\omega) = \begin{cases} (\underline f_{1,3}^{\beta + \epsilon})'(x,\omega)&\text{if}\ x < z_0,\\ (f_4^\beta)'(x,\omega)&\text{if}\ x > z_0\end{cases} \]
with the notation in \eqref{eq:ezgi4} and \eqref{eq:bizdik}.
Observe that, at every $x\in\mathbb{R}\setminus\{z_0\}$, the function $f_{4,1,3}^\beta(\cdot,\omega)$ (resp.\ $f_{1,3,4}^\beta(\cdot,\omega)$) satisfies \eqref{eq:vsol} (resp.\ \eqref{eq:vsag}) with $\lambda = \beta + \epsilon$ (resp.\ $\lambda = \beta$). In addition,
\begin{align*}
(f_{4,1,3}^\beta)_+'(z_0,\omega) &= (f_{1,3,4}^\beta)_-'(z_0,\omega) = (\underline f_{1,3}^{\beta + \epsilon})'(z_0,\omega) = G_3^{-1}(\beta + \epsilon - \beta V(z_0,\omega)) = G_3^{-1}(2\epsilon) < 0\quad\text{and}\\
(f_{4,1,3}^\beta)_-'(z_0,\omega) &= (f_{1,3,4}^\beta)_+'(z_0,\omega) = (f_4^\beta)'(z_0,\omega) = G_4^{-1}(\beta - \beta V(z_0,\omega)) = G_4^{-1}(\epsilon) > 0.
\end{align*}
The third equality in the first line of this display follows from the observation that
\[z_0(\omega)\in(\overline y_i(\beta + \epsilon,\omega),\underline x_{i+1}(\beta + \epsilon,\omega))\ \ \text{for some}\ i\in\mathbb{Z} \]
by \eqref{eq:ofkeyonet} since $\beta + \epsilon - \beta V(z_0,\omega) = 2\epsilon < m$.
It is clear from Condition \ref{cond:pm} that
\[ \beta - \epsilon \le G(p) + \beta V(z_0,\omega) = G(p) + \beta - \epsilon \le \beta + \epsilon \]
for every $p\in [G_3^{-1}(2\epsilon),G_4^{-1}(\epsilon)],$
i.e., $f_{4,1,3}^\beta(\cdot,\omega)$ satisfies \eqref{eq:vsol} at $x = z_0$ with $\lambda = \beta + \epsilon$ and $f_{1,3,4}^\beta(\cdot,\omega)$ satisfies \eqref{eq:vsag} at $x = z_0$ with $\lambda = \beta - \epsilon$. We deduce that $f_{4,1,3}^\beta(\cdot,\omega)$ is a subsolution of \eqref{eq:aux} with $\lambda = \beta + \epsilon$ and $f_{1,3,4}^\beta(\cdot,\omega)$ is a supersolution of \eqref{eq:aux} with $\lambda = \beta - \epsilon$.
Moreover,
\begin{align*}
	\lim_{x\to+\infty} \frac{f_{4,1,3}^\beta(x,\omega)}{x} &= \lim_{x\to-\infty} \frac{f_{1,3,4}^\beta(x,\omega)}{x} = \underline\theta_{1,3}(\beta + \epsilon) < 0\\
	&< \theta_4(\beta) = \lim_{x\to-\infty} \frac{f_{4,1,3}^\beta(x,\omega)}{x} = \lim_{x\to+\infty} \frac{f_{1,3,4}^\beta(x,\omega)}{x}
\end{align*}
for $\mathbb{P}$-a.e.\ $\omega$. Therefore,
\[ \Ham^U(\theta) \le \beta + \epsilon \ \ \text{and}\ \  \Ham^L(\theta) \ge \beta - \epsilon\ \ \text{for every}\ \theta\in(\underline\theta_{1,3}(\beta + \epsilon),\theta_4(\beta)) \]
by Proposition \ref{prop:onesided}. Since $\epsilon\in(0,m/2)$ is arbitrary, we recall \eqref{eq:chll} and conclude that
\[ \Ham(\theta) = \Ham^L(\theta) = \Ham^U(\theta) = \beta\ \ \text{for every}\ \theta\in[\theta_{1,3}(\beta),\theta_4(\beta)). \]

\subsection{Possible flat piece at height $m + \beta$ (excluding the easy subcase)}

We proved in Subsection \ref{sub:eflat} that, in Cases II and III (except their easy subcases),
\[ \Ham(\theta) = \Ham^L(\theta) = \Ham^U(\theta) = m + \beta\ \ \text{for every}\ \theta\in(\theta_1(m + \beta),\theta_{1,3}(m + \beta)]. \]
In words, the graph of $\Ham$ has a flat piece at height $m + \beta$ if and only if $\theta_1(m + \beta) < \theta_{1,3}(m + \beta)$. We will characterize this strict inequality in the proof of Theorem \ref{thm:ergthmapp} (see Section \ref{sec:newflat}).

\section{Completing the proof of homogenization}\label{sec:summary}

\begin{definition}\label{def:lamb}
	Recall \eqref{eq:ddnono}, \eqref{eq:hurma1}, \eqref{eq:hurma2}, \eqref{eq:chll} and \eqref{eq:nowaste}.
	\begin{itemize}
		\item [(a)] When $\beta < M = m + \beta$ (i.e., the easy subcase of medium potential), let
					\[ \overline\Lambda(\theta) = M\ \text{on}\ (\theta_1(M),\theta_3(M)).\]
		\item [(b)] When $\beta < M < m + \beta$ (i.e., medium potential, excluding the easy subcase above), let
					\[ \overline\Lambda(\theta) = \begin{cases} m + \beta\ \text{on}\ (\theta_1(m + \beta),\theta_{1,3}(m + \beta)]&\text{(possible flat piece)},\\
																\Lambda(\theta)\ \text{on}\ (\theta_{1,3}(m + \beta),\theta_{1,3}(M))&\text{(nonincreasing)},\\
																M\ \text{on}\ [\theta_{1,3}(M),\theta_3(M))&\text{(possible flat piece)}.
												  \end{cases} \]
		\item [(c)] When $M \le \beta = m + \beta$ (i.e., the easy subcase of strong potential), let
					\[ \overline\Lambda(\theta) = \beta\ \text{on}\ (\theta_1(\beta),\theta_4(\beta)). \]
		\item [(d)] When $M \le \beta  < m + \beta$ (i.e., strong potential, excluding the easy subcase above), let
				    \[ \overline\Lambda(\theta) = \begin{cases} m + \beta\ \text{on}\ (\theta_1(m + \beta),\theta_{1,3}(m + \beta)]&\text{(possible flat piece)},\\
																\Lambda(\theta)\ \text{on}\ (\theta_{1,3}(m + \beta),\theta_{1,3}(\beta))&\text{(nonincreasing)},\\
														  		\beta\ \text{on}\ [\theta_{1,3}(\beta),\theta_4(\beta))&\text{(flat piece)}.
												  \end{cases} \]
	\end{itemize}
\end{definition}

\begin{proof}[Proof of Theorem \ref{thm:velinim}]
	Under the assumptions in Theorem \ref{thm:velinim}, we showed in Sections \ref{sec:weak}, \ref{sec:medium} and \ref{sec:strong} that, in Cases I, II and III, respectively,
	\[ \Ham(\theta) = \Ham^L(\theta) = \Ham^U(\theta) \]
	for every $\theta\in\mathbb{R}$. Moreover, in each case, we verified that $\Ham$ has the piecewise description in the statement of Theorem \ref{thm:velinim} which, in Cases II and III, involves the function $\overline \Lambda$ in Definition \ref{def:lamb}. Therefore, the desired conclusions follow from Theorem \ref{thm:ottur}.
\end{proof}

\section{Flat pieces of the graph of $\overline\Lambda$ at its extreme values}\label{sec:newflat}

We break down the proof of Theorem \ref{thm:ergthmapp} into several lemmas involving the quantity
\[ q_1 = \mathbb{P}\left(V(x,\omega) = 1\ \text{for some}\ x\in[0,1]\right). \]

\begin{lemma}\label{lem:aynaer}
	Assume that $\max\{\beta,M\} < m + \beta$. Recall the bi-infinite sequence defined by \eqref{eq:cipayok}--\eqref{eq:cip2}. If $q_1 > 0$, then for every $\lambda\in[\beta,+\infty)\cap(M,m + \beta)$, there exist $K,\delta>0$ such that the event
	\[A_{\lambda',K,\delta} = \{ 0 \le \underline y_i(\lambda',\omega) < \underline y_i(\lambda',\omega) + \delta \le \overline x_{i+1}(\lambda',\omega) \le K\ \text{for some}\ i\in\mathbb{Z} \} \]
	satisfies $\mathbb{P}(A_{\lambda',K,\delta}) \ge \frac1{2}$ whenever $\lambda'\in(\lambda,m + \beta)$.
\end{lemma}

\begin{proof}
	For every $\lambda\in[\beta,+\infty)\cap(M,m + \beta)$, our ergodicity assumption ensures that
	\[ \mathbb{P}(\lambda - \beta V(x,\omega) > M\ \text{for some}\ x\ge0) = 1. \]
	Similarly, if $q_1>0$, then  $\mathbb{P}(V(x,\omega) = 1\ \text{for some}\ x\ge0) = 1$. Therefore, $\exists\,K'>0$ such that
	\[ \mathbb{P}(\lambda - \beta V(x,\omega) > M\ \text{for some}\ x\in[0,K']) \ge \frac{7}{8}\quad\text{and}\quad\mathbb{P}(V(x,\omega) = 1\ \text{for some}\ x\in[0,K']) \ge \frac{7}{8}. \]
	
	Consider the event
		\begin{align*}
		B_{\lambda,K'} &= \left\{\exists\,z_1,z_2,z_3\in\mathbb{R}\ \text{s.t.}\ 0\le z_1 \le K' \le z_2 \le 2K' \le z_3 \le 3K',\ \lambda - \beta V(z_1,\omega) > M, \right.\\
		&\quad\ \left. \, V(z_2,\omega) = 1\ \text{and}\ \lambda - \beta V(z_3,\omega) > M\right\}.
		\end{align*}
	It follows from \eqref{eq:stationary} and a union bound that
	\[ \mathbb{P}((B_{\lambda,K'})^c) \le (1 - \frac{7}{8}) + (1 - \frac{7}{8}) + (1 - \frac{7}{8}) = \frac{3}{8}. \]
	The continuity of the potential implies the existence of a $\delta\in(0,K')$ such that the event
	\begin{align*}
	B_{\lambda,K',\delta} &= \left\{ \exists\, z_1,z_2,z_3\in\mathbb{R}\ \text{s.t.}\ 0\le z_1 \le K' \le z_2 \le 2K' \le z_3 \le 3K',\ \lambda - \beta V(z_1,\omega) > M,\right.\\
	&\quad\ \,\left. V(z_2,\omega) = 1,\ m + \beta - \beta V(z,\omega)\le M\ \text{for all}\ z\in[z_2,z_2 + \delta]\ \text{and}\ \lambda - \beta V(z_3,\omega) > M\right\}
	\end{align*}
	satisfies
	\[\mathbb{P}(B_{\lambda,K',\delta}) \ge \mathbb{P}(B_{\lambda,K'}) - \frac1{8} \ge \frac1{2}.\]
	
	Whenever $\lambda'\in(\lambda,m + \beta)$ and $K \ge 3K'$, 
		\begin{align*}
		B_{\lambda,K',\delta} &\subset \left\{ \exists\, z_1,z_2,z_3\in\mathbb{R}\ \text{s.t.}\ 0\le z_1 \le K' \le z_2 \le 2K' \le z_3 \le 3K',\ \lambda' - \beta V(z_1,\omega) > M,\right.\\
		&\quad\ \left. \lambda' - \beta V(z_2,\omega) < m,\ \lambda' - \beta V(z,\omega) < M\ \text{for all}\ z\in[z_2,z_2 + \delta]\ \text{and}\ \lambda' - \beta V(z_3,\omega) > M\right\}\\
		&\subset \{ \exists\, z_1,z_2,z_3\in\mathbb{R}\ \text{and}\ i\in\mathbb{Z}\ \text{s.t.}\ 0 \le z_1 < \underline y_i(\lambda',\omega) < z_2 < z_2+\delta < \overline x_{i+1}(\lambda',\omega) < z_3 \le 3K'\}\\
		&\subset A_{\lambda',K,\delta}
		\end{align*}
		by \eqref{eq:ayakbas}. We conclude that $\mathbb{P}(A_{\lambda',K,\delta}) \ge \mathbb{P}(B_{\lambda,K',\delta}) \ge \frac1{2}$.
\end{proof}

\begin{lemma}\label{lem:oncelli}
	Assume that $\max\{\beta,M\} < m + \beta$. Recall the definitions in \eqref{eq:ddnono} and \eqref{eq:chll}. If $q_1 > 0$, then $\theta_1(m + \beta) < \theta_{1,3}(m + \beta)$.
\end{lemma}

\begin{proof}
	Fix a $\lambda\in[\beta,+\infty)\cap(M,m + \beta)$. If $q_1 > 0$, then the ergodic theorem and Lemma \ref{lem:aynaer} ensure the existence of $K,\delta>0$ such that
	\begin{equation}\label{eq:dedcruz}
		\lim_{L\to +\infty}\frac1{L}\int_0^L\one_{ \{ \tau_x\omega\in A_{\lambda',K,\delta} \} }dx = \mathbb{P}(A_{\lambda',K,\delta}) \ge \frac1{2}
	\end{equation}
	for $\mathbb{P}$-a.e.\ $\omega$ and every $\lambda'\in(\lambda,m + \beta)$. Let
	\[ I_{\lambda',\delta} = \{ i\in\mathbb{Z}:\,\underline y_i(\lambda',\omega) + \delta \le \overline x_{i+1}(\lambda',\omega) \} \]
	and observe that
	\begin{equation}
	\begin{aligned}\label{eq:nazseltan}
		\int_0^L\one_{ \{ \tau_x\omega\in A_{\lambda',K,\delta} \} }dx &= \int_0^L\one_{ \{ x \le \underline y_i(\lambda',\omega) \le \overline x_{i+1}(\lambda',\omega) \le x + K\ \text{for some}\ i\in I_{\lambda',\delta} \} }dx\\
		&\le \sum_{i\in I_{\lambda',\delta}}\int_0^L\one_{ \{ x \le \underline y_i(\lambda',\omega) \le \overline x_{i+1}(\lambda',\omega) \le x + K \} }dx\\
		&\le (K - \delta)\#\{i\in I_{\lambda',\delta}:\, 0 \le \underline y_i(\lambda',\omega) \le \overline x_{i+1}(\lambda',\omega) \le L + K \}.
	\end{aligned}
	\end{equation}
	
	Recall \eqref{eq:ezgi1} and \eqref{eq:anbizdik}. For every $\lambda'\in(\lambda,m + \beta)$, $x\in\mathbb{R}$ and $\mathbb{P}$-a.e.\ $\omega$,
	\[ (f_1^{m + \beta})'(x,\omega) = G_1^{-1}(m + \beta - \beta V(x,\omega)) < G_1^{-1}(\lambda' - \beta V(x,\omega)) \le (\overline f_{1,3}^{\lambda'})_\pm'(x,\omega). \]
	Moreover, on any interval of the form $(\underline y_i(\lambda',\omega), \overline x_{i+1}(\lambda',\omega))$,
	\[ (f_1^{m + \beta})'(x,\omega) = G_1^{-1}(m + \beta - \beta V(x,\omega)) \le p_m < p_M \le G_3^{-1}(\lambda' - \beta V(x,\omega)) = (\overline f_{1,3}^{\lambda'})'(x,\omega). \]
	Therefore,
	\begin{align*}
	&\overline\theta_{1,3}(\lambda') - \theta_1(m + \beta) = \lim_{L\to +\infty}\frac1{L}\left( \int_0^{L + K} (\overline f_{1,3}^{\lambda'})_\pm'(x,\omega)dx - \int_0^{L + K}(f_1^{m + \beta})'(x,\omega) dx \right)\\
	&\quad \ge \liminf_{L\to +\infty}\frac{(p_M - p_m)\delta}{L} \#\{i\in I_{\lambda',\delta}:\, 0 \le \underline y_i(\lambda',\omega) \le \overline x_{i+1}(\lambda',\omega) \le L + K \}\\
	&\quad \ge \lim_{L\to +\infty}\frac{(p_M - p_m)\delta}{(K - \delta)L}\int_0^L\one_{ \{ \tau_x\omega\in A_{\lambda',K,\delta} \} }dx \ge \frac{(p_M - p_m)\delta}{2K}
	\end{align*}
	by \eqref{eq:dedcruz} and \eqref{eq:nazseltan}. We send $\lambda'\uparrow m + \beta$ and conclude that
	\[ \theta_{1,3}(m + \beta) - \theta_1(m + \beta) \ge \frac{(p_M - p_m)\delta}{2K} > 0.\qedhere \]
\end{proof}

Here is the converse of Lemma \ref{lem:oncelli}.

\begin{lemma}\label{lem:vans}
	Assume that $\max\{\beta,M\} < m + \beta$. If $q_1 = 0$, then $\theta_1(m + \beta) = \theta_{1,3}(m + \beta)$.
\end{lemma}

\begin{proof}
	Fix a $\lambda\in[\beta,+\infty)\cap(M,m + \beta)$. For every $\eta>0$, the event
	\[A_{\lambda,K} = \{ 0\le \overline y_j(\lambda,\omega) \le \underline x_{j+1}(\lambda,\omega) \le K\ \text{for some}\ j\in\mathbb{Z} \} \]
	satisfies
	\begin{equation}\label{eq:hidden}
		\mathbb{P}(A_{\lambda,K}) > 1 - \eta
	\end{equation}
	for sufficiently large $K > 0$. If $q_1 = 0$, then
	\begin{equation}\label{eq:karayk}
		\mathbb{P}(\lambda' - \beta V(y,\omega) \le m\ \text{for some}\ y\in[-2K,2K]) < \eta
	\end{equation}
	by continuity when $\lambda'\in(\lambda,m + \beta)$ is sufficiently large. 
	
	Recall \eqref{eq:ezgi1} and \eqref{eq:anbizdik}. Observe that, on any interval of the form $(\overline x_i(\lambda',\omega),\underline y_i(\lambda',\omega))$,
	\[ (\overline f_{1,3}^{\lambda'})'(x,\omega) - (f_1^{m + \beta})'(x,\omega) \le \max\{G_1^{-1}(\lambda' - \alpha) - G_1^{-1}(m + \beta - \alpha):\,0 \le \alpha \le \beta\} < \eta \]
	when $\lambda'\in(\lambda,m + \beta)$ is sufficiently large. Similarly, on any interval of the form $(\underline y_i(\lambda',\omega), \overline x_{i+1}(\lambda',\omega))$,
	\[ (\overline f_{1,3}^{\lambda'})'(x,\omega) - (f_1^{m + \beta})'(x,\omega) \le G_3^{-1}(\lambda' - \beta) - G_1^{-1}(m + \beta) \le G_3^{-1}(m) - G_1^{-1}(m + \beta) + \eta \]
	when $\lambda'\in(\lambda,m + \beta)$ is sufficiently large. If we let $C = G_3^{-1}(m) - G_1^{-1}(m + \beta)$, then
	\begin{align}
		0 &\le \theta_{1,3}(m + \beta) - \theta_1(m + \beta) \le \overline\theta_{1,3}(\lambda') - \theta_1(m + \beta)\nonumber\\
		&= \lim_{L\to +\infty}\frac1{L}\left( \int_0^L (\overline f_{1,3}^{\lambda'})_\pm'(x,\omega)dx - \int_0^L(f_1^{m + \beta})'(x,\omega) dx \right)\nonumber\\
		&\le \eta + \lim_{L\to +\infty}\frac{C}{L}\int_0^L\one_{\{ \underline y_i(\lambda',\omega) \le x < \overline x_{i+1}(\lambda',\omega)\ \text{for some}\ i\in\mathbb{Z} \}}dx\nonumber\\
		&\le \eta + \lim_{L\to +\infty}\frac{C}{L}\int_0^L\one_{\{ \overline y_j(\lambda,\omega) \le x < \underline x_{j+1}(\lambda,\omega)\ \text{for some}\ j\in J_{\lambda,\lambda'} \}}dx\label{eq:juanc}\\
		&= \eta + \lim_{L\to +\infty}\sum_{j\in J_{\lambda,\lambda'}}\frac{C}{L}\int_0^L\one_{\{ \overline y_j(\lambda,\omega) \le x < \underline x_{j+1}(\lambda,\omega) \}}dx\nonumber\\
		&\le \eta + \lim_{L\to +\infty}\sum_{j\in J_{\lambda,\lambda'}\cap N_{2K}}\frac{C}{L}\int_0^L\one_{\{ \overline y_j(\lambda,\omega) \le x < \underline x_{j+1}(\lambda,\omega) \}}dx\label{eq:firstterm}\\
		&\quad\ \  + \lim_{L\to +\infty}\sum_{j\in (N_{2K})^c}\frac{C}{L}\int_0^L\one_{\{ \overline y_j(\lambda,\omega) \le x < \underline x_{j+1}(\lambda,\omega) \}}dx.\label{eq:secondterm}
	\end{align}
Here, \eqref{eq:juanc} follows from Lemma \ref{lem:interlace} with $\lambda_1 = \lambda$ and $\lambda_2 = \lambda'$,
\[ J_{\lambda,\lambda'} = \{ j\in\mathbb{Z}:\, \lambda' - \beta V(y,\omega) \le m\ \text{for some}\ y\in[\overline y_j(\lambda,\omega),\underline x_{j+1}(\lambda,\omega)) \} \]
and
\[ N_{2K} = \{ j\in\mathbb{Z}:\, \underline x_{j+1}(\lambda,\omega) - \overline y_j(\lambda,\omega) \le 2K \}. \]
The limit in \eqref{eq:firstterm} is controlled as follows:
\begin{align*}
	&\quad\, \lim_{L\to +\infty}\sum_{j\in J_{\lambda,\lambda'}\cap N_{2K}}\frac{C}{L}\int_0^L\one_{\{ \overline y_j(\lambda,\omega) \le x < \underline x_{j+1}(\lambda,\omega) \}}dx\\
	&= \lim_{L\to +\infty}\frac{C}{L}\int_0^L\one_{\{ \overline y_j(\lambda,\omega) \le x < \underline x_{j+1}(\lambda,\omega)\ \text{for some}\ j\in J_{\lambda,\lambda'}\cap N_{2K} \}}dx\\
	&\le \lim_{L\to +\infty}\frac{C}{L}\int_0^L\one_{\{ \lambda' - \beta V(y,\omega)\le m\ \text{for some}\ y\in[x-2K,x+2K] \}}dx\\
	&= C\,\mathbb{P}(\lambda' - \beta V(y,\omega)\le m\ \text{for some}\ y\in[-2K,2K]) < C\eta
\end{align*}
by \eqref{eq:karayk} and the ergodic theorem. To control the limit in \eqref{eq:secondterm}, observe that
\begin{align*}
	\int_0^L\one_{\{ \tau_x\omega\in (A_{\lambda,K})^c \}}dx &=\int_0^L\one_{\{ \text{There is no $j\in\mathbb{Z}$ such that}\ x\le \overline y_j(\lambda,\omega) \le \underline x_{j+1}(\lambda,\omega) \le x + K \}}dx\\
	&\ge \sum_{j\in (N_{2K})^c}\left( \left| [\overline y_j(\lambda,\omega), \underline x_{j+1}(\lambda,\omega))\cap[0,L] \right| - K\right)\times\one_{\{\underline x_{j+1}(\lambda,\omega)\ge 2K,\ \overline y_j(\lambda,\omega)\le L-2K \}}\\
	&\ge \sum_{j\in (N_{2K})^c}\frac1{2}\left| [\overline y_j(\lambda,\omega), \underline x_{j+1}(\lambda,\omega))\cap[0,L] \right|\times\one_{\{\underline x_{j+1}(\lambda,\omega)\ge 2K,\ \overline y_j(\lambda,\omega)\le L-2K \}}\\
	&\ge \frac1{2}\int_{2K}^{L-2K}\one_{\{ \overline y_j(\lambda,\omega) \le x < \underline x_{j+1}(\lambda,\omega)\ \text{for some}\ j\in (N_{2K})^c \}}dx\\
	&\ge \frac1{2}\int_0^L\one_{\{ \overline y_j(\lambda,\omega) \le x < \underline x_{j+1}(\lambda,\omega)\ \text{for some}\ j\in (N_{2K})^c \}}dx - 2K
\end{align*}
for $L \ge 4K$. Therefore,
\begin{align*}
\lim_{L\to+\infty}\sum_{j\in (N_{2K})^c}\frac{C}{L}\int_0^L\one_{\{ \overline y_j(\lambda,\omega) \le x < \underline x_{j+1}(\lambda,\omega) \}}dx 
	&\le \lim_{L\to +\infty}\frac{2C}{L} \int_0^L\one_{\{ \tau_x\omega\in (A_{\lambda,K})^c \}}dx\\
	&= 2C\,\mathbb{P}((A_{\lambda,K})^c) < 2C\eta
\end{align*}
by \eqref{eq:hidden} and the ergodic theorem. We plug these bounds in \eqref{eq:firstterm}--\eqref{eq:secondterm} and deduce that
\[ 0 \le \theta_{1,3}(m + \beta) - \theta_1(m + \beta) \le (3C+1)\eta. \]
Since $\eta>0$ is arbitrary, we conclude that $\theta_1(m + \beta) = \theta_{1,3}(m + \beta)$.
\end{proof}

\begin{proof}[Proof of Theorem \ref{thm:ergthmapp}]
		We treat the four parts of the theorem in order.
		\begin{itemize}
			\item [(a)] When $\beta < M = m + \beta$, it follows readily from Theorem \ref{thm:velinim} and Definition \ref{def:lamb}(a) that
						\[ \Ham(\theta) = M\ \text{on}\ [\theta_1(M),\theta_3(M)]. \]
						
			\item [(b)] When $\beta < M < m + \beta$, it follows from Theorem \ref{thm:velinim} and Definition \ref{def:lamb}(b) that the graph of $\Ham$ has flat pieces at heights $M$ and $m + \beta$ iff
			\begin{equation}\label{eq:spasib}
				\theta_{1,3}(M) < \theta_3(M) \quad \text{and} \quad \theta_1(m + \beta) < \theta_{1,3}(m + \beta),
			\end{equation}
			respectively. The second inequality in \eqref{eq:spasib} is equivalent to $q_1 >0$ by Lemmas \ref{lem:oncelli} and \ref{lem:vans}, which is in turn equivalent to \eqref{eq:ezgiha} by ergodicity. This proves item (ii). Similarly, the proof of item (i) reduces to showing that the first inequality in \eqref{eq:spasib} is equivalent to
			\[ q_0 = \mathbb{P}\left(V(x,\omega) = 0\ \text{for some}\ x\in[0,1]\right) > 0. \]
			This is a routine adaptation of Lemmas \ref{lem:aynaer}--\ref{lem:vans}. We leave the details to the reader.
		
			\item [(c)] When $M \le \beta = m + \beta$, it follows readily from Theorem \ref{thm:velinim} and Definition \ref{def:lamb}(c) that
			\[ \Ham(\theta) = \beta\ \text{on}\ [\theta_1(\beta),\theta_4(\beta)]. \]
			
			\item [(d)] When $M \le \beta  < m + \beta$, it follows from Theorem \ref{thm:velinim} and Definition \ref{def:lamb}(d) that the graph of $\Ham$ has flat pieces at heights $\beta$ and $m + \beta$ iff
			\begin{equation}\label{eq:nazdar}
				\theta_{1,3}(\beta) < \theta_4(\beta) \quad \text{and} \quad \theta_1(m + \beta) < \theta_{1,3}(m + \beta),
			\end{equation}
			respectively. The second inequality in \eqref{eq:nazdar} is equivalent to \eqref{eq:ezgiha} precisely as in part (b). This proves item (ii). The proof of item (i) is easier since the first inequality in \eqref{eq:nazdar} always holds. Indeed,
			\[ \theta_{1,3}(\beta) \le \theta_{3}(\beta) < 0 < \theta_4(\beta).\qedhere \]
		\end{itemize}
\end{proof}

\section{Flat pieces of the graph of $\overline\Lambda$ at its intermediate values}\label{sec:interflat}

We break down the proof of Theorem \ref{thm:intermed} into several lemmas.

\begin{lemma}\label{lem:adqat}
	Assume that $\max\{\beta,M\} < m + \beta$. For any $\lambda\in(\max\{\beta,M\},m + \beta)$, if
	\begin{equation}\label{eq:aser}
		\mathbb{P}(\exists\,j\in\mathbb{Z}\ \text{s.t.}\ \underline x_0(\lambda,\omega) = \overline x_j(\lambda,\omega)) = 1,
	\end{equation}
	then
	$\mathbb{P}(\forall\,i\in\mathbb{Z}\ \exists\,j\in\mathbb{Z}\ \text{s.t.}\ \underline x_i(\lambda,\omega) = \overline x_j(\lambda,\omega)) = 1$.
	Similarly, if
	\[ \mathbb{P}(\exists\,k\in\mathbb{Z}\ \text{s.t.}\ \underline y_0(\lambda,\omega) = \overline y_k(\lambda,\omega)) = 1, \]
	then
	$\mathbb{P}(\forall\,i\in\mathbb{Z}\ \exists\,k\in\mathbb{Z}\ \text{s.t.}\ \underline y_i(\lambda,\omega) = \overline y_k(\lambda,\omega)) = 1$.
\end{lemma}

\begin{proof}
	Fix any $\lambda\in(\max\{\beta,M\},m + \beta)$. For every $z\in\mathbb{R}$ and $\omega\in\Omega_0$ (defined in \eqref{eq:sopra}), let
	\begin{equation}\label{eq:sapkadev}
	\begin{aligned}
		\hat x_1(\lambda,\omega,z) &= \sup\{x\le z:\,\lambda - \beta V(x,\omega) \ge M \} \le z,\\
		\hat x_2(\lambda,\omega,z) &= \sup\{x\le \hat x_1(\lambda,\omega,z):\,\lambda - \beta V(x,\omega) < m \} < \hat x_1(\lambda,\omega,z),\\
		\hat x_3(\lambda,\omega,z) &= \inf\{x\ge \hat x_2(\lambda,\omega,z):\,\lambda - \beta V(x,\omega) \ge M \} \in(\hat x_2(\lambda,\omega,z),\hat x_1(\lambda,\omega,z)],\\
		\hat x_4(\lambda,\omega,z) &= \inf\{x\ge \hat x_3(\lambda,\omega,z):\,\lambda - \beta V(x,\omega) < m \} > \hat x_1(\lambda,\omega,z),\\
		\hat x_5(\lambda,\omega,z) &= \inf\{x\ge \hat x_4(\lambda,\omega,z):\,\lambda - \beta V(x,\omega) \ge M \} > z\quad\text{and}\\
		\hat x_6(\lambda,\omega,z) &= \inf\{x\ge \hat x_5(\lambda,\omega,z):\,\lambda - \beta V(x,\omega) > M \} \ge \hat x_5(\lambda,\omega,z).
	\end{aligned}
	\end{equation}
	The inequalities and inclusions that are stated above regarding these quantities can be easily checked by considering the cases $\lambda - \beta V(z,\omega) < M$ and $\lambda - \beta V(z,\omega) \ge M$.
	Note also that they enjoy the following property which is due to stationarity:
	\begin{equation}\label{eq:statgibi}
		\hat x_\ell(\lambda,\omega,z) = \hat x_\ell(\lambda,\tau_z\omega,0) + z \ \ \text{for every}\ \ell\in\{1,2,3,4,5,6\}\ \text{and}\ z\in\mathbb{R}.
	\end{equation}
	
	It follows from the coupled recursion in \eqref{eq:noanchor} and the anchoring condition in \eqref{eq:cip1} that
	\[ \underline x_{-1}(\lambda,\omega) = \hat x_3(\lambda,\omega,0),\quad \overline y_{-1}(\lambda,\omega) = \hat x_4(\lambda,\omega,0)\quad\text{and}\quad\underline x_0(\lambda,\omega) = \hat x_5(\lambda,\omega,0). \]
	The last equality generalizes: For every $i\in\mathbb{Z}$ and $z\in\mathbb{R}$,
	\begin{equation}\label{eq:vjpin}
		B_{\lambda,i,z} := \{\omega\in\Omega_0:\,\underline x_i(\lambda,\omega) = \hat x_5(\lambda,\omega,z) \} = \{\omega\in\Omega_0:\,\underline x_{i-1}(\lambda,\omega)\le z < \underline x_i(\lambda,\omega) \}.
	\end{equation}
	It is clear from the second set representation of $B_{\lambda,i,z}$ that
	\begin{equation}\label{eq:o1visa}
		\mathbb{P}(\bigcup_{z\in\mathbb{Q}}B_{\lambda,i,z}) = 1\ \ \text{for every $i\in\mathbb{Z}$}.
	\end{equation}
		
	Recall from Lemma \ref{lem:interpm} that, for every $i\in\mathbb{Z}$ and $\omega\in\Omega_0$, there exists a $j\in\mathbb{Z}$ such that
	\[ \underline y_{j-1}(\lambda,\omega) < \underline x_i(\lambda,\omega) \le \overline x_j(\lambda,\omega). \]
	It follows easily that
	\begin{equation}\label{eq:dizman}
		B_{\lambda,i,z} \cap \{ \exists\,j\in\mathbb{Z}\ \text{s.t.}\ \underline x_i(\lambda,\omega) = \overline x_j(\lambda,\omega) \}  = B_{\lambda,i,z} \cap \{ \hat x_5(\lambda,\omega,z) = \hat x_6(\lambda,\omega,z) \}
	\end{equation}
	for every $i\in\mathbb{Z}$ and $z\in\mathbb{R}$.
	
	Finally, suppose \eqref{eq:aser} is true. For every $z\in\mathbb{R}$, define
	\[ \hat\Omega_z = \{\omega\in\Omega_0:\, \hat x_5(\lambda,\omega,z) = \hat x_6(\lambda,\omega,z) \} = \{\omega\in\Omega_0:\, \hat x_5(\lambda,\tau_z\omega,0) = \hat x_6(\lambda,\tau_z\omega,0) \}, \]
	where the second equality is due to \eqref{eq:statgibi}. Since $\mathbb{P}$ is invariant under $\tau_z$ and $\mathbb{P}(B_{\lambda,0,0}) = 1$ by \eqref{eq:cip1}, we use the set equality in \eqref{eq:dizman} to deduce that
	\begin{align*}
		\mathbb{P}(\hat\Omega_z) &= \mathbb{P}(\hat x_5(\lambda,\tau_z\omega,0) = \hat x_6(\lambda,\tau_z\omega,0)) = \mathbb{P}(\hat x_5(\lambda,\omega,0) = \hat x_6(\lambda,\omega,0))\\
		&= \mathbb{P}(B_{\lambda,0,0}\cap\{ \hat x_5(\lambda,\omega,0) = \hat x_6(\lambda,\omega,0) \}) = \mathbb{P}(B_{\lambda,0,0}\cap\{\exists\,j\in\mathbb{Z}\ \text{s.t.}\ \underline x_0(\lambda,\omega) = \overline x_j(\lambda,\omega)\})\\
		&= \mathbb{P}(\exists\,j\in\mathbb{Z}\ \text{s.t.}\ \underline x_0(\lambda,\omega) = \overline x_j(\lambda,\omega)) = 1.
	\end{align*}
	Therefore, 
	\begin{align*}
		 &\mathbb{P}\left(\forall\,i\in\mathbb{Z}\ \exists\,j\in\mathbb{Z}\ \text{s.t.}\ \underline x_i(\lambda,\omega) = \overline x_j(\lambda,\omega)\right)\\
		=\,&\mathbb{P}\left(\bigcap_{i\in\mathbb{Z}}\bigcup_{z\in\mathbb{Q}}B_{\lambda,i,z}\cap\{\exists\,j\in\mathbb{Z}\ \text{s.t.}\ \underline x_i(\lambda,\omega) = \overline x_j(\lambda,\omega)\}\right)\\
		=\,&\mathbb{P}\left(\bigcap_{i\in\mathbb{Z}}\bigcup_{z\in\mathbb{Q}}B_{\lambda,i,z}\cap\{ \hat x_5(\lambda,\omega,z) = \hat x_6(\lambda,\omega,z) \}\right)\\
		=\,&\mathbb{P}\left(\bigcap_{i\in\mathbb{Z}}\bigcup_{z\in\mathbb{Q}}B_{\lambda,i,z}\cap\hat\Omega_z\right) 
		= \mathbb{P}\left(\bigcap_{i\in\mathbb{Z}}\bigcup_{z\in\mathbb{Q}}B_{\lambda,i,z}\right) = 1
	\end{align*}
	by \eqref{eq:o1visa} and \eqref{eq:dizman}. The second implication in the statement of the lemma is proved similarly.
\end{proof}

\begin{lemma}\label{lem:mex1}
	Assume that $\max\{\beta,M\} < m + \beta$. For any $\lambda\in(\max\{\beta,M\},m + \beta)$, if
	\begin{equation}\label{eq:koktengri}
		\mathbb{P}(\forall\,i\in\mathbb{Z}\ \exists\,j,k\in\mathbb{Z}\ \text{s.t.}\ \underline x_i(\lambda,\omega) = \overline x_j(\lambda,\omega)\ \text{and}\ \underline y_i(\lambda,\omega) = \overline y_k(\lambda,\omega)) = 1,
	\end{equation}
	then
	\begin{equation}\label{eq:nikod}
		\mathbb{P}(\exists\,s\in\{0,1\}\ \text{s.t.}\ (\underline x_i(\lambda,\omega))_{i\in\mathbb{Z}} = (\overline x_{i+s}(\lambda,\omega))_{i\in\mathbb{Z}}\ \text{and}\ (\overline y_i(\lambda,\omega))_{i\in\mathbb{Z}} = (\underline y_{i+s}(\lambda,\omega))_{i\in\mathbb{Z}}) = 1.
	\end{equation}
	Moreover,
	\[ s = s(\omega) = \begin{cases} 0 &\text{if}\ \ \overline x_0(\lambda,\omega) > 0,\\ 1 &\text{if}\ \ \overline x_0(\lambda,\omega) = 0. \end{cases} \]
\end{lemma}

\begin{proof}
	Fix any $\lambda\in(\max\{\beta,M\},m + \beta)$. Recall from Lemma \ref{lem:interpm} that, for every $i\in\mathbb{Z}$, there exists a $j\in\mathbb{Z}$ such that
	\[ \underline x_j(\lambda,\omega) \le \overline x_i(\lambda,\omega) < \underline y_i(\lambda,\omega) \le \overline y_j(\lambda,\omega). \]
	If \eqref{eq:koktengri} holds, then, for $\mathbb{P}$-a.e.\ $\omega$, there exist $k,\ell\in\mathbb{Z}$ such that
	\[ \underline x_j(\lambda,\omega) = \overline x_k(\lambda,\omega) \le \overline x_i(\lambda,\omega) < \underline y_i(\lambda,\omega) = \overline y_\ell(\lambda,\omega) \le \overline y_j(\lambda,\omega). \]
	It follows that $k=i$. (Otherwise, for $\mathbb{P}$-a.e.\ $\omega$, there would exist an $m\in\mathbb{Z}$ such that
	\[ \underline x_j(\lambda,\omega) = \overline x_k(\lambda,\omega) < \underline y_k(\lambda,\omega) = \overline y_m(\lambda,\omega) < \overline x_i(\lambda,\omega) < \overline y_j(\lambda,\omega), \]
	which would be a contradiction.) Moreover, $\ell  = j$ since $\underline x_j(\lambda,\omega) < \overline y_\ell(\lambda,\omega) \le \overline y_j(\lambda,\omega)$.
	Therefore,
	\begin{equation}\label{eq:sabir}
		\underline x_j(\lambda,\omega) = \overline x_i(\lambda,\omega) < \underline y_i(\lambda,\omega) = \overline y_j(\lambda,\omega).
	\end{equation} 
	Similarly, for every $i\in\mathbb{Z}$ and $\mathbb{P}$-a.e.\ $\omega$, there exists a $k\in\mathbb{Z}$ such that
	\begin{equation}\label{eq:sukunet}
		\underline y_k(\lambda,\omega) = \overline y_i(\lambda,\omega) < \underline x_{i+1}(\lambda,\omega) = \overline x_{k+1}(\lambda,\omega).
	\end{equation}
	Performing two separate inductions on the index sets $\{ \ldots,-2,-1,0 \}$ and $\{ 0,1,2,\ldots \}$ (and using \eqref{eq:sabir} and \eqref{eq:sukunet} at each induction step), we easily see that \eqref{eq:nikod} is true, except that it remains to determine the possible values of $s = s(\omega)$ defined there. To this end, recall \eqref{eq:cip1} and \eqref{eq:cip2}. In particular, note that $\overline x_0(\lambda,\omega) \ge 0$. We have the following dichotomy:
	\begin{itemize}
		\item If $\overline x_0(\lambda,\omega) > 0$, then $\underline x_{-s}(\lambda,\omega) = \overline x_0(\lambda,\omega) > 0$ and $s\le 0$. Moreover, $0 < \underline x_0(\lambda,\omega) = \overline x_s(\lambda,\omega)$ and $s\ge0$. We deduce that $s=0$.
		\item If $\overline x_0(\lambda,\omega) = 0$, then $\underline x_{-s}(\lambda,\omega) = \overline x_0(\lambda,\omega) = 0$ and $s=1$.\qedhere
	\end{itemize} 
\end{proof}

\begin{lemma}\label{lem:mex2}
	Assume that $\max\{\beta,M\} < m + \beta$. For any $\lambda\in(\max\{\beta,M\},m + \beta)$, if \eqref{eq:nikod} holds, then $\underline \theta_{1,3}(\lambda) = \overline \theta_{1,3}(\lambda)$.
\end{lemma}

\begin{proof}
	For any $\lambda\in(\max\{\beta,M\},m + \beta)$, if \eqref{eq:nikod} holds, then the functions $\underline f_{1,3}^\lambda(\cdot,\omega)$ and $\overline f_{1,3}^\lambda(\cdot,\omega)$ (satisfying \eqref{eq:bizdik} and \eqref{eq:anbizdik}) are identically equal on $\mathbb{R}$ for $\mathbb{P}$-a.e.\ $\omega$. The equality of $\underline \theta_{1,3}(\lambda)$ and $\overline \theta_{1,3}(\lambda)$ now follows from their definitions in \eqref{eq:hurma1} and \eqref{eq:hurma2}.
\end{proof}

Next, we prove the converse of what we have shown in Lemmas \ref{lem:adqat}, \ref{lem:mex1} and \ref{lem:mex2} (all combined).

\begin{lemma}\label{eq:maliki}
	Assume that $\max\{\beta,M\} < m + \beta$. For any $\lambda\in(\max\{\beta,M\},m + \beta)$, if
	\begin{align}
	&\mathbb{P}(\exists\,j\in\mathbb{Z}\ \text{s.t.}\ \underline x_0(\lambda,\omega) = \overline x_j(\lambda,\omega)) < 1\quad\text{or}\label{eq:mketo}\\
	&\mathbb{P}(\exists\,k\in\mathbb{Z}\ \text{s.t.}\ \underline y_0(\lambda,\omega) = \overline y_k(\lambda,\omega)) < 1,\label{eq:cas}
	\end{align}
	equivalently
	\begin{align}
	&\mathbb{P}(\inf\{x\ge \underline x_0(\lambda,\omega):\,\lambda - \beta V(x,\omega) > M \} = \underline x_0(\lambda,\omega)) < 1\quad\text{or}\label{eq:mmemis}\\
	&\mathbb{P}(\inf\{x\ge \underline y_0(\lambda,\omega):\,\lambda - \beta V(x,\omega) < m \} = \underline y_0(\lambda,\omega)) < 1,\label{eq:cha}
	\end{align}
	then $\underline \theta_{1,3}(\lambda) < \overline \theta_{1,3}(\lambda)$.
\end{lemma}

\begin{proof}
	Fix any $\lambda\in(\max\{\beta,M\},m + \beta)$. Recall the proof of Lemma \ref{lem:adqat} and note that
	\begin{align*}
		\exists\,j\in\mathbb{Z}\ \text{s.t.}\ \underline x_0(\lambda,\omega) = \overline x_j(\lambda,\omega) &\iff \hat x_5(\lambda,\omega,0) = \hat x_6(\lambda,\omega,0)\\
		&\iff \inf\{x\ge \underline x_0(\lambda,\omega):\,\lambda - \beta V(x,\omega) > M \} = \underline x_0(\lambda,\omega)
	\end{align*}
	for every $\omega\in\Omega_0$. Hence, \eqref{eq:mketo} and \eqref{eq:mmemis} are equivalent. Similarly for \eqref{eq:cas} and \eqref{eq:cha}.
	
	Add the following definition to the list in \eqref{eq:sapkadev}:
	\[ \hat x_7(\lambda,\omega,z) = \inf\{x\ge \hat x_5(\lambda,\omega,z):\,\lambda - \beta V(x,\omega) < m \} > \hat x_5(\lambda,\omega,z). \]	
	If \eqref{eq:mketo} holds, then there exist $K,\delta,\eta > 0$ such that the event
	\[ S_{\lambda,K,\delta} = \{\hat x_5(\lambda,\omega,0) - \hat x_3(\lambda,\omega,0)\le K\ \text{and}\ \min\{\hat x_6(\lambda,\omega,0),\hat x_7(\lambda,\omega,0)\} - \hat x_5(\lambda,\omega,0)\ge\delta \} \]
	satisfies $\mathbb{P}(S_{\lambda,K,\delta}) \ge \eta$. By the ergodic theorem, for $\mathbb{P}$-a.e.\ $\omega$,
	\begin{equation}\label{eq:garse}
		\lim_{L\to +\infty}\frac1{L}\int_0^L\one_{ \{ \tau_z\omega\in S_{\lambda,K,\delta} \} }dz = \mathbb{P}(S_{\lambda,K,\delta}) \ge \eta.
	\end{equation}
	
	For every $z\in\mathbb{R}$, the sets $(B_{\lambda,i,z})_{i\in\mathbb{Z}}$ (defined in \eqref{eq:vjpin}) are disjoint and
	\[ \bigcup_{i\in\mathbb{Z}}B_{\lambda,i,z} = \{\underline{x}_i(\lambda,\omega)\in(-\infty,+\infty)\ \text{for every}\ i\in\mathbb{Z} \} \supset \Omega_0. \]
	Moreover, it follows from \eqref{eq:statgibi} that
	\begin{align*}
		&\quad\ B_{\lambda,i,z}\cap\{\tau_z\omega\in S_{\lambda,K,\delta} \}\\
		&= B_{\lambda,i,z}\cap\{\hat x_5(\lambda,\omega,z) - \hat x_3(\lambda,\omega,z)\le K\ \text{and}\ \min\{\hat x_6(\lambda,\omega,z),\hat x_7(\lambda,\omega,z)\} - \hat x_5(\lambda,\omega,z)\ge\delta \}\\
		&= B_{\lambda,i,z}\cap\{\underline x_i(\lambda,\omega) - \underline x_{i-1}(\lambda,\omega) \le K\ \text{and}\ \min\{\overline x_{j(i)}(\lambda,\omega),\overline y_i(\lambda,\omega)\} - \underline x_i(\lambda,\omega) \ge \delta \},
	\end{align*}
	where $j(i) = j(i,\omega) = \inf\{j\in\mathbb{Z}:\,\overline x_j(\lambda,\omega) \ge \underline x_i(\lambda,\omega) \}$. Let \[ I_{\lambda,K,\delta} = \{ i\in\mathbb{Z}:\,\underline x_i(\lambda,\omega) - \underline x_{i-1}(\lambda,\omega) \le K\ \text{and}\ \min\{\overline x_{j(i)}(\lambda,\omega),\overline y_i(\lambda,\omega)\} - \underline x_i(\lambda,\omega) \ge \delta \}. \]
	With this notation, for every $L > 0$,
	\begin{equation}\label{eq:lorke}
	\begin{aligned}
		\int_0^L\one_{ \{ \tau_z\omega\in S_{\lambda,K,\delta} \} }dz &= \sum_{i\in\mathbb{Z}}\int_0^L\one_{ B_{\lambda,i,z}\cap\{ \tau_z\omega\in S_{\lambda,K,\delta} \} }dz = \sum_{i\in I_{\lambda,K,\delta}}\int_0^L\one_{B_{\lambda,i,z}}dz\\
		&= \sum_{i\in I_{\lambda,K,\delta}}\int_0^L\one_{\{\underline x_{i-1}(\lambda,\omega)\le z < \underline x_i(\lambda,\omega) \}}dz\\
		&\le K\#\{i\in I_{\lambda,K,\delta}:\, 0 \le \underline x_i(\lambda,\omega)\ \ \text{and}\ \ \underline x_{i-1}(\lambda,\omega) \le L \}.
	\end{aligned}
	\end{equation}

	Recall from the proof of Lemma \ref{lem:basic} that $(\underline f_{1,3}^\lambda)'(x,\omega) \le (\overline f_{1,3}^\lambda)'(x,\omega)$ for every $\omega\in\Omega_0$ and a.e.\ $x\in\mathbb{R}$ (with respect to Lebesgue measure). Moreover, for every $i\in I_{\lambda,K,\delta}$, 
	\[ (\underline x_i(\lambda,\omega),\underline x_i(\lambda,\omega) + \delta) \subset (\underline x_i(\lambda,\omega),\overline y_i(\lambda,\omega)) \cap (\underline y_{j(i)-1}(\lambda,\omega),\overline x_{j(i)}(\lambda,\omega)). \]
	Therefore,
	\[ (\underline f_{1,3}^\lambda)'(x,\omega) = G_1^{-1}(\lambda - \beta V(x,\omega)) \le p_m < p_M \le G_3^{-1}(\lambda - \beta V(x,\omega)) = (\overline f_{1,3}^\lambda)'(x,\omega) \]
	whenever $x\in (\underline x_i(\lambda,\omega),\underline x_i(\lambda,\omega) + \delta)$. We conclude that
	\begin{align*}
	\overline\theta_{1,3}(\lambda) - \underline\theta_{1,3}(\lambda) &= \lim_{L\to +\infty}\frac1{L}\left( \int_0^{L+K+\delta} (\overline f_{1,3}^\lambda)'(x,\omega)dx - \int_0^{L+K+\delta}(\underline f_{1,3}^\lambda)'(x,\omega) dx \right)\\
	&\ge \liminf_{L\to +\infty}\frac{(p_M - p_m)\delta}{L} \#\{i\in I_{\lambda,K,\delta}:\, 0 \le \underline x_i(\lambda,\omega)\ \ \text{and}\ \ \underline x_{i-1}(\lambda,\omega) \le L \}\\
	&\ge \lim_{L\to +\infty}\frac{(p_M - p_m)\delta}{KL}\int_0^L\one_{ \{ \tau_z\omega\in S_{\lambda,K,\delta} \} }dz \ge \frac{(p_M - p_m)\delta\eta}{K} > 0
	\end{align*}
	by \eqref{eq:garse} and \eqref{eq:lorke}.
	
	The proof of $\underline \theta_{1,3}(\lambda) < \overline \theta_{1,3}(\lambda)$ under the conditions \eqref{eq:cas} and \eqref{eq:cha} (which are equivalent to each other) is similar.
\end{proof}

\begin{proof}[Proof of Theorem \ref{thm:intermed}]
	It follows from Theorem \ref{thm:gennet} and display \eqref{eq:nadad} at the end of its proof that, for every $\lambda\in(\max\{\beta,M\},m + \beta)$, the graph of the effective Hamiltonian $\Ham$ has a flat piece at height $\lambda$ if and only if
	\begin{equation}\label{eq:kilkuy}
	\underline \theta_{1,3}(\lambda) < \overline \theta_{1,3}(\lambda).
	\end{equation}
	Recall the events $U_\lambda$ and $D_\lambda$ which are defined in \eqref{eq:upcross} and \eqref{eq:downcross}, respectively. Lemmas \ref{lem:adqat}, \ref{lem:mex1}, \ref{lem:mex2} and \ref{eq:maliki} readily imply that $\mathbb{P}(U_\lambda \cap D_\lambda) < 1$ if and only if \eqref{eq:kilkuy} holds. This establishes the desired characterization.
\end{proof}

\section{Examples}\label{sec:examples}

Theorems \ref{thm:velinim}, \ref{thm:ergthmapp} and \ref{thm:intermed} enable us to identify the set
\[ \mathcal{L}(\Ham) = \{ \lambda\in\mathbb{R}:\,\text{the graph of $\Ham$ has a flat piece height $\lambda$} \}. \]
In particular, in the following cases, this set depends only on the parameters $\beta,m,M$.
\begin{itemize}
	\item When $\beta \le m + \beta < M$ (i.e., weak potential), \[ \mathcal{L}(\Ham) = \{\beta\}\cup\{m + \beta, M\}. \]
	\item When $\beta < M = m + \beta$ (i.e., the easy subcase of medium potential), \[ \mathcal{L}(\Ham) = \{\beta, M\}. \]
	\item When $M \le \beta = m + \beta$ (i.e., the easy subcase of strong potential), \[ \mathcal{L}(\Ham) = \{\beta\}. \]
\end{itemize}
In the remaining parameter regime, i.e., when $\max\{\beta,M\} < m + \beta$,
\[ \mathcal{L}(\Ham) \subset \{\beta\}\cup[\max\{\beta,M\},m + \beta] \] and it depends also on the law of $V(\cdot,\omega)$ under $\mathbb{P}$. Subsections \ref{sub:exper} and \ref{sub:exlin} provide more insight into this latter dependence by fully determining $\mathcal{L}(\Ham)$ in two basic classes of examples. 

\subsection{Periodic potentials}\label{sub:exper}

Take any $G:\mathbb{R}\to[0,+\infty)$ that satisfies \eqref{eq:coercive} and Condition \ref{cond:pm}. Assume that $\max\{\beta,M\} < m + \beta$ (i.e., medium or strong potential, excluding their easy subcases). 

Let $(\Omega,\mathcal{F},\mathbb{P})$ be the interval $[0,1)$ equipped with the Lebesgue measure. Denote a generic element of $[0,1)$ by $w$ (instead of the usual $\omega$). For every $x\in\mathbb{R}$, define $\tau_x:[0,1)\to[0,1)$ by
\[ \tau_xw = \kes(x+w) = (x + w) - \lfloor x + w \rfloor. \]
It is well known that $\mathbb{P}$ is stationary \& ergodic under $(\tau_x)_{x\in\mathbb{R}}$.

Given a $1$-periodic, continuous and surjective function $v_0:\mathbb{R} \to [0,1]$ that satisfies $v_0(0) = 1$, introduce $V:\mathbb{R}\times[0,1)\to[0,1]$ by setting
\[ V(x,w) = v_0(x + w) = v_0(\tau_xw).\]
It is clear that $V$ satisfies \eqref{eq:stationary}, \eqref{eq:bucon} and \eqref{eq:evren}. Therefore, Theorems \ref{thm:velinim}, \ref{thm:ergthmapp} and \ref{thm:intermed} are applicable.

In this periodic setting, since \eqref{eq:kimbegen} and \eqref{eq:ezgiha} always hold, we have the following results. 
\begin{itemize}
	\item When $\beta < M < m + \beta$ (i.e., medium potential, excluding its easy subcase), \[ \mathcal{L}(\Ham) \supset \{\beta, M, m + \beta\}. \]
	\item When $M \le \beta < m + \beta$ (i.e., strong potential, excluding its easy subcase), \[ \mathcal{L}(\Ham) \supset \{\beta, m + \beta\}. \]
\end{itemize}
It remains to identify all elements of the set
\[ \mathcal{L}(\Lambda) := \mathcal{L}(\Ham)\cap (\max\{\beta,M\},m + \beta). \]
We carry out this task under additional assumptions on the graph of the function $v_0$ which we impose for the sake of concreteness.

\begin{example}\label{ex:birmumdur}
	Assume that the function $v_0$ is strictly decreasing on $[0,z_1]$ and strictly increasing on $[z_1,1]$ for some $z_1\in(0,1)$. It follows from our surjectivity assumption that $v_0(z_1) = 0$.
	
	For every $\lambda\in(\max\{\beta,M\},m + \beta)$ and $w\in[0,1)$, the function $\lambda - \beta V(\cdot,w):\mathbb{R}\to[\lambda - \beta,\lambda]$ is a surjection and $0 < \lambda - \beta < m < M < \lambda$. It is easy to see that
	\[ 0 < \underline x_0(\lambda,0) = \overline x_0(\lambda,0) < z_1 < \underline y_0(\lambda,0) = \overline y_0(\lambda,0) < 1, \]
	\[ \kes(\underline x_0(\lambda,w)) = \kes(\underline x_0(\lambda,0) - w) \quad\text{and}\quad \kes(\underline y_0(\lambda,w)) = \kes(\underline y_0(\lambda,0) - w). \]
	Therefore, $\lambda - \beta V(\cdot,w)$ is locally invertible at $\underline x_0(\lambda,w)$ and $\underline y_0(\lambda,w)$.
	In particular, $\mathbb{P}(U_\lambda \cap D_\lambda) = 1$. By Theorem \ref{thm:intermed}, there is no flat piece at height $\lambda$. We conclude that
	\[ \mathcal{L}(\Lambda) = \mathcal{L}(\Ham)\cap (\max\{\beta,M\},m + \beta) = \emptyset. \]
\end{example}

\begin{example}
	Assume that the function $v_0$ is strictly decreasing on $[0,z_1]$, strictly increasing on $[z_1,z_2]$, strictly decreasing on $[z_2,z_3]$ and strictly increasing on $[z_3,1]$ for some $z_1,z_2,z_3\in(0,1)$ such that $z_1<z_2<z_3$. It follows from our surjectivity assumption that $\min\{v_0(z_1),v_0(z_3)\} = 0$.
	
	Let $\lambda_1 = M + \beta v_0(z_1)$. If $\lambda_1\in(\max\{\beta,M\},m + \beta)$, then $\underline x_0(\lambda_1,0) = z_1$. Moreover, if $w\in[0,z_1)$, then $\underline x_0(\lambda_1,w) = z_1 - w$. Therefore, $\lambda_1 - \beta V(\cdot,w)$ has a local maximum at $\underline x_0(\lambda_1,w)$. In particular, $\mathbb{P}(U_{\lambda_1}) < 1$. By Theorem \ref{thm:intermed}, there is a flat piece at height $\lambda_1$.
	
	Let $\lambda_2 = m + \beta v_0(z_2)$. If $\lambda_2\in(\max\{\beta,M\},m + \beta)$, then we consider three subcases.
	\begin{itemize}
		\item If $\lambda_2 = \lambda_1$, then it is already covered above. 
		\item If $\lambda_2 > \lambda_1$, then
		\[ 0 < \underline x_0(\lambda_2,0) = \overline x_0(\lambda_2,0) < z_1 < \underline y_0(\lambda_2,0) = z_2. \]
		Moreover, if $w\in[0,\overline x_0(\lambda_2,0)]$, then $\underline y_0(\lambda_2,w) = z_2 - w$. Therefore, $\lambda_2 - \beta V(\cdot,w)$ has a local minimum at $\underline y_0(\lambda_2,w)$. In particular, $\mathbb{P}(D_{\lambda_2}) < 1$. By Theorem \ref{thm:intermed}, there is a flat piece at height $\lambda_2$.
		\item If $\lambda_2 < \lambda_1$, then
		\[ z_2 < \underline x_0(\lambda_2,0) = \overline x_0(\lambda_2,0) < z_3 < \underline y_0(\lambda_2,0) = \overline y_0(\lambda_2,0) < 1.\]
		Moreover, for every $w\in[0,1)$,
	    \[ \kes(\underline x_0(\lambda_2,w)) = \kes(\underline x_0(\lambda_2,0) - w) \quad\text{and}\quad \kes(\underline y_0(\lambda_2,w)) = \kes(\underline y_0(\lambda_2,0) - w). \]
	    Therefore, $\lambda_2 - \beta V(\cdot,w)$ is locally invertible at $\underline x_0(\lambda_2,w)$ and $\underline y_0(\lambda_2,w)$.
	    In particular, $\mathbb{P}(U_{\lambda_2} \cap D_{\lambda_2}) = 1$. By Theorem \ref{thm:intermed}, there is no flat piece at height $\lambda_2$.
	\end{itemize}
	
	Let $\lambda_3 = M + \beta v_0(z_3)$. If $\lambda_3\in(\max\{\beta,M\},m + \beta)$, then $v_0(z_3) > v_0(z_1) = 0$ and $\lambda_3 > \lambda_1$. We consider three subcases.
	\begin{itemize}
		\item If $\lambda_3 = \lambda_2$, then it is already covered above.
		\item If $\lambda_3 < \lambda_2$, then
		\[ 0 < \underline x_0(\lambda_3,0) = \overline x_0(\lambda_3,0) < z_1 < \underline y_0(\lambda_3,0) = \overline y_0(\lambda_3,0) < z_2 < \underline x_1(\lambda_3,0) = z_3.\]
		Moreover, if $w\in[\underline x_0(\lambda_3,0),z_3)$, then $\underline x_0(\lambda_3,w) = \underline x_1(\lambda_3,0) - w = z_3 - w$. Therefore, $\lambda_3 - \beta V(\cdot,w)$ has a local maximum at $\underline x_0(\lambda_3,w)$. In particular, $\mathbb{P}(U_{\lambda_3}) < 1$. By Theorem \ref{thm:intermed}, there is a flat piece at height $\lambda_3$.
		\item If $\lambda_3 > \lambda_2$, then 
		\[ 0 < \underline x_0(\lambda_3,0) = \overline x_0(\lambda_3,0) < z_1 < z_3 < \underline y_0(\lambda_3,0) = \overline y_0(\lambda_3,0) < 1. \]
		Moreover, for every $w\in[0,1)$,
		\[ \kes(\underline x_0(\lambda_3,w)) = \kes(\underline x_0(\lambda_3,0) - w) \quad\text{and}\quad \kes(\underline y_0(\lambda_3,w)) = \kes(\underline y_0(\lambda_3,0) - w). \]
		Therefore, $\lambda_3 - \beta V(\cdot,w)$ is locally invertible at $\underline x_0(\lambda_3,w)$ and $\underline y_0(\lambda_3,w)$.
		In particular, $\mathbb{P}(U_{\lambda_3} \cap D_{\lambda_3}) = 1$. By Theorem \ref{thm:intermed}, there is no flat piece at height $\lambda_3$.
	\end{itemize}
	
	For every $\lambda\in(\max\{\beta,M\},m + \beta)$, if $\lambda - \beta V(\cdot,w)$ has a local maximum at $\underline x_i(\lambda,w)$ for some $i\in\mathbb{Z}$ and $w\in[0,1)$, then
	\[ \kes(\underline x_i(\lambda,w) + w) \in \{z_1,z_3\} \]
	and $\lambda = M + \beta V(\underline x_i(\lambda,w),w) = M + \beta v_0(\underline x_i(\lambda,w) + w)\in\{\lambda_1,\lambda_3\}$. Similarly, if $\lambda - \beta V(\cdot,w)$ has a local minimum at $\underline y_i(\lambda,w)$ for some $i\in\mathbb{Z}$ and $w\in[0,1)$, then $\lambda = m + \beta v_0(z_2) = \lambda_2$. In particular, if $\lambda\notin\{\lambda_1,\lambda_2,\lambda_3\}$, then $\mathbb{P}(U_\lambda \cap D_\lambda) = 1$. Therefore, by Theorem \ref{thm:intermed}, there is no flat piece at any height $\lambda\in(\max\{\beta,M\},m + \beta)\setminus\{\lambda_1,\lambda_2,\lambda_3\}$.
	
	Putting everything together, 
	we conclude that the set $\mathcal{L}(\Lambda) = \mathcal{L}(\Ham)\cap (\max\{\beta,M\},m + \beta)$ is given by
	\[ \mathcal{L}(\Lambda) = \begin{cases} \{\lambda_1\}&\text{if}\ \lambda_3 = M \le \max\{ \beta,M \} < \max\{\lambda_1,\lambda_2\} = \lambda_1 < m + \beta\\
											   &\text{or}\ \lambda_3 = M  \le \max\{ \beta,M \} < \lambda_1 < m + \beta = \lambda_2,\\
								  \{\lambda_2\}&\text{if}\ \lambda_1 = M \le \max\{ \beta,M \} < \lambda_2 = \min\{\lambda_2,\lambda_3\} < m + \beta\\
											   &\text{or}\ \max\{\lambda_1,\lambda_3\} \le \max\{ \beta,M \} < \lambda_2 < m + \beta,\\
								  \{\lambda_3\}&\text{if}\ \lambda_1 = M \le \max\{ \beta,M \} < \lambda_3 < m + \beta = \lambda_2,\\
						\{\lambda_1,\lambda_2\}&\text{if}\ \lambda_3 = M \le \max\{ \beta,M \} < \lambda_1 < \lambda_2 < m + \beta,\\
						\{\lambda_2,\lambda_3\}&\text{if}\ \lambda_1 = M \le \max\{ \beta,M \} < \lambda_3 < \lambda_2 < m + \beta,\\
								  \ \ \emptyset    &\text{otherwise,}
												 \end{cases} \]
	where $\lambda_1 = M + \beta v_0(z_1)$, $\lambda_2 = m + \beta v_0(z_2)$ and $\lambda_3 = M + \beta v_0(z_3)$.
\end{example}

\subsection{Piecewise linear random potentials}\label{sub:exlin}

Take any $G:\mathbb{R}\to[0,+\infty)$ that satisfies \eqref{eq:coercive} and Condition \ref{cond:pm}.
Assume that $\max\{\beta,M\} < m + \beta$ (i.e., medium or strong potential, excluding their easy subcases).

Let $(\Omega,\mathcal{F},\mathbb{P})$ be the product space $[0,1)\times[0,1]^{\mathbb{Z}}$ equipped with $\mathbb{P} = \ell\times\mathbb{Q}$, where $\ell$ is the Lebesgue measure on $[0,1)$ and $\mathbb{Q}$ is the law of a discrete-time stationary \& ergodic stochastic process taking values in $[0,1]$. Assume that the one-dimensional marginal distribution of $\mathbb{Q}$ is equal to some Borel probability measure $\mu$ on $[0,1]$ such that
\begin{equation}\label{eq:supp0}
	\{0,1\} \subset \text{supp}(\mu).
\end{equation}
Denote a generic element of $\Omega$ by $\omega = (w,(v_i)_{i\in\mathbb{Z}})$. For every $x\in\mathbb{R}$, define $\tau_x:\Omega\to\Omega$ by
\[ \tau_x(w,(v_i)_{i\in\mathbb{Z}}) = (\kes(x + w),(v_{\lfloor x + w \rfloor + i})_{i\in\mathbb{Z}}). \]
We leave it to the reader to check that $\mathbb{P}$ is stationary \& ergodic under $(\tau_x)_{x\in\mathbb{R}}$.

Introduce $V:\mathbb{R}\times\Omega\to[0,1]$ by setting
\[ V(x,\omega) = (1 - \kes(x + w))v_{\lfloor x + w \rfloor} + \kes(x + w)v_{\lfloor x + w \rfloor + 1}. \]
This is a linear interpolation. In particular,
\[ V(0,\omega) = (1 - w)v_0 + wv_1. \]
It is easy to see that $V$ satisfies \eqref{eq:stationary}, \eqref{eq:bucon} and \eqref{eq:evren}. Therefore, Theorems \ref{thm:velinim}, \ref{thm:ergthmapp} and \ref{thm:intermed} are applicable.

Let $\mathcal{A}(\mu)$ denote the set of atoms of $\mu$. Note that \eqref{eq:kimbegen} (resp.\ \eqref{eq:ezgiha}) holds if and only if $\mu$ has an atom at $0$ (resp.\ $1$).
Therefore, in this piecewise linear setting,  we have the following results. 
\begin{itemize}
	\item When $\beta < M < m + \beta$ (i.e., medium potential, excluding its easy subcase), $\beta\in\mathcal{L}(\Ham)$. Moreover,
				\begin{equation}\label{eq:bunref1}
					M\in\mathcal{L}(\Ham)\ \iff\ 0\in\mathcal{A}(\mu) \qquad\text{and}\qquad m + \beta \in\mathcal{L}(\Ham)\ \iff\ 1\in\mathcal{A}(\mu).
				\end{equation} 
	\item When $M \le \beta < m + \beta$ (i.e., strong potential, excluding its easy subcase), $\beta\in\mathcal{L}(\Ham)$. Moreover,
				\begin{equation}\label{eq:bunref2}
					m + \beta \in\mathcal{L}(\Ham)\ \iff\ 1\in\mathcal{A}(\mu).
				\end{equation}
\end{itemize}
It remains to identify all elements of the set
\[ \mathcal{L}(\Lambda) := \mathcal{L}(\Ham)\cap (\max\{\beta,M\},m + \beta). \]
We carry out this task under additional structural assumptions on the probability measure $\mathbb{Q}$.

\begin{example}\label{ex:iid}
	Assume that $\mathbb{Q} = \mu^{\mathbb{Z}}$, i.e., it is the law of a bi-infinite sequence of independent and identically distributed (i.i.d.) random variables with common distribution $\mu$.
	(See Remark \ref{rem:genel} for a generalization.)
	
	For every atom $a\in\mathcal{A}(\mu)$, if $\lambda_1(a) = M + \beta a \in (\max\{\beta,M\},m + \beta)$, then the event
	\begin{equation}\label{eq:olaybir}
	\begin{aligned}
		& \left\{ \omega = (w,(v_i)_{i\in\mathbb{Z}})\in\Omega:\, \lambda_1(a) - \beta v_0 < m,\ \lambda_1(a) - \beta v_1 = M,\ \lambda_1(a) - \beta v_2 < m \right\}\\
		=& \left\{ \omega = (w,(v_i)_{i\in\mathbb{Z}})\in\Omega:\, \lambda_1(a) - \beta V(-w,\omega) < m,\ \lambda_1(a) - \beta V(1-w,\omega) = M,\right.\\
		&\hspace{76mm}\left.\lambda_1(a) - \beta V(2-w,\omega) < m \right\}
	\end{aligned}
	\end{equation}
	has positive $\mathbb{P}$-probability by \eqref{eq:supp0}. On this event,
	\[ \overline y_{-1}(\lambda_1(a),\omega) < -w \le 0 < \underline x_0(\lambda_1(a),\omega) = 1 - w < \overline y_0(\lambda_1(a),\omega) < 2 - w \]
	and $\lambda_1(a) - \beta V(\cdot,\omega)$ has a local maximum at $\underline x_0(\lambda_1(a),\omega)$. We deduce that $\mathbb{P}(U_{\lambda_1(a)}) < 1$. Therefore, by Theorem \ref{thm:intermed}, there is a flat piece at height $\lambda_1(a)$.
	
	For every atom $a\in\mathcal{A}(\mu)$, if $\lambda_2(a) = m + \beta a \in (\max\{\beta,M\},m + \beta)$, then the event
	\begin{equation}\label{eq:olayiki}
	\begin{aligned}
		& \left\{ \omega = (w,(v_i)_{i\in\mathbb{Z}})\in\Omega:\, \lambda_2(a) - \beta v_0 > M,\ \lambda_2(a) - \beta v_1 = m,\ \lambda_2(a) - \beta v_2 > M \right\}\\
		=& \left\{ \omega = (w,(v_i)_{i\in\mathbb{Z}})\in\Omega:\, \lambda_2(a) - \beta V(-w,\omega) > M,\ \lambda_2(a) - \beta V(1-w,\omega) = m,\right.\\
		&\hspace{76mm}\left.\lambda_2(a) - \beta V(2-w,\omega) > M \right\}
	\end{aligned}
	\end{equation}
	has positive $\mathbb{P}$-probability by \eqref{eq:supp0}. On this event,
	\[ \overline x_{-1}(\lambda_2(a),\omega) < -w \le 0 < \underline y_{-1}(\lambda_2(a),\omega) = 1 - w < \overline x_0(\lambda_2(a),\omega) < 2 - w \]
	and $\lambda_2(a) - \beta V(\cdot,\omega)$ has a local minimum at $\underline y_{-1}(\lambda_2(a),\omega)$. We recall Lemma \ref{lem:adqat} and deduce that $\mathbb{P}(D_{\lambda_2(a)}) < 1$. Therefore, by Theorem \ref{thm:intermed}, there is a flat piece at height $\lambda_2(a)$.
	
	For every $\lambda \in (\max\{\beta,M\},m + \beta)\setminus\{\lambda_1(a),\lambda_2(a):\,a\in\mathcal{A}(\mu)\}$ and $\mathbb{P}$-a.e.\ $\omega$, the piecewise linear function $\lambda - \beta V(\cdot,\omega)$ is locally invertible at $\underline x_0(\lambda,\omega)$ and $\underline y_0(\lambda,\omega)$. In particular, $\mathbb{P}(U_\lambda \cap D_\lambda) = 1$. By Theorem \ref{thm:intermed}, there is no flat piece at height $\lambda$.
	
	Putting everything together, we conclude that the set $\mathcal{L}(\Lambda) = \mathcal{L}(\Ham)\cap (\max\{\beta,M\},m + \beta)$ is given by
    \begin{equation}\label{eq:iidfor}
   \begin{aligned}
   \mathcal{L}(\Lambda) &= \left(\{\lambda_1(a):\,a\in\mathcal{A}(\mu)\} \cup \{\lambda_2(a):\,a\in\mathcal{A}(\mu)\}\right) \cap (\max\{\beta,M\},m + \beta)\\
   &=\left( (M + \beta\mathcal{A}(\mu)) \cup (m + \beta\mathcal{A}(\mu)) \right) \cap (\max\{\beta,M\},m + \beta).
   \end{aligned}    
    \end{equation} 
\end{example}

\begin{example}\label{ex:Markov}
	Let $\mu_1$ and $\mu_2$ be two Borel probability measures on $[0,1]$ such that
	\begin{equation}\label{eq:supp12}
		0\in \text{supp}(\mu_1) \subset [0,c) \quad\text{and}\quad 1\in \text{supp}(\mu_2) \subset (c,1]
	\end{equation}
	for some $c\in(0,1)$. Assume that $\mathbb{Q}$ is the law of the discrete-time stationary Markov process whose transition kernel $\pi$ and invariant probability distribution $\mu$ are given by
	\[ \pi(v,\cdot) = \begin{cases} \mu_2(\cdot) &\text{if}\ v\in\text{supp}(\mu_1),\\ \mu_1(\cdot) &\text{if}\ v\in\text{supp}(\mu_2), \end{cases} \]
	and $\mu = \frac1{2}(\mu_1 + \mu_2)$. In words, $\mathbb{Q}$ is the law of two independent and interlaced bi-infinite sequences of i.i.d.\ random variables with common distributions $\mu_1$ and $\mu_2$, and which of these two distributions is used at the index $i=0$ is determined by a fair coin toss. (See Remark \ref{rem:genel} for a generalization.)
	
	For every atom $a\in\mathcal{A}(\mu_1)$, if $\lambda_1(a) = M + \beta a \in (\max\{\beta,M\},m + \beta)$, then the event in \eqref{eq:olaybir} has positive $\mathbb{P}$-probability by \eqref{eq:supp12}. Therefore, by the corresponding argument in Example \ref{ex:iid}, there is a flat piece at height $\lambda_1(a)$. 
		
	For every atom $a\in\mathcal{A}(\mu_2)$, if $\lambda_2(a) = m + \beta a \in (\max\{\beta,M\},m + \beta)$, then the event in \eqref{eq:olayiki} has positive $\mathbb{P}$-probability by \eqref{eq:supp12}. Therefore, by the corresponding argument in Example \ref{ex:iid}, there is a flat piece at height $\lambda_2(a)$.	
	
	For every $\lambda \in (\max\{\beta,M\},m + \beta)$ and $\omega = (w,(v_i)_{i\in\mathbb{Z}})\in\Omega$, it follows from \eqref{eq:supp12} that the set of local maxima of the piecewise linear function $\lambda - \beta V(\cdot,\omega)$ is equal to $\{ 2i - w:\,i\in\mathbb{Z} \}$ if $v_0 \in\text{supp}(\mu_1)$ and $\{ 2i + 1 - w:\,i\in\mathbb{Z} \}$ if $v_0 \in\text{supp}(\mu_2)$. Therefore, if the event
	\[ (U_\lambda)^c = \{ \omega\in\Omega:\,\text{$\lambda - \beta V(\cdot,\omega)$ has a local maximum at $\underline x_0(\lambda,\omega)$} \} \]
	has positive $\mathbb{P}$-probability, then $\lambda = \lambda_1(a)$ for some $a\in\mathcal{A}(\mu_1)$. Similarly, if $\mathbb{P}((D_\lambda)^c) > 0$, then $\lambda = \lambda_2(a)$ for some $a\in\mathcal{A}(\mu_2)$.
	
	Putting everything together, we conclude that the set $\mathcal{L}(\Lambda) = \mathcal{L}(\Ham)\cap (\max\{\beta,M\},m + \beta)$ is given by
	\begin{equation}\label{eq:Markovfor}
	\begin{aligned}
		\mathcal{L}(\Lambda) &= \left(\{\lambda_1(a):\,a\in\mathcal{A}(\mu_1)\} \cup \{\lambda_2(a):\,a\in\mathcal{A}(\mu_2)\}\right) \cap (\max\{\beta,M\},m + \beta)\\
		&=\left( (M + \beta\mathcal{A}(\mu_1)) \cup (m + \beta\mathcal{A}(\mu_2)) \right) \cap (\max\{\beta,M\},m + \beta).
	\end{aligned}
	\end{equation}
\end{example}

\begin{remark}\label{rem:desir}
	In Example \ref{ex:Markov}, it is clear from the description in \eqref{eq:Markovfor} that, for any finite or countable $S \subset (\max\{\beta,M\},m + \beta)$, there exist a $c\in(0,1)$ and measures $\mu_1,\mu_2$ satisfying \eqref{eq:supp12} such that $\mathcal{L}(\Lambda) = S$. (Note that there is not enough degree of freedom in the description in \eqref{eq:iidfor} to realize this in the setting of Example \ref{ex:iid}.) We recall the characterizations in \eqref{eq:bunref1}--\eqref{eq:bunref2} and conclude that we can make $\mathcal{L}(\Ham)$ equal to any desired finite or countable subset of $\{\beta\}\cup[\max\{\beta,M\},m + \beta]$ containing $\beta$.
\end{remark}

\begin{remark}\label{rem:genel}
	In Example \ref{ex:iid} (resp.\ Example \ref{ex:Markov}), it is clear from the arguments we gave that the description of the set $\mathcal{L}(\Lambda)$ in \eqref{eq:iidfor} (resp.\ \eqref{eq:Markovfor}) would not change if we replace $\mathbb{Q}$ with the law $\mathbb{Q}'$ of any discrete-time stationary \& ergodic stochastic process taking values in $[0,1]$ such that the finite-dimensional marginals of $\mathbb{Q}$ and $\mathbb{Q}'$ are mutually absolute continuous.
\end{remark}

\section*{Acknowledgments}

This work was conceived and shaped during an extensive and helpful electronic correspondence with A.\ Davini and E.\ Kosygina.
The author thanks them for generously sharing their intuition, answering many questions, making suggestions 
and giving valuable feedback on a preliminary version of the manuscript.

\bibliographystyle{abbrv}
\bibliography{1d_2well_inviscid_ref}

\end{document}